\documentclass[10pt,a4paper]{article}
\usepackage{latexsym}
\usepackage{amssymb}
\usepackage{bm}
\usepackage{amsmath}
\usepackage{mathtools}
\usepackage{graphicx}
\usepackage{epstopdf}
\usepackage[dvips]{epsfig}
\usepackage{wrapfig}
\usepackage{tikz}
\usepackage{amsthm}
\usepackage{float}
\usepackage{caption}
\usepackage{subcaption}
\usepackage{morefloats}
\usepackage{stmaryrd}
\usepackage{hyperref}
\usepackage{url}
\hypersetup{colorlinks=true,linkcolor=blue,citecolor=blue,urlcolor=blue}
\usepackage{epstopdf}
\usepackage{xcolor}
\usepackage[top=2.5cm, left=2.0cm, right=2.0cm, bottom=2.5cm]{geometry}
\newtheorem{theorem} {\bf Theorem}
\newtheorem{lemma} {\bf Lemma}

\usepackage{booktabs}
\usepackage{appendix}
\allowdisplaybreaks
\begin{document}
\title{\bf Nitsche's method for the stationary Boussinesq system under mixed and nonlinear boundary conditions}

\author{
  A. Bansal\thanks{Department of Mathematics, Indian Institute of Technology Roorkee, Roorkee 247667, India. Email: \texttt{a\_bansal@ma.iitr.ac.in}.}
  \and N. A. Barnafi\thanks{Instituto de Ingeniería Matemática y Computacional \& Facultad de Ciencias Biológicas, Pontificia Universidad Católica de Chile, Avenida Vicuña Mackenna 4860, Santiago, Chile; and Centro de Modelamiento Matemático (CNRS IRL2807), Santiago, Chile. Email: \texttt{nicolas.barnafi@uc.cl}.}
  \and G. Sperone\thanks{Facultad de Matemáticas, Pontificia Universidad Católica de Chile, Avenida Vicuña Mackenna 4860, 7820436 Santiago, Chile. Email: \texttt{gianmarco.sperone@uc.cl}.}
  \and D. N. Pandey\thanks{Department of Mathematics, Indian Institute of Technology Roorkee, Roorkee 247667, India. Email: \texttt{dwijpfma@iitr.ac.in}.}
}
\date{}
\maketitle

\begin{abstract}
In this paper we analyze Nitsche’s method for the stationary Boussinesq system with Navier's slip and a nonlinear boundary condition. Our analysis of the formulation establishes the robustness of a finite elements scheme in arbitrarily complex boundaries. The well-posedness of the discrete problem is established using fixed-point theorems under a standard smallness assumption on the data. We also provide optimal convergence rates for the approximation error. Furthermore, the efficiency and reliability of residual-based a posteriori error estimators are established.  We validate our theory through several numerical tests.
\end{abstract}

\vspace{0.5cm}


 \vspace{0.5cm}

\section{Introduction}
Non-isothermal flows describe fluid motion in which the temperature varies within the domain. Such flows arise naturally in many applications and are of particular importance in desalination processes, for example in sweeping gas membrane distillation \cite{Perfilov2018}. The mathematical description of non-isothermal flows is based on the Boussinesq approximation. In this framework, the model consists of the Navier–Stokes equations, which govern the fluid velocity and pressure, coupled with an advection–diffusion equation for the temperature. The coupling is introduced through a buoyancy force in the momentum equation and through convective heat transport in the temperature equation. Owing to the practical relevance and mathematical complexity of this coupled system, a wide range of numerical methods for its approximation have been developed in the literature  \cite{MR3649317, MR2947652, MR3913622, MR3918673, MR2111747, MR4489621}. Traditionally, it is assumed that the fluid adheres to the walls of its container, a condition known as the \textit{no-slip boundary condition}. However, the validity of this assumption has been widely debated ~\cite{gerard2015influence,goldstein1938modern}. Well-known boundary conditions in PDE and numerical analysis such as periodic, natural, and no-slip conditions, are insufficient for a variety of practical applications. In particular, physical phenomena such as inkjet printing~\cite{MR3179782}, pipe flow~\cite{berg2008two}, complex turbulent flows~\cite{MR1753115}, and slide coating~\cite{christodoulou1989fluid}, are well known for having non zero tangential velocity. This is captured through the \textit{Navier boundary conditions} \cite{berselli2010some}. 

Navier-slip boundary conditions constrain only the normal component of the velocity, which makes their numerical enforcement challenging in non-trivial geometries. Early approaches imposed the Navier-slip conditions weakly using Lagrange multipliers~\cite{MR1707832,verfurth1986finite,MR1124131}, which are consistent at the cost of 
increased computational cost due to additional variables and inf–sup stability requirements. Penalty methods~\cite{MR351118,MR853660,MR768643} offered simpler implementation by enforcing the Navier-slip condition through regularization terms, but at the cost of asymptotic consistency and with stronger regularity assumptions on the solution. Both approaches have been extended to curved domains~\cite{MR2571324,MR3171814,MR3311462,MR3117433}, where geometric mismatches may give rise to the Babu{\v{s}}ka-type paradox. To better address such variational inconsistencies, Nitsche’s method~\cite{MR341903} provides a consistent, primal framework for imposing boundary conditions and has been widely adapted for Navier-slip conditions. Various symmetric, non-symmetric, and penalty-free variants have been systematically reviewed in~\cite{chouly2024review}, with particular emphasis on stabilized formulations and curved boundaries. Following the Nitsche-type approach of Juntunen \& Stenberg~\cite{MR2501054}, a dedicated treatment of Navier-slip boundary conditions was introduced in~\cite{MR3759094}. The literature on the Navier–Stokes (and Stokes) equations with Navier-slip boundary conditions is extensive. More recently, Gjerde \& Scott~\cite{MR4379970} propose a symmetric Nitsche formulation for the Navier–Stokes equations with a geometry-consistent discretization that avoids a Babuška-type paradox while preserving optimal convergence. Extensions to both symmetric and non-symmetric variants for the Stokes equations are further studied in~\cite{MR4744101}. Bansal et al.~\cite{MR4812237, MR4994515} study the Navier–Stokes equations using inf–sup stable and equal-order stabilized finite element methods, and this work is further extended to general dynamic boundary conditions in~\cite{GazcaOrozco2025Nitsche}. In adaptive finite element methods, \emph{a posteriori} error estimators provide a means to quantify the local distribution of discretization errors. A reliable estimator not only bound the true error but also serves as a stopping criterion in the adaptive refinement process. Moreover, the efficiency of such estimators ensures that their convergence rate matches that of the actual error. Regarding the design and rigorous analysis of residual-based a posteriori error estimators for flow-transport 
 couplings, the literature is predominantly focused on the stationary case (see, e.g., \cite{allendes2020stabilized, wilfrid2019posteriori, dib2019posteriori, alvarez2018posteriori, agroum2018posteriori, alvarez2016posteriori, zhang2011posteriori,MR2111747, becker2002solution} and the references therein).

This work presents both \emph{a priori} and \emph{a posteriori} error analyses for the stationary Boussinesq equations with Navier slip and nonlinear boundary conditions, employing Nitsche’s method in conjunction with an inf-sup stable finite elements. 

\paragraph{Paper structure: } In Section~\ref{section1}, we present the stationary Boussinesq equations with Navier slip boundary conditions for the velocity and nonlinear boundary conditions for the temperature, derive the weak formulation of the problem, and discuss the solvability analysis of the continuous case. In Section~\ref{section2}, we introduce the Nitsche scheme, derive the variational formulation, and establish the well-posedness of the linearized discrete problem using the Banach–Ne\v{c}as–Babu\v{s}ka theorem. This well-posedness result is extended to the Navier–Stokes equations using Banach’s fixed point theorem in a standard way in Section~\ref{section3}. In Section~\ref{section4}, we derive a priori estimates and prove the optimal convergence of the method. Section~\ref{section5} is devoted to the construction and analysis of the efficiency and reliability of a residual-based a posteriori error estimator tailored to the stationary problem. In Section~\ref{section6}, we present numerical tests to support our theoretical results.

\subsection*{Notations}
Throughout this manuscript, we utilize the classical Sobolev spaces $L^2(\Omega)$ and $H^1(\Omega)$, equipped with their respective norms $\|\cdot\|_{L^2(\Omega)}$ and $\|\cdot\|_{H^1(\Omega)}$. The $L^2$-inner product is denoted by $(\cdot, \cdot)$, and, for any arbitrary Hilbert space $H$, the duality pairing with its dual space $H'$ is represented by $\langle\cdot, \cdot\rangle_{H', H}$. We follow the convention of denoting scalars, vectors, and tensors by $a$, $\boldsymbol{a}$, and $\mathbb{A}$, respectively. 

For the sake of simplicity, throughout the analysis, we use $C$ to denote a generic positive constant that is independent of the mesh size $h$ but may depend on the model parameters. We also adopt a slight abuse of notation by using $\delta_i, \gamma_j > 0$, with $i,j \in \mathbb{N}$, to represent arbitrary positive constants that may take different values at different occurrences, typically arising from the application of Young's inequality. Moreover, whenever an inequality involves positive constants that are independent of the mesh size but may depend on model parameters, we use the symbols $\lesssim$ or $\gtrsim$ to omit explicit constants. The assumption of homogeneity in the boundary conditions is made to simplify the subsequent analysis, as lifting operators have already been established~\cite{MR3974685}. Non-homogeneous boundary conditions are nevertheless used in the numerical tests in Section~\ref{section6}.

\section{Model Problem}\label{section1}
Let $\Omega \subset \mathbb{R}^d$, $d \in \{2,3\}$, be a bounded domain having a Lipschitz boundary $\partial \Omega$. We assume that $\partial \Omega$ can be decomposed as $\partial \Omega = \Gamma_{I} \cup \Gamma_{W} \cup \Gamma_{O}$, where the intersection between each of the boundary components are sets of null 3D-measure. Roughly speaking, $\Gamma_{I}$ and $\Gamma_{O}$ represent the inlet and outlet sections, respectively, of the container $\Omega$, while $\Gamma_{W}$ includes all the \textit{physical walls} of $\partial \Omega$ (namely, lateral walls of $\Omega$ and the surface of obstacles strictly contained inside $\Omega$).

In the present article we analyze the steady motion of a viscous, incompressible and heat-conductive fluid across the domain $\Omega$. Following \cite{acevedo2016boussinesq,acevedo2019p,arndt2020existence,ceretani2019boussinesq,kravcmarnon}, such motion will be described by the velocity field of the fluid $\boldsymbol{u} : \Omega \longrightarrow \mathbb{R}^d$, its scalar pressure $p : \Omega \longrightarrow \mathbb{R}$ and temperature $\theta : \Omega \longrightarrow \mathbb{R}$, in the presence of an external body force $\boldsymbol{f} : \Omega \longrightarrow \mathbb{R}^d$ and external sources of heat $g : \Omega \longrightarrow \mathbb{R}$; these functions are assumed to satisfy the stationary Boussinesq equations in $\Omega$, that is, the following coupled system of partial differential equations:
\begin{equation}\label{nsstokes0}
\left\{
\begin{aligned}
& -\nu\Delta \boldsymbol{u}+(\boldsymbol{u}\cdot\nabla)\boldsymbol{u}+\nabla p= \alpha \theta \boldsymbol{f} \, , \quad  \nabla\cdot \boldsymbol{u}=0 \, , \quad -\kappa \Delta \theta +  \boldsymbol{u} \cdot \nabla \theta =g \ \ \mbox{ in } \ \ \Omega \,, \\[3pt]
& \boldsymbol{u}= \boldsymbol{u}_{\star} \ \ \ \mbox{ on } \ \ \Gamma_{I} \,, \qquad \theta = \theta_{\star} \ \ \mbox{ on } \ \ \Gamma_{I} \, , \\[3pt]
& \boldsymbol{u} \cdot \boldsymbol{n} =0 \,, \quad \left[ \mathbb{T}(\boldsymbol{u},p) \boldsymbol{n} \right]_{\tau} + \gamma \boldsymbol{u} = 0 \,, \quad \kappa \dfrac{\partial \theta}{\partial \boldsymbol{n}} + \beta \theta = 0  \ \ \mbox{ on } \ \ \Gamma_{W} \,,   \\[3pt]
& \mathbb{T}(\boldsymbol{u},p) \boldsymbol{n} = \boldsymbol{0} \,, \quad \kappa \dfrac{\partial \theta}{\partial \boldsymbol{n}} = (\boldsymbol{u} \cdot \boldsymbol{n}) \, \theta \, \psi(\boldsymbol{u} \cdot \boldsymbol{n}) \ \ \mbox{ on } \ \ \Gamma_{O} \,.
\end{aligned}
\right.
\end{equation}
\noindent
In \eqref{nsstokes0}$_1$, $\nu>0$ is the (constant) kinematic viscosity of the fluid, while $\alpha > 0$ and $\kappa > 0$ denote, respectively, the coefficients of thermal expansion and conductivity of the fluid. In \eqref{nsstokes0}$_2$, the functions $\boldsymbol{u}_{\star} : \Gamma_{I} \longrightarrow \mathbb{R}^d$ and $\theta_{\star} : \Gamma_{I} \longrightarrow \mathbb{R}$ describe, respectively, the inlet velocity of the fluid and its temperature on $\Gamma_{I}$. In \eqref{nsstokes0}$_{3}$, $\beta > 0$ denotes the diffusivity constant and $\mathbb{T}(\boldsymbol{u},p)$ denotes the stress tensor of the fluid, defined as
$$
\mathbb{T}(\boldsymbol{u},p) = -p\mathbb{I}_{d} + 2 \nu \varepsilon(\boldsymbol{u})  \quad \text{in} \quad \Omega \qquad \text{with} \quad \varepsilon(\boldsymbol{u}) \doteq \dfrac{1}{2}[\nabla \boldsymbol{u} + (\nabla \boldsymbol{u})^{\intercal}] \, ,
$$
where $\mathbb{I}_{3}$ is the $d \times d$-identity matrix. Then, $\left[ \mathbb{T}(\boldsymbol{u},p) \boldsymbol{n} \right]_{\tau}$ is the tangential component of the force with which the fluid acts on the boundary of $\Omega$, so that
\begin{equation} \label{tangential}
\left[ \mathbb{T}(\boldsymbol{u},p) \boldsymbol{n} \right]_{\tau} = \mathbb{T}(\boldsymbol{u},p) \boldsymbol{n} - \left[\mathbb{T}(\boldsymbol{u},p) \boldsymbol{n} \cdot \boldsymbol{n} \right] \boldsymbol{n} = \sum_{i=1}^{d-1}
\left( \mathbb{T}(\boldsymbol{u},p)\boldsymbol{n} \cdot \boldsymbol{\tau}^i \right)\boldsymbol{\tau}^i \, ,
\end{equation}
while $\gamma \geq 0$ is the (constant) coefficient of friction between the fluid and $\partial \Omega$ and $\boldsymbol{\tau}^i$, $i \in \{1, \ldots , d-1 \}$, represent the unit tangential vectors to the boundary $\partial \Omega$. The first two identities in \eqref{nsstokes0}$_{3}$ impose the \textit{Navier-slip boundary conditions} (firstly introduced C. L. Navier in 1823, see \cite{navier1823memoire} and the survey article by Berselli \cite{berselli2010some}), expressing both the impermeability of $\Gamma_{W}$ and the requirement that the tangential component of the force with which the fluid acts on $\Gamma_{W}$ be proportional to the tangential velocity. Then, the third equation in \eqref{nsstokes0}$_{3}$ dictates that $\Gamma_{W}$ is a thermally convective part of $\partial \Omega$, while the first equality in \eqref{nsstokes0}$_{4}$ prescribes a \textit{do-nothing} or \textit{constant traction} boundary condition for the velocity field on the outlet $\Gamma_{O}$. Finally, for a given Lipschitz-continuous function $\psi : \mathbb{R} \longrightarrow \mathbb{R}$, the second equality in \eqref{nsstokes0}$_{4}$, denotes an appropriate switching function, corresponding to the artificial boundary condition introduced in~\cite{ceretani2019boussinesq}. This condition is designed to model the effective thermal behavior at the outlet when the computational domain represents only a truncated portion of a larger physical system. In particular, it mimics the balance between convective heat transport and conductive flux across $\Gamma_{O}$, ensuring that heat is transported outward with the flow while preventing non-physical backflow effects. Such a formulation provides a mathematically consistent and energetically stable closure that captures the dominant thermal-conduction mechanisms occurring at the open boundary without explicitly resolving the exterior domain.  Throughout the paper, it is assumed that $\boldsymbol{u} _{\star} \in \boldsymbol{H}^{1/2}(\Gamma_{I})$,  $\theta _{\star} \in H^{1/2}(\Gamma_{I})$, $\boldsymbol{f} \in \boldsymbol{L}^2(\Omega)$, $g \in L^{2}(\Omega)$ and $\psi \in L^{\infty}(\mathbb{R})$ are given data.

Define the sets of functions:
\begin{align*}
\mathrm{V} & \coloneqq \{ \boldsymbol{u} \in \boldsymbol{H}^1(\Omega): \boldsymbol{u} \cdot \boldsymbol{n} = 0 \text{ on } \Gamma_W, \ \ \boldsymbol{u} = \boldsymbol{u} _{\star} \text{ on } \Gamma_I \},  \\
     \mathrm{V}_{0} & \coloneqq \{ \boldsymbol{u} \in \boldsymbol{H}^1(\Omega): \boldsymbol{u} \cdot \boldsymbol{n} = 0 \text{ on } \Gamma_W, \ \ \boldsymbol{u} = \boldsymbol{0} \text{ on } \Gamma_I \},  \\
    \mathrm{M} & \coloneqq \{ \theta \in H^1(\Omega): \theta = \theta_{\star} \text{ on } \Gamma_I \}, \\
    \mathrm{M}_{0} & \coloneqq \{ \theta \in H^1(\Omega): \theta = 0 \text{ on } \Gamma_I \}, \\
    \Pi & \coloneqq L^2(\Omega).
\end{align*}
The standard weak formulation of \eqref{nsstokes0} reads: Find $(\boldsymbol{u},p,\theta) \in \mathrm{V} \times \Pi \times \mathrm{M}$, such that
\begin{align}\label{bb}
	\mathcal{A}\left[\left(\boldsymbol{u}, p, \theta \right) ;\left(\boldsymbol {v}, q, \varphi \right)\right]=\mathcal{F}(\boldsymbol{v}, q, \varphi )  \qquad \forall  \left(\boldsymbol{v}, q, \varphi  \right) \in \mathrm{V}_{0} \times \Pi \times \mathrm{M}_{0},
	\end{align}
	where
 \begin{align*}
\mathcal{A}\!\left[\left(\boldsymbol{u}, p, \theta \right);
\left(\boldsymbol{v}, q, \varphi \right)\right]
&\coloneqq
2\nu \bigl(\varepsilon(\boldsymbol{u}), \varepsilon(\boldsymbol{v})\bigr)
+ (\boldsymbol{u} \cdot \nabla \boldsymbol{u}, \boldsymbol{v})- (p, \nabla \cdot \boldsymbol{v})- (q, \nabla \cdot \boldsymbol{u}) - \alpha (\theta \boldsymbol{f}, \boldsymbol{v}) + \kappa (\nabla \theta, \nabla \varphi) \\[3pt] & \quad + (\boldsymbol{u} \cdot \nabla \theta, \varphi) + \gamma \int\limits_{\Gamma_W} 
\sum_{i=1}^{d-1}\left(\boldsymbol{\tau}^i \cdot \boldsymbol{v}\right)\left(\boldsymbol{\tau}^i \cdot \boldsymbol{u}\right) \, ds + \beta \int\limits_{\Gamma_W} \theta  \, \varphi \, ds  -  \int\limits_{\Gamma_O} (\boldsymbol{u} \cdot \boldsymbol{n}) \, \theta \, \psi(\boldsymbol{u} \cdot \boldsymbol{n}) \,
\varphi \, ds, \\
\mathcal{F}(\boldsymbol{v}, q, \varphi)
&\coloneqq
\int_{\Omega} g \, \varphi \, dx .
\end{align*}
Given that the primary objective of this study is to analyze the discrete formulation of \eqref{bb} using Nitsche’s method, we do not present the continuous analysis, which can be carried out similarly to \cite{MR4812237} using fixed-point arguments assuming that the nonlinearity $\psi$ is Lipschitz-continuous and bounded.

  \section{Discrete Problem}\label{section2}
  \label{sec3}
  This section studies the solvability and convergence analysis of Nitsche's scheme for the problem. We assume that the polygonal computational domain ${\Omega}$ is discretized using a collection of regular partitions, denoted as $\{\mathcal{K}_h\}_{h>0}$, where $\Omega \subset \mathbb{R}^d$ is divided into simplices $K$ (triangles in 2D or tetrahedra in 3D) with a diameter $h_{{K}}$. The characteristic length of the finite element mesh $\mathcal{K}_h$ is denoted as $h \coloneqq \max _{{K} \in \mathcal{K}_h} h_{{K}}$. For a given triangulation $\mathcal{K}_h$, we define $\mathcal{E}_h$ as the set of all faces in $\mathcal{K}_h$, with the following partitioning
  		$$
  		\mathcal{E}_h \coloneqq \mathcal{E}_{\Omega} \cup \mathcal{E}_I \cup \mathcal{E}_O \cup \mathcal{E}_{W}, \text{ and } \mathcal{E}_{\Gamma} = \mathcal{E}_I \cup \mathcal{E}_O \cup \mathcal{E}_{W}, 
  		$$
  		where $\mathcal{E}_{\Omega}$ represents the faces lying in the interior of $\Omega$, $\mathcal{E}_{\Gamma}$ is the union of all boundary edges, and $\mathcal{E}_{W}, \mathcal{E}_{I}, \mathcal{E}_{O}$ represent the faces lying on their corresponding boundaries. Additionally, $h_e$ denotes the $(d-1)$ dimensional diameter of a face. Here \emph{faces} loosely refers to the geometrical entities of co-dimension 1. Finally, we introduce the finite element spaces. 
  		$$
  		\begin{aligned}
  		& \mathrm{V}_h \coloneqq \left\{\boldsymbol{v}_h \in \boldsymbol{C}(\overline{\Omega}) : \left.\boldsymbol{v}_h\right|_{K} \in \mathbb {P}_k({K}) \quad \forall {K} \in \mathcal{K}_h\right\}, \\
  		& \mathrm{\Pi}_h \coloneqq \left\{q_h \in \mathrm{C}(\overline{\Omega}) :\left.q_h\right|_{{K}} \in \mathbb {P}_{k-1}({K}) \quad \forall {K} \in \mathcal{K}_h\right\} \cap L^2(\Omega), 
        \\
        & \mathrm{M}_{h} \coloneqq \left\{\theta_h \in {C}(\overline{\Omega}) : \left.\theta_h\right|_{K} \in \mathbb {P}_k({K}) \quad \forall {K} \in \mathcal{K}_h\right\},
  		\end{aligned}
  		$$
  		where $\mathbb{P}_k(K)$ is the space of polynomials of degree $k$ defined on $K$. This is the classical Taylor–Hood finite elements for the velocity and pressure, which are inf–sup stable. The degree $k$ for the temperature has been chosen in order to have a consistent order of convergence among all fields.

  		\subsection{Nitsche's Method}
 The main objective of Nitsche's method is to impose boundary conditions weakly so that they hold only asymptotically. In this work, this approach is applied to the Navier-slip boundary condition on $\Gamma_W$ and the Dirichlet boundary conditions on $\Gamma_I$. As a result, the weak formulation with the symmetric variant of the Nitsche method can be expressed as follows:
  		Find $\left(\boldsymbol{u}_h, p_h, \theta_h \right) \in \mathrm{V}_h \times \Pi_h \times \mathrm{M}_h$, such that
  		\begin{align}\label{41}
  		\mathcal{A}_h\left[\left(\boldsymbol{u}_h, p_h, \theta_h \right) ;\left(\boldsymbol{v}_h, q_h,\varphi_h\right)\right]=\mathcal{F}(\boldsymbol{v}_h, q_h,\varphi_h)  \quad \forall  \left(\boldsymbol{v}_h, q_h,\varphi_h \right) \in \mathrm{V}_h \times \Pi_h \times \mathrm{M}_{h},
  		\end{align}
  where the forms are defined as 
  		$$
  		\begin{aligned}
  		\mathcal{A}_h\left[(\boldsymbol{u}_h, p_h, \theta_h);(\boldsymbol{v}_h, q_h, \varphi_h)\right]  &\coloneqq \mathcal{A}\!\left[\left(\boldsymbol{u}_h, p_h , \theta_h \right);
\left(\boldsymbol{v}_h, q_h, \varphi_h \right)\right]  +
   		\sum_{E\in \mathcal{E}_W}\bigg(-\int_{E} \boldsymbol{n}^t(2 \nu \varepsilon(\boldsymbol{u}_h)-p_h I) \boldsymbol{n}(\boldsymbol{n} \cdot \boldsymbol{v}_h) d s \\& -\int_{E} \boldsymbol{n}^t(2 \nu \varepsilon(\boldsymbol{v}_h)-q_h I) \boldsymbol{n}(\boldsymbol{n} \cdot \boldsymbol{u}_h) d s   +
  		\gamma_N \int_{E} {h_e}^{-1}(\boldsymbol{u}_h \cdot \boldsymbol{n})(\boldsymbol{v}_h \cdot \boldsymbol{n}) d s \bigg) \\&   - 
   		\sum_{E\in \mathcal{E}_O} 
   		\sum_{E\in \mathcal{E}_I}\bigg(-\int_{E} 2 \nu \varepsilon(\boldsymbol{u}_h) \boldsymbol{n} \cdot  \boldsymbol{v}_h d s   -\int_{E} 2 \nu \varepsilon(\boldsymbol{v}_h) \boldsymbol{n} \cdot \boldsymbol{u}_h d s 
		+ \int_{E} p_h (\boldsymbol{n} \cdot \boldsymbol{v}_h) d s \\& + \int_{E} q_h (\boldsymbol{n} \cdot \boldsymbol{u}_h) d s 
       -\int_{E} \kappa   (\nabla \theta_h \cdot \boldsymbol{n}) \varphi_h d s -\int_{E} \kappa   (\nabla \varphi_h \cdot \boldsymbol{n}) \theta_h d s  \\ &  + \gamma_N \int_{E} {h_e}^{-1}(\boldsymbol{u}_h \cdot \boldsymbol{v}_h) d s    +\gamma_N \int_{E} {h_e}^{-1}(\theta_h \varphi_h)d s \bigg), \\
        \mathcal{F}_h(\boldsymbol{v}_h, q_h, \varphi_h) &\coloneqq \int_{\Omega} g \varphi~ dx + \sum_{E\in \mathcal{E}_I}\bigg(-\int_{E} 2 \nu \varepsilon(\boldsymbol{v}_h) \cdot \boldsymbol{u}_{\star} d s + \int_{E} q_h (\boldsymbol{n} \cdot \boldsymbol{u}_{\star}) d s  \\&   + \gamma_N \int_{E} {h_e}^{-1}(\boldsymbol{u}_{\star} \cdot \boldsymbol{v}_h) d -\int_{E} \kappa   (\nabla \varphi_h \cdot \boldsymbol{n}) \theta_{\star} d s +\gamma_N \int_{E} {h_e}^{-1}(\theta_{\star} \varphi_h)d s \bigg),
  		\end{aligned}
  		$$
  	and $\gamma_N>0$ is a positive constant that needs to be chosen sufficiently large, as proved in Lemma~\ref{S} later. We can rewrite \eqref{41} as:
  	    Find $(\boldsymbol{u}_h,p_h,\theta_h) \in \mathrm{V}_h \times \mathrm{\Pi}_h \times \mathrm{M}_h$, such that
  	    \begin{equation}\label{E}
  	    \begin{aligned}
  	    \begin{array}{rlrl}
  	    {A}_S^h(\boldsymbol{u}_h, \boldsymbol{v}_h) + {O}_S^h(\boldsymbol{u}_h ; \boldsymbol{u}_h, \boldsymbol{v}_h) + {B}^h(\boldsymbol{v}_h,p_h) - D(\theta_h,\boldsymbol{v}_h)& = F_{v}(\boldsymbol{v}_h) & \forall \boldsymbol{v}_h \in \mathrm{V}_h, \\
  	    {B}^h(\boldsymbol{u}_h, q_h) & = F_q(q_h) & \forall q_h \in \mathrm{\Pi}_h, \\ 
       {A}^h_T(\theta_h,\varphi_h) + O_T(\boldsymbol{u}_h; \theta_h, \varphi_h) & = F_{\theta}(\varphi_h) & \forall \varphi_h \in \mathrm{M}_{h},
  	    \end{array}
  	    \end{aligned}
  	    \end{equation} 
 where 
  	    \begin{align}\label{F}
  	    A_S^h(\boldsymbol{u}_h, \boldsymbol{v}_h) & \coloneqq  \sum_{K\in \mathcal{K}_{h}} 2 \nu (\varepsilon(\boldsymbol{u}_h), \varepsilon(\boldsymbol{v}_h))_{K} + \sum_{E\in \mathcal{E}_W}\bigg(-\int_{E} 2 \nu \varepsilon(\boldsymbol{u}_h) \boldsymbol{n} \cdot \boldsymbol{n} (\boldsymbol{n} \cdot \boldsymbol{v}_h) d s-\int_{E} 2 \nu \varepsilon(\boldsymbol{v}_h) \boldsymbol{n} \cdot \boldsymbol{n} (\boldsymbol{n} \cdot \boldsymbol{u}_h) d s  \nonumber\\ & + \int_{E} \gamma \sum_{i=1}^{d-1}\left(\boldsymbol{\tau}^i \cdot \boldsymbol{v}_h \right)\left(\boldsymbol{\tau}^i \cdot \boldsymbol{u}_h \right) d s 
  		+\gamma_N \int_{E} {h_e}^{-1}(\boldsymbol{u}_h \cdot \boldsymbol{n})(\boldsymbol{v}_h \cdot \boldsymbol{n}) d s \bigg) - 
   		\sum_{E\in \mathcal{E}_I}\bigg(\int_{E} 2 \nu \varepsilon(\boldsymbol{u}_h) \boldsymbol{n} \cdot  \boldsymbol{v}_h d s
       \nonumber  \\ & +\int_{E} 2 \nu \varepsilon(\boldsymbol{v}_h) \boldsymbol{n} \cdot \boldsymbol{u}_h d s - \gamma_N \int_{E} {h_e}^{-1}(\boldsymbol{u}_h \cdot \boldsymbol{v}_h) d s  \bigg)  \nonumber \\
  	    B^h(\boldsymbol{u}_h, q_h) &\coloneqq \sum_{K\in \mathcal{K}_{h}} -(q_h, \nabla \cdot \boldsymbol{u}_h)_{K} + \sum_{E\in \mathcal{E}_W \cup \mathcal{E}_I} \int_{E}q_h(\boldsymbol{n} \cdot \boldsymbol{u}_h) d s,  \nonumber \\
        {O}_S^h(\boldsymbol{u}_h ; \boldsymbol{u}_h, \boldsymbol{v}_h) & \coloneqq \sum_{K\in \mathcal{K}_{h}} (\boldsymbol{u}_h \cdot \nabla \boldsymbol{u}_h, \boldsymbol{v}_h)_{K} ,  \nonumber \\
        D(\theta_h,\boldsymbol{v}_h) & \coloneqq  \sum_{K\in \mathcal{K}_{h}} (\alpha \theta_h \boldsymbol{f}, \boldsymbol{v}_h)_K ,  \nonumber \\ 
        {A}^h_T(\theta_h,\varphi_h) & \coloneqq \sum_{K\in \mathcal{K}_{h}} (\kappa \nabla \theta_h, \nabla  \varphi_h )_K 
   		 + \sum_{E \in \mathcal{E}_W} \int_E \beta \theta_h \varphi_h ds  +
   		\sum_{E\in \mathcal{E}_I}\bigg( -\int_{E} \kappa (\nabla \theta_h \cdot \boldsymbol{n}) \varphi_h d s \nonumber \\& \quad -\int_{E} \kappa(\nabla \varphi_h \cdot \boldsymbol{n}) \theta_h d s  +\gamma_N \int_{E} {h_e}^{-1}(\theta_h \varphi_h)d s \bigg) , \\ 
        O_T^h(\boldsymbol{u}_h; \theta_h, \varphi_h) & \coloneqq \sum_{K\in \mathcal{K}_{h}} (\boldsymbol{u}_h \cdot \nabla \theta_h , \varphi_h)_{K} -\sum_{E\in \mathcal{E}_O} \int_E (\boldsymbol{u}_h \cdot \boldsymbol{n}) \theta_h  \psi(\boldsymbol{u}_h \cdot \boldsymbol{n}) \varphi,  \nonumber \\
        F_v(\boldsymbol{v}_h) & \coloneqq \sum_{E\in \mathcal{E}_I}\bigg(-\int_{E} 2 \nu \varepsilon(\boldsymbol{v}_h) \cdot \boldsymbol{u}_{\star} d s + \gamma_N \int_{E} {h_e}^{-1}(\boldsymbol{u}_{\star} \cdot \boldsymbol{v}_h) d s  \bigg),
        \nonumber \\
        F_q(q_h) & \coloneqq  \sum_{E\in \mathcal{E}_I}\int_{E} q_h (\boldsymbol{n} \cdot \boldsymbol{u}_{\star}) d s,  
       \nonumber  \\
        F_{\theta}(\varphi_h) & \coloneqq \int_{\Omega} g \varphi_h ~ dx + \sum_{E\in \mathcal{E}_I}\bigg( -\int_{E} \kappa(\nabla \varphi_h \cdot \boldsymbol{n}) \theta_{\star} d s +\gamma_N \int_{E} {h_e}^{-1}(\theta_{\star} \varphi_h)d s \bigg).  \nonumber
  	    \end{align}

  	    We highlight that we have chosen a unified stabilization coefficient $\gamma_N$ for each of the Dirichlet boundary conditions to be imposed (slip, inlet-fluid, inlet-temperature). This is convenient only for the analysis, but preliminary numerical tests have shown only a mild sensitivity to having different parameters, so we have chosen the unified approach in practice as well for simplicity.
  		\subsubsection{Discrete stability properties}
 We need to define the energy norms as
\begin{align*}
\| \boldsymbol{v}_h \|_{1,h} &\coloneqq \| \varepsilon (\boldsymbol{v}_h) \|_{0,\Omega}^2 + \sum_{E \in \mathcal{E}_I} \frac{1}{h_e} \| \boldsymbol{v}_h  \|_{0,E}^2  + \sum_{E \in \mathcal{E}_W} \frac{1}{h_e} \| \boldsymbol{v}_h \cdot \boldsymbol{n} \|_{0,E}^2,  \\
\| \theta_h \|_{1,h} & \coloneqq \| \nabla \theta \|_{0,\Omega}^2 +  \sum_{E \in \mathcal{E}_I} \frac{1}{h_e} \| \theta_h \|_{0,E}^2, \\
\left|\!\left|\!\left| \boldsymbol{v}_h, \theta_h, q_h \right|\!\right|\!\right| & \coloneqq \| \varepsilon (\boldsymbol{v}_h) \|_{0,\Omega}^2 + \sum_{E \in \mathcal{E}_I} \frac{1}{h_e} \| \boldsymbol{v}_h  \|_{0,E}^2  + \sum_{E \in \mathcal{E}_W} \frac{1}{h_e} \| \boldsymbol{v}_h \cdot \boldsymbol{n} \|_{0,E}^2 +\| \nabla \theta \|_{0,\Omega}^2 +  \sum_{E \in \mathcal{E}_I} \frac{1}{h_e} \| \theta_h \|_{0,E}^2+  \|q\|_{0,\Omega}^2.
\end{align*} 

\newpage
  		\begin{theorem}\label{qa4}
  For a given function $\psi  \in L^{\infty}(\mathbb{R};\mathbb{R})$, there exist positive constants $C_1$, $C_2$, $C_3$, $C_4$, $C_5$, $C_6$, independent of $h$, such that
  		\begin{equation*}
  		\begin{aligned}
  		\left|{A}_S^h(\boldsymbol{u}_h, \boldsymbol{v}_h)\right| \leq& C_1 \|\boldsymbol{u}_h\|_{1, h}\|\boldsymbol{v}_h\|_{1,h}, &&  \forall \boldsymbol{u}_h, \boldsymbol{v}_h \in \mathrm{V}_h,
  		\\
  		\left|O_S^h\left(\boldsymbol{w}_h; \boldsymbol{u}_h, \boldsymbol{v}_h\right)\right| \leq& {C}_2 \left\|\boldsymbol{w}_h\right\|_{1,h}\|\boldsymbol{u}_h\|_{1,h}\|\boldsymbol{v}_h\|_{1, h}, && \forall \boldsymbol{w}_h, \boldsymbol{u}_h, \boldsymbol{v}_h \in \mathrm{V}_h,
  		\\
  		\left|{B}^h(\boldsymbol{v}_h, q_h)\right| \leq& C_3 \|\boldsymbol{v}_h\|_{1, h}\|q_h\|_{0, \Omega}, && \forall \boldsymbol{v}_h \in \mathrm{V}_h, q_h \in \Pi_h, \\ 
        \left| D( \theta_h, \boldsymbol{v}_h) \right|  \leq & C_4 \|\theta_h\|_{1,h}\|\boldsymbol{v}_h\|_{1,h} ,  && \forall \theta_h \in \mathrm{M}_h, \boldsymbol{v}_h \in \mathrm{V}_h, \\
        \left| A_T^h( \theta_h, \varphi_h)\right| \leq & C_5 \| \theta_h \|_{1,h} \| \varphi \|_{1,h} , && \forall  \theta_h, \varphi_h \in \mathrm{M}_h, \\
        \left| O_T^h(\boldsymbol{v}_h; \theta_h, \varphi_h) \right| \leq & C_6 \left\|\boldsymbol{v}_h\right\|_{1,h}\|\theta_h \|_{1,h}\|\varphi_h\|_{1,h} , && \forall \boldsymbol{v}_h \in \mathrm{V}_h, \theta_h, \varphi_h \in \mathrm{M}_h. 
  		\end{aligned}
  		\end{equation*}
        \end{theorem}  
  	\begin{proof}	
        The proof of the above inequalities is a direct consequence of the trace, inverse trace, Cauchy-Schwarz and H\"{o}lder inequalities, along with the usual Sobolev embedding theorems. 
  	\end{proof}
    \noindent The discrete kernel of ${B}_h$ is defined by:
  		$$
  		\mathrm{Z}_h \coloneqq \left\{\boldsymbol{v}_h \in \mathrm{V}_h: {B}_h\left( \boldsymbol{v}_h,p_h \right)=0 \ \ \ \forall {p}_h \in \mathrm{\Pi}_h\right\}.
  		$$
  		It can be equivalently written as
  		\begin{align}\label{a7}
  		\mathrm{Z}_h \coloneqq \left\{\boldsymbol{v}_h \in \mathrm{V}_h: \sum_{K\in \mathcal{K}_{h}} \int_{K} p_h \nabla \cdot \boldsymbol{v}_h d \boldsymbol{x} - \sum_{E\in \mathcal{E}_{\text{Nav}}} \int_{E} {p_h} (\boldsymbol{n} \cdot \boldsymbol{v}_h) ds =0 \quad \forall p_h \in \Pi_h \right\}.
  		\end{align}
  		The following lemma establishes the ellipticity of the above bilinear forms.
  		\begin{lemma}\label{S}
  			 There exist positive constants $\gamma_N, C_0$ and $C_S$, independent of $h$, such that
  			\begin{align}\label{G}
  			A^h_S(\boldsymbol{v}_h, \boldsymbol{v}_h) \geq C_S \|\boldsymbol{v}_h\|^2_{1,h} \quad \forall \boldsymbol{v}_h \in \mathrm{Z}_h,
  			\end{align}
  			where $C_S = \min \{\xi - {4 \nu  C_{5}^2}{ C_0},\gamma_N -\frac{2 \nu}{C_0}\}$ with $\gamma_N  > \frac{2 \nu}{C_0}$, $C_0 < \frac{\xi}{ 4 \nu C_{5}^2 }$ and $\xi = \min \{ 2 \nu, \gamma \}$.
  		\end{lemma}
  		\begin{proof}
  		We use \eqref{F} to obtain 
  			\begin{align*}
  			A^h_S(\boldsymbol{v}_h, \boldsymbol{v}_h) &=\sum_{K\in \mathcal{K}_{h}} 2 \nu (\varepsilon(\boldsymbol{v}_h), \varepsilon(\boldsymbol{v}_h))_{K} + \sum_{E\in \mathcal{E}_W}\bigg(-\int_{E} 4 \nu \varepsilon(\boldsymbol{v}_h) \boldsymbol{n} \cdot \boldsymbol{n} (\boldsymbol{n} \cdot \boldsymbol{v}_h) d s  + \int_{E} \gamma \sum_{i=1}^{d-1}\left(\boldsymbol{\tau}^i \cdot \boldsymbol{v}_h \right)\left(\boldsymbol{\tau}^i \cdot \boldsymbol{v}_h\right) d s \\ &
  		+\gamma_N \int_{E} {h_e}^{-1}(\boldsymbol{v}_h \cdot \boldsymbol{n})(\boldsymbol{v}_h \cdot \boldsymbol{n}) d s \bigg) + 
   		\sum_{E\in \mathcal{E}_I}\bigg(-\int_{E} 4 \nu \varepsilon(\boldsymbol{v}_h) \boldsymbol{n} \cdot  \boldsymbol{v}_h d s + \gamma_N \int_{E} {h_e}^{-1}(\boldsymbol{v}_h \cdot \boldsymbol{v}_h) d s  \bigg).
  			\end{align*}
  			Using the trace, inverse trace, Cauchy-Schwarz and H\"{o}lder inequalities, the following estimate can be established 
  			\begin{align*}
  			A^h_S(\boldsymbol{v}_h, \boldsymbol{v}_h) &\geq \xi  \| \varepsilon( \boldsymbol{v}_h) \|^2_{0,\Omega} -4 \nu \sum_{E \in \mathcal{E}_W} \|\boldsymbol{v}_h \cdot \boldsymbol{n} \|_{0,E} \| \varepsilon  (\boldsymbol{v}_h) \boldsymbol{n} \|_{0,E} +  \frac{\gamma_N}{h_e} \sum_{E \in \mathcal{E}_W} \|\boldsymbol{v}_h \cdot \boldsymbol{n} \|_{0,E}^2  \\ & \quad -4 \nu \sum_{E \in \mathcal{E}_I} \|\boldsymbol{v}_h \|_{0,E} \| \varepsilon  (\boldsymbol{v}_h) \boldsymbol{n} \|_{0,E} +  \frac{\gamma_N}{h_e} \sum_{E \in \mathcal{E}_I} \|\boldsymbol{v}_h  \|_{0,E}^2 
  			 \\  & \hspace{-5mm} \geq \xi \| \varepsilon( \boldsymbol{v}_h) \|^2_{0,\Omega} - 4 \nu \sum_{E \in \mathcal{E}_W} h_e^{-1/2} \|\boldsymbol{v}_h \cdot \boldsymbol{n} \|_{0,E} h_e^{1/2} \| \varepsilon(\boldsymbol{v}_h) \boldsymbol{n} \|_{0,E} +  \frac{\gamma_N}{h_e} \sum_{E \in \mathcal{E}_W} \|\boldsymbol{v}_h \cdot \boldsymbol{n} \|_{0,E}^2 \\& \- 4 \nu \sum_{E \in \mathcal{E}_I} h_e^{-1/2} \|\boldsymbol{v}_h \|_{0,E} h_e^{1/2} \| \varepsilon(\boldsymbol{v}_h) \boldsymbol{n} \|_{0,E} +  \frac{\gamma_N}{h_e} \sum_{E \in \mathcal{E}_I} \|\boldsymbol{v}_h  \|_{0,E}^2  \\ & \hspace{-5mm}
  			\geq \xi \| \varepsilon( \boldsymbol{v}_h) \|^2_{0,\Omega} - 4 \nu  \sum_{E \in \mathcal{E}_W} \left( \frac{h_e C_0}{2} \| \varepsilon(\boldsymbol{v}_h)  \|_{0,E}^2   +   \frac{h_e^{-1}}{2 C_0}  \|\boldsymbol{v}_h \cdot \boldsymbol{n} \|_{0,E}^2 \right)+  \frac{\gamma_N}{h_e} \sum_{E \in \mathcal{E}_W} \|\boldsymbol{v}_h \cdot \boldsymbol{n} \|_{0,E}^2 \\ & - 4 \nu  \sum_{E \in \mathcal{E}_I} \left( \frac{h_e C_0}{2} \| \varepsilon(\boldsymbol{v}_h)  \|_{0,E}^2   +   \frac{h_e^{-1}}{2 C_0}  \|\boldsymbol{v}_h \|_{0,E}^2 \right)+  \frac{\gamma_N}{h_e} \sum_{E \in \mathcal{E}_W} \|\boldsymbol{v}_h \|_{0,E}^2   \\ & \hspace{-5mm}
  			\geq \xi \| \varepsilon( \boldsymbol{v}_h) \|^2_{0,\Omega} - 4 \nu {C_{5}^2}{C_0} \sum_{K \in \mathcal{T}_h} \| \varepsilon(\boldsymbol{v}_h)  \|_{0,K}^2 -  
  			 \frac{2 \nu}{C_0 \gamma_N} \sum_{E \in \mathcal{E}_W} \frac{\gamma_N}{h_e}	\|\boldsymbol{v}_h \cdot \boldsymbol{n} \|_{0,E}^2  +  \frac{\gamma_N}{h_e} \sum_{E \in \mathcal{E}_W} \|\boldsymbol{v}_h \cdot \boldsymbol{n} \|_{0,E}^2 \\&  -  
  			 \frac{2 \nu}{C_0 \gamma_N} \sum_{E \in \mathcal{E}_I} \frac{\gamma_N}{h_e}	\|\boldsymbol{v}_h  \|_{0,E}^2  +  \frac{\gamma_N}{h_e} \sum_{E \in \mathcal{E}_I} \|\boldsymbol{v}_h \|_{0,E}^2  
  	         \\  & \hspace{-5mm} \geq \xi \| \varepsilon( \boldsymbol{v}_h) \|^2_{0,\Omega} - {4 \nu C_{5}^2}{C_0} \| \nabla \boldsymbol{v}_h \|^2_{0,\Omega} +  \left(1-\frac{2 \nu }{C_0 \gamma_N}\right) \sum_{E \in \mathcal{E}_W} \frac{\gamma_N}{h_e} \|\boldsymbol{v}_h \cdot \boldsymbol{n} \|_{0}^2 +  \left(1-\frac{2 \nu }{C_0 \gamma_N}\right) \sum_{E \in \mathcal{E}_W} \frac{\gamma_N}{h_e} \|\boldsymbol{v}_h \|_{0}^2 
  	         \\& \hspace{-5mm}  \geq \left(\xi - {4 \nu C_{5}^2}{C_0}\right) \| \varepsilon( \boldsymbol{v}_h) \|^2_{0,\Omega} + \left(\gamma_N-\frac{2 \nu}{C_0}\right) \sum_{E \in \mathcal{E}_W} \frac{1}{h_e} \|\boldsymbol{v}_h \cdot \boldsymbol{n} \|_{0,E}^2 + \left(\gamma_N-\frac{2 \nu}{C_0}\right) \sum_{E \in \mathcal{E}_I} \frac{1}{h_e} \|\boldsymbol{v}_h  \|_{0,E}^2
  	        \\& \hspace{-5mm} \geq C_S \|\boldsymbol{v}_h\|^2_{1,h}.
  			\end{align*}
  			By selecting the positive parameter $C_0$ such that $C_0 < \frac{\xi}{ 4 \nu C_{5}^2 }$, we ensure that $\left(\xi - {4 \nu C_{5}^2}{ C_0}\right) >0 $. Additionally, we define $C_S = \min \{\xi - {4 \nu C_{5}^2}{ C_0},\gamma_N -\frac{2 \nu}{C_0}\}$ with $\gamma_N > \frac{2 \nu}{C_0}$.
  		\end{proof}
  	\begin{lemma}\label{infsup}
  				There exists $\hat{\theta}>0$, independent of $h$, such that
  				\begin{align}\label{46}
  				\sup _{\boldsymbol{0} \neq \boldsymbol{v}_h \in \mathrm{V}_h} \frac{\left | B^h( \boldsymbol{v}_h,q_h) \right | }{\|\boldsymbol{v}_h\|_{1,h}} \geq \hat{\theta} \| q_h \|_{0,\Omega} \quad \forall q \in \Pi_h. 
  				\end{align}
  			\end{lemma}	
  \begin{proof}
 See \cite{MR4812237}. 
\end{proof}	
 \begin{lemma}\label{coercivity_t}
  There exists a positive constant $C_T$, independent of $h$, such that
  			\begin{align}\label{temp_G}
  			A^h_T(\theta_h, \theta_h) \geq C_T \|\theta_h\|^2_{1,\Omega} \quad \forall \theta_h \in \mathrm{M}_h,
  			\end{align}
              where $ C_T = \min \{ \kappa - C_0 C_5^2, \gamma_N - \frac{1}{C_0} \} $ with $\gamma_N \geq \gamma_0 > \frac{1}{C_0}$, $C_0 < \frac{\kappa}{ C_{5}^2 }$.
 \end{lemma}
 \begin{proof}
     We use the trace, inverse trace, Cauchy-Schwarz and H\"{o}lder inequalities to obtain 
     \begin{align*}
         {A}^h_T(\theta_h,\theta_h) & \coloneqq \sum_{K\in \mathcal{K}_{h}} (\kappa \nabla \theta_h, \nabla  \theta_h )_K 
   		 + \sum_{E \in \mathcal{E}_W} \int_E \beta (\theta_h )^2 ds  +
   		\sum_{E\in \mathcal{E}_I}\bigg( - 2 \int_{E} (\nabla \theta_h \cdot \boldsymbol{n}) \theta_h d s +\gamma_N \int_{E} {h_e}^{-1}(\theta_h)^2 d s \bigg) \\& \geq  \kappa \| \nabla \theta_h \|_{0,\Omega} - 2   \sum_{E \in \mathcal{E}_I} \left( \frac{h_e C_0}{2} \|  \nabla \theta_h   \|_{0,E}^2   +   \frac{h_e^{-1}}{2 C_0}  \|\theta_h  \|_{0,E}^2 \right)+  \frac{\gamma_N}{h_e} \sum_{E \in \mathcal{E}_W} \|\theta_h \|_{0,E}^2  \\& \geq (\kappa - C_0 C_5^2) \| \nabla \theta_h \|_{0,\Omega}  + \left( \gamma_N - \frac{1}{C_0}\right)  \frac{1}{h_e} \sum_{E \in \mathcal{E}_W} \|\theta_h \|_{0,E}^2 \\ & \geq C_T \| \theta_h\|^2_{1,h}.
     \end{align*}
 \end{proof}
\section{Existence and uniqueness of discrete   solutions}\label{section3}
In this section, we aim to establish the well-posedness of problem \eqref{E}. We will employ a fixed-point operator linked to a linearised form of the problem and demonstrate that this operator has a  unique fixed point. This will allow us to prove the well-posedness of problem \eqref{E} using the Banach fixed-point theorem.
  		
\subsubsection{The discrete fixed-point operator and its well-posedness}
  Let us introduce the set
  		\begin{align}\label{qa5}
  		\boldsymbol{K}_h &\coloneqq\left\{(\boldsymbol{u}_h, \theta_h) \in \mathrm{V}_h \times \mathrm{M}_h :\|(\boldsymbol{u}_h, \theta_h)\|_{1,h} \leq {\hat{\alpha}}^{-1}C(\kappa, \nu, \boldsymbol{u}_{\star}, \theta_{\star}, \boldsymbol{f}, g, \alpha) \right\}, 
  		\end{align}
  		with $\hat{\alpha}>0$ being the constant defined below in Theorem~\ref{51}. Now, we define the discrete fixed-point operator as
  		\begin{align*}
        \mathcal{J}_h: \boldsymbol{K}_h \rightarrow \boldsymbol{K}_h, &\quad (\boldsymbol{w}_h,\psi_h) \rightarrow \mathcal{J}_h\left(\boldsymbol{w}_h, \psi_h\right)=(\boldsymbol{u}_h, \theta_h),
  		\end{align*} 
  		where, for a given $(\boldsymbol{w}_h,\psi_h) \in \boldsymbol{K}_h, 
     (\boldsymbol{u}_h, \theta_h)$ represents the first component of the solution of the following linearised version of problem \eqref{E}: Find $\left( \boldsymbol{u}_h, \theta_h, p_h \right) \in \mathrm{V}_h \times \mathrm{M}_h \times  \mathrm{\Pi}_h $ 
  		\begin{equation}\label{43}
  		\begin{array}{rlrl}
        {A}_S^h(\boldsymbol{u}_h, \boldsymbol{v}_h) + {O}_S^h(\boldsymbol{w}_h ; \boldsymbol{u}_h, \boldsymbol{v}_h) + {B}^h(\boldsymbol{v}_h,p_h) & = D(\psi_h,\boldsymbol{v}_h)+F_v(\boldsymbol{v}_h) & \forall \boldsymbol{v}_h \in \mathrm{V}_h, \\
  	    {B}^h(\boldsymbol{u}_h, q_h) & = F_q(q_h) & \forall q_h \in \mathrm{\Pi}_h, \\ 
       {A}^h_T(\theta_h,\varphi_h) + O_T(\boldsymbol{w}_h; \theta_h, \varphi_h) & = F_{\varphi}(\varphi_h) & \forall \varphi_h \in \mathrm{M}_{h}.
  		\end{array}
  		\end{equation}
  		Based on the above, we can establish the following relation
    \begin{align}\label{44}
  		\mathcal{J}_h\left(\boldsymbol{u}_h, \theta_h\right)=(\boldsymbol{u}_h,\theta_h) \Leftrightarrow\left( \boldsymbol{u}_h,\theta_h,p_h\right) \in \mathrm{V}_h \times \mathrm{M}_h \times  \mathrm{\Pi}_h 
  \quad\text{satisfies} \ \ \ \eqref{E}.
\end{align}
  		
 To guarantee the well-posedness of the discrete problem \eqref{E}, it is enough to demonstrate the existence of a unique fixed-point for $\mathcal{J}_h$ within the set $\boldsymbol{K}_h$. However, before delving into the solvability analysis, we first need to establish that the operator $\mathcal{J}_h$ is well-defined.
Let us introduce the bilinear form.
  		\begin{align}\label{aa}
  		\mathcal{C}_h\left[(\boldsymbol{u}_h, \theta_h, p_h);( \boldsymbol{v}_h,\varphi_h, q_h)\right]={A}_S^h(\boldsymbol{u}_h, \boldsymbol{v}_h) + {B}^h(\boldsymbol{v}_h,p_h) + {B}^h(\boldsymbol{u}_h, q_h) +  {A}^h_T(\theta_h,\varphi_h) .
  		\end{align}
  		\begin{theorem}\label{51}
  			There exist a positive constant $\hat{\alpha}= \hat{\alpha}(C_S,C_T,\hat{\theta})$ such that 
  			\begin{align*}
  			\sup _{\boldsymbol{0} \neq (\boldsymbol{v}_h,\varphi_h , q_h) \in \mathrm{V}_h \times \mathrm{M}_h \times  \mathrm{\Pi}_h}\frac{\mathcal{C}_h\left[\left( \boldsymbol{u}_h, \theta_h, p_h\right);\left( \boldsymbol{v}_h, \varphi_h, q_h \right)\right]}{\left\|\left( \boldsymbol{v}_h, \varphi_h, q_h \right)\right\|} \geq \hat{\alpha}\left\|\left(\boldsymbol{u}_h,\theta_h, p_h\right)\right\| \quad \forall\left( \boldsymbol{u}_h,\varphi_h,p_h\right) \in \mathrm{V}_h \times \mathrm{M}_h \times \mathrm{\Pi}_h,
  			\end{align*}
  			with $C_S$, $C_T$ and $\hat{\theta}$ are the coercivity and inf-sup stability constants. 
  		\end{theorem}
\begin{proof}
    Owing to Theorem~\ref{qa4}, it is clear that $\mathcal{C}_h\left[\cdot ; \cdot \right]$ is bounded. Moreover, from Lemma~\ref{S}, Lemma~\ref{infsup}, and \cite[Proposition 2.36]{MR2050138}, it is not difficult to see that the above inf-sup condition holds.
\end{proof}

Now, we are in position to establish the well-posedness of $\mathcal{J}_{\text {h }}$.
  		\begin{theorem}\label{qa7}
  			Assume that
  			\begin{align}\label{53}
  			\frac{2}{\hat{\alpha}^2} C(\kappa, \nu, \boldsymbol{u}_{\star}, \theta_{\star}, \boldsymbol{f}, g, \alpha) \leq 1.
  			\end{align}
  			Then, given $(\boldsymbol{w}_h, \psi_h) \in \boldsymbol{K}_h$, there exists a unique $(\boldsymbol{u}_h, \theta_h) \in \boldsymbol{K}_h$ such that $\mathcal{J}_h\left(\boldsymbol{w}_h, \psi_h \right)=(\boldsymbol{u}_h, \theta_h)$.	
  		\end{theorem} 
  		\begin{proof}
  			Given $ (\boldsymbol{w}_h, \psi_h) \in \boldsymbol{K}_h$, we begin by defining the bilinear forms:
  			\begin{align*}
  			\mathcal{A}_{\boldsymbol{w}_h}\left[(\boldsymbol{u}_h,\theta_h, p_h);(\boldsymbol{v}_h,\varphi_h, q_h)\right]&:=\mathcal{C}_h\left[( \boldsymbol{u}_h, \theta_h,p_h);( \boldsymbol{v}_h, \varphi_h, q_h)\right]+ {O}_S^h(\boldsymbol{w}_h ; \boldsymbol{u}_h, \boldsymbol{v}_h) + O_T(\boldsymbol{w}_h; \theta_h, \varphi_h),
            \\ 
            \mathcal{F}_{\psi_h}(\boldsymbol{v}_h, \varphi_h, q_h) & := D(\psi_h,\boldsymbol{v}_h)+F_v(\boldsymbol{v}_h) + F_q(q_h) + F_{\varphi}(\varphi_h), 
  			\end{align*}
  			where $\mathcal{C}_h$, $O_S^h$ and $O_T$ are the forms defined in Theorem~\ref{51} and \eqref{F}, respectively, that is
  			\begin{align}\label{A_wh}
  			\mathcal{A}_{\boldsymbol{w}_h}\left[(\boldsymbol{u}_h,\theta_h,p_h);(\boldsymbol{v}_h,\varphi_h,q_h)\right]&={A}_S^h(\boldsymbol{u}_h, \boldsymbol{v}_h) + {B}^h(\boldsymbol{v}_h,p_h)  + {B}^h(\boldsymbol{u}_h, q_h) +  {A}^h_T(\theta_h,\varphi_h)  \nonumber \\ & \quad + {O}_S^h(\boldsymbol{w}_h ; \boldsymbol{u}_h, \boldsymbol{v}_h) + O_T(\boldsymbol{w}_h; \theta_h, \varphi_h).
  			\end{align}
  			Then, problem \eqref{E} can be rewritten equivalently as: Given $(\boldsymbol{w}_h,\psi_h)$ in $\boldsymbol{K}_h$, find $(\boldsymbol{u}_h, \theta_h, p_h) \in \mathrm{V}_h \times \mathrm{M}_h, \Pi_h$, such that
  			\begin{align}\label{32}
  			\mathcal{A}_{\boldsymbol{w}_h}\left[(\boldsymbol{u}_h, \theta_h, p_h);(\boldsymbol{v}_h, \varphi_h, q_h)\right]=\mathcal{F}_{\psi_h}(\boldsymbol{v}_h, \varphi_h, q_h) \quad \forall( \boldsymbol{v}_h,q_h) \in \mathrm{V}_h \times \Pi_h.
  			\end{align}
  			First, we establish that $\mathcal{J}_h$ is well defined in order to demonstrate the well-posedness of problem \eqref{32} using the Banach-Ne\v{c}as-Babu\v{s}ka theorem  \cite[Theorem 2.6] {MR2050138}.
  		     Consider $(\boldsymbol{u}_h, \theta_h, p_h),( \hat{\boldsymbol{v}}_h, \hat{\varphi}_h, \hat{q}_h) \in \mathrm{V}_h \times \mathrm{M}_h \times \Pi_h$ with $(\hat{\boldsymbol{v}}_h, \hat{\varphi}_h, \hat{q}_h) \neq 0$. From Theorem~\ref{qa4} we can observe that
  			$$
  			\begin{aligned}
  			& \sup _{\boldsymbol{0} \neq(\boldsymbol{v}_h, \varphi_h, p_h) \in \mathrm{V}_h \times \Pi_h} \frac{\mathcal{A}_{\boldsymbol{w}_h}\left[(\boldsymbol{u}_h,\theta_h, p_h);(\boldsymbol{v}_h, \varphi_h, q_h)\right]}{\|(\boldsymbol{v}_h, \varphi_h, q_h)\|} \\[3pt]
            & \geq \frac{|\mathcal{C}_h\left[(\boldsymbol{u}_h, \varphi_h, p_h);( \hat{\boldsymbol{v}}_h, \hat{\varphi}_h, \hat{q}_h)\right]|}{\|( \hat{\boldsymbol{v}}_h, \hat{\varphi}_h,\hat{q}_h)\|}-\frac{|{O}_S^h(\boldsymbol{w}_h ; \boldsymbol{u}_h, \hat{\boldsymbol{v}}_h) + O_T(\boldsymbol{w}_h; \theta_h, \hat{\varphi}_h)|}{\|(\hat{\boldsymbol{v}}_h, \hat{\varphi}_h, \hat{q}_h)\|} \\[3pt]
  			& \geq \frac{|\mathcal{C}_h\left[(\boldsymbol{u}_h, \theta_h, p_h);( \hat{\boldsymbol{v}}_h, \hat{\varphi}_h, \hat{q}_h)\right]|}{\|( \hat{\boldsymbol{v}}_h, \hat{\varphi}_h,\hat{q}_h)\|}-\|\boldsymbol{w}_h\|_{1,h}\|(\boldsymbol{u}_h, \theta_h, p_h)\| ,
  			\end{aligned}
  			$$
  			which, together with Theorem~\ref{51} and the fact that $(\hat{\boldsymbol{v}}_h, \hat{\varphi}_h, \hat{p}_h)$ is arbitrary, implies
  			\begin{align}\label{33}
  			\sup _{\boldsymbol{0} \neq (\boldsymbol{v}_h, \varphi_h, p_h) \in \mathrm{V}_h \times \mathrm{M}_h \times \Pi_h} \frac{\mathcal{A}_{\boldsymbol{w}_h}\left[(\boldsymbol{u}_h,\theta_h, p_h);(\boldsymbol{v}_h, \varphi_h, q_h)\right]}{\|(\boldsymbol{v}_h, \varphi_h, p_h)\|} \geq\left(\hat{\alpha}-\|\boldsymbol{w}_h\|_{1,h}\right)\|( \boldsymbol{u}_h, \theta_h, p_h)\|,
  			\end{align}
  		where $\hat\alpha$ is independent of $\boldsymbol{w}_h$. Hence, from the definition of set $\boldsymbol{K}_h$ (see \eqref{qa5}) and assumption \eqref{53}, we get
  			\begin{align}\label{34}
  			\|\boldsymbol{w}_h\|_{1,h} \leq \frac{1}{\hat{\alpha}} C(\kappa, \nu, \boldsymbol{u}_{\star}, \theta_{\star}, \boldsymbol{f}, g, \alpha) \leq \frac{\hat{\alpha}}{2}.
  			\end{align}
  			Then, combining \eqref{33} and \eqref{34}, we obtain
  			\begin{align}\label{35}
  			\sup _{\boldsymbol{0} \neq  (\boldsymbol{v}_h,\theta_h, p_h) \in \mathrm{V}_h \times \mathrm{M}_h \times  \Pi_h} \frac{\mathcal{A}_{\boldsymbol{w}_h}\left[(\boldsymbol{u}_h, \theta_h, p_h);(\boldsymbol{v}_h, \varphi_h, q_h)\right]}{\|(\boldsymbol{v}_h, \varphi_h, q_h)\|} \geq \frac{\hat{\alpha}}{2}\|(\boldsymbol{u}_h,\theta_h, p_h)\| \quad \forall(\boldsymbol{u}_h, \theta_h, p_h) \in  \mathrm{V}_h \times \mathrm{M}_h \times \Pi_h.
  			\end{align}
  			On the other hand, for a given $ (\boldsymbol{u}_h, \theta_h, q_h) \in \mathrm{V}_h \times \mathrm{M}_h \times  \Pi_h$, we observe that
  			$$
  			\begin{aligned}
  			& \sup_{\boldsymbol{0} \neq (\boldsymbol{v}_h, \varphi_h, q_h) \in \boldsymbol{V}_h \times \mathrm{M}_h \times \Pi_h\hspace{-1cm}}   \mathcal{A}_{\mathrm{w}_h}\left[(\boldsymbol{v}_h, \varphi_h, q_h);(\boldsymbol{u}_h, \theta_h,  p_h)\right]  \geq \sup_{\boldsymbol{0} \neq(\boldsymbol{v}_h, \varphi_h, q_h) \in \mathrm{V}_h \times \mathrm{M}_h \times  \Pi_h} \frac{\mathcal{A}_{\boldsymbol{w}_h}\left[(\boldsymbol{v}_h, \varphi_h, q_h);(\boldsymbol{u}_h, \theta_h, p_h)\right]}{\|(\boldsymbol{v}_h, \varphi_h, q_h)\|} \\
  			&  =\sup _{\boldsymbol{0} \neq(\boldsymbol{v}_h, \varphi_h, q_h) \in \mathrm{V}_h \times \mathrm{M}_h \times \Pi_h} \frac{\mathcal{C}_h\left[(\boldsymbol{v}_h, \varphi_h, q_h);(\boldsymbol{u}_h, \theta_h, q_h)\right]+{O}_S^h(\boldsymbol{w}_h ; \boldsymbol{u}_h, \boldsymbol{v}_h) + O_T(\boldsymbol{w}_h; \theta_h, \varphi_h)}{\|(\boldsymbol{v}_h, \varphi_h, q_h)\|},
  			\end{aligned}
  			$$
  			from which,
  			$$
  			\begin{aligned}
  			& \sup _{\boldsymbol{0} \neq (\boldsymbol{v}_h, \varphi_h, q_h) \in \mathrm{V}_h \times \mathrm{M}_h \times \Pi_h} \mathcal{A}_{\boldsymbol{w}_h}\left[( \boldsymbol{v}_h, \varphi_h, q_h);(\boldsymbol{u}_h, \theta_h, q_h)\right] \\ & \geq \sup _{\boldsymbol{0} \neq(\boldsymbol{v}_h, \varphi_h, q_h) \in \mathrm{V}_h \times \mathrm{M}_h \times \Pi_h}\frac{|\mathcal{C}_h\left[(\boldsymbol{v}_h, \varphi_h, q_h);( \boldsymbol{u}_h, \theta_h, p_h)\right]+{O}_S^h(\boldsymbol{w}_h ; \boldsymbol{u}_h, \boldsymbol{v}_h) + O_T(\boldsymbol{w}_h; \theta_h, \varphi_h)|}{\|( \boldsymbol{v}_h, \varphi_h, q_h)\|} \\
  			& \geq \sup _{\boldsymbol{0} \neq(\boldsymbol{v}_h, \varphi_h, q_h) \in \mathrm{V}_h \times \mathrm{M}_h \times \Pi_h} \frac{|\mathcal{C}_h\left[(\boldsymbol{v}_h, \varphi_h, q_h);(\boldsymbol{u}_h, \theta_h, p_h)\right]|}{\|( \boldsymbol{v}_h,\varphi_h,q_h)\|}-\frac{|{O}_S^h(\boldsymbol{w}_h ; \boldsymbol{u}_h, \boldsymbol{v}_h) + O_T(\boldsymbol{w}_h; \theta_h, \varphi_h)|}{\|( \boldsymbol{v}_h, \varphi_h, q_h)\|},
  			\end{aligned}
  			$$
  			for all $\boldsymbol{0} \neq (\boldsymbol{v}_h, \varphi_h, q_h) \in \mathrm{V}_h \times \mathrm{M}_h \times \Pi_h $, which together with Theorem~\ref{qa4}, implies
  			\begin{align}\label{36}
  			\sup _{\boldsymbol{0} \neq (\boldsymbol{v}_h,\varphi_h, q_h) \in \mathrm{V}_h \times \mathrm{M}_h \times  \Pi_h} & \mathcal{A}_{\boldsymbol{w}_h}\left[(\boldsymbol{v}_h, \varphi_h, q_h);(\boldsymbol{u}_h,\theta_h, p_h)\right] \\ & \geq \sup _{\boldsymbol{0} \neq ( \boldsymbol{v}_h,\varphi_h, q_h) \in \mathrm{V}_h \times \mathrm{M}_h \times  \Pi_h} \frac{\mathcal{C}_h\left[( \boldsymbol{v}_h, \varphi_h, q_h);(\boldsymbol{u}_h, \theta_h, p_h)\right]}{\|(\boldsymbol{v}_h, \varphi_h, q_h)\|}-\|\boldsymbol{w}_h\|_{1,h}\|( \boldsymbol{u}_h, \theta_h, p_h)\|.
  			\end{align}
  			Therefore, using the fact that $\mathcal{C}_h\left[\cdot ;\, \cdot\right]$ is symmetric, from the inequality in Theorem~\ref{51} and \eqref{36} we obtain
  			\begin{align*}
  			\sup _{\boldsymbol{0} \neq (\boldsymbol{v}_h, \varphi_h, q_h) \in \mathrm{V}_h \times \mathrm{M}_h \times  \Pi_h} \mathcal{A}_{\boldsymbol{w}_h}\left[( \boldsymbol{v}_h, \varphi_h, q_h);(\boldsymbol{u}_h, \theta_h, p_h)\right] \geq \hat{\alpha} \|( \boldsymbol{u}_h, \theta_h, p_h)\|-\|\boldsymbol{w}_h\|_{1,h}\|( \boldsymbol{u}_h, \theta_h, p_h)\|,
  			\end{align*}
  			which combined with \eqref{34}, yields
  			\begin{align}\label{37}
  			\sup _{\boldsymbol{0} \neq (\boldsymbol{v}_h, \varphi_h, q_h) \in \mathrm{V}_h \times \mathrm{M}_h \times \Pi_h} \mathcal{A}_{\boldsymbol{w}_h}\left[(\boldsymbol{v}_h, \varphi_h, q_h);(\boldsymbol{u}_h, \theta_h, p_h)\right] & \geq \frac{\hat{\alpha}}{2}\|(\boldsymbol{u}_h, \theta_h, p_h)\| >0  \\ & \forall(\boldsymbol{u}_h,\theta_h, p_h) \in \mathrm{V}_h \times \mathrm{M}_h \times \Pi_h,(\boldsymbol{u}_h, \theta_h, p_h) \neq 0.
  			\end{align}
  			By examining \eqref{35} and \eqref{37}, we can deduce that $\mathcal{A}_{\boldsymbol{w}_h}\left( \cdot , \cdot\right)$ satisfies the conditions of the Banach-Ne\v{c}as-Babu\v{s}ka theorem \cite[Theorem 2.6]{MR2050138}. This guarantees the existence of a unique solution $(\boldsymbol{u}_h,\theta_h, p_h) \in \mathrm{V}_h \times \mathrm{M}_h \times  \Pi_h$ to \eqref{E}, or equivalently, the existence of a unique $(\boldsymbol{u}_h , \theta_h) \in \mathrm{V}_h \times \mathrm{M}_h$ such that $\mathcal{J}_h(\boldsymbol{w}_h, \psi_h)=(\boldsymbol{u}_h, \theta_h)$. Furthermore, from \eqref{32} and \eqref{35} we derive the following inequality:
  			$$
  			\|(\boldsymbol{u}_h, \theta_h)\| \leq \|(\boldsymbol{u}_h, \theta_h, p_h)\| \leq \frac{1}{\hat{\alpha}} C(\kappa, \nu, \boldsymbol{u}_{\star}, \theta_{\star}, \boldsymbol{f}, g, \alpha).
  			$$
  		This concludes the proof by showing that $(\boldsymbol{u_h}, \theta_h) \in \boldsymbol{K_h}$.	
  		\end{proof} 
  		\subsubsection{Well-posedness of the discrete problem}
  		The subsequent theorem establishes the well-posedness of Nitsche's scheme \eqref{F}.
  		\begin{theorem}\label{van}
  			Let $(\boldsymbol{f},g) \in \boldsymbol{L}^2(\Omega) \times L^{2}(\Omega)$ and $\boldsymbol{u}_{\star}, \theta_{\star}$ be the boundary data such that
  			\begin{align}\label{55}
  			\frac{2}{\hat{\alpha}^2}C(\kappa, \nu, \boldsymbol{u}_{\star}, \theta_{\star}, \boldsymbol{f}, g, \alpha) \leq 1.
  	    	\end{align}
  			Then, there exists a unique $\left( \boldsymbol{u}_h, \theta_h, p_h\right) \in \mathrm{V}_h \times \mathrm{M}_h \times \mathrm{\Pi}_h$ solution to \eqref{E}. In addition, 
  			\begin{align}\label{qa8}
  			\left\|\boldsymbol{u}_h\right\|_{1,h}+ \left\|\theta_h \right\|_{1,\Omega} + \left\|p_h\right\|_{0,\Omega} \lesssim C(\kappa, \nu, \boldsymbol{u}_{\star}, \theta_{\star}, \boldsymbol{f}, g, \alpha).
  			\end{align}
  		\end{theorem}
  		\begin{proof}
  			
  			According to the relations given in \eqref{44}, we aim to establish the well-posedness of \eqref{E}. This can be accomplished by using Banach's Theorem to demonstrate that $\mathcal{J}_h$ possesses a unique fixed point in $\boldsymbol{K}_h$.
  			\newline
  			The validity of Assumption \eqref{55} as shown in Theorem~\ref{qa7}, ensures that $\mathcal{J}_h$ is well-defined.
  			Now, let $\boldsymbol{w}_{h1}$, $\boldsymbol{w}_{h2}$, $\boldsymbol{u}_{h1}$, $\boldsymbol{u}_{h2}$ $\in \boldsymbol{K}_h$, be such that $\boldsymbol{u}_{h1}=\mathcal{J}_h\left(\boldsymbol{w}_{h1}\right)$ and $\boldsymbol{u}_{h2}=\mathcal{J}_h\left(\boldsymbol{w}_{h2}\right)$. By employing the definition of 
     $\mathcal{J}_h$ and \eqref{32}, it follows that there exist a unique pair $( \theta_{h1},p_{h1}), (\theta_{h2},p_{h2}) \in \mathrm{M}_h \times {L}^2(\Omega)$ satisfying the following equations:
  		\begin{align*}
        \mathcal{A}_{\boldsymbol{w}_{h1}}\left[\left(\boldsymbol{u}_{h1},\theta_{h1},p_{h1}\right);( \boldsymbol{v}_h,\varphi_h,q_h)\right]& =\mathcal{F}_{\psi_{h1}}(\boldsymbol{v}_h,\varphi_h,q_h), \\ \mathcal{A}_{\boldsymbol{w}_{h2}}\left[\left( \boldsymbol{u}_{h2},\theta_{h2},p_{h2} \right);( \boldsymbol{v}_h,\varphi_h,q_h)\right]&=\mathcal{F}_{\psi_{h2}}(\boldsymbol{v}_h,\varphi_h,q_h), \quad \forall (\boldsymbol{v}_h,\varphi_h,q_h) \in \mathrm{V}_h \times \mathrm{M}_h \times \Pi_h.  
  		\end{align*}
  			By adding and subtracting appropriate terms, we can derive the following:
            \begin{equation} \label{40}
  			\begin{aligned}
  			\mathcal{A}_{\boldsymbol{w}_{h1}}\left[\left( \boldsymbol{u}_{h1}-\boldsymbol{u}_{h2}, \theta_{h1}-\theta_{h2}, p_{h1}-p_{h2}\right);( \boldsymbol{v}_h, \varphi_h,q_h)\right] & =-O_S^h\left(\boldsymbol{w}_{h1}-\boldsymbol{w}_{h2} ;  \boldsymbol{u}_{h2}, \boldsymbol{v}_h \right) -O_T\left(\boldsymbol{w}_{h1}-\boldsymbol{w}_{h2} ;  \theta_{h2}, \varphi_h \right)  \\
            & + D(\psi_{h1} -\psi_{h2},\boldsymbol{v}_h)  \nonumber  \qquad  \forall(\boldsymbol{v}_h,\varphi_h,q_h) \in \mathrm{V}_h \times \mathrm{M}_h \times  \Pi_h.
  			\end{aligned}
            \end{equation}
            Given that $\boldsymbol{w}_{h1} \in \boldsymbol{K}_h$ and using \eqref{40}, \eqref{35}, and Theorem~\ref{qa4}, we can deduce:
  			$$
  			\begin{aligned}
  			& \frac{\hat{\alpha}}{2}\left\|(\boldsymbol{u}_{h1}, \theta_{h1})-(\boldsymbol{u}_{h2}, \theta_{h2})\right\|_{1,h} \\ 
            & \leq \sup _{\boldsymbol{0} \neq(\boldsymbol{v}_h,\varphi_h,p_h) \in \mathrm{V}_h \times \mathrm{M}_h \times  \Pi_h} \frac{\mathcal{A}_{\boldsymbol{w}_{h1}}\left[\left( \boldsymbol{u}_{h1}-\boldsymbol{u}_{h2}, \theta_{h1}-\theta_{h2}, p_{h1}-p_{h2}\right);( \boldsymbol{v}_h, \varphi_h,q_h)\right]}{\|( \boldsymbol{v}_h,\varphi_h,q_h)\|} \\
  			& =\sup _{\boldsymbol{0} \neq( \boldsymbol{v}_h,q_h) \in \mathrm{V}_h \times \Pi_h} \frac{-O_S^h\left(\boldsymbol{w}_{h1}-\boldsymbol{w}_{h2} ;  \boldsymbol{u}_{h2}, \boldsymbol{v}_h \right) -O_T\left(\boldsymbol{w}_{h1}-\boldsymbol{w}_{h2} ;  \theta_{h2}, \varphi_h \right) + D(\psi_{h1} -\psi_{h2},\boldsymbol{v}_h)}{\|(\boldsymbol{v}_h,\varphi_h,q_h)\|} \\
  			& \leq \left\|(\boldsymbol{w}_{h1},\psi_{h1})-(\boldsymbol{w}_{h2},\psi_{h2})\right\|_{1,h}\left\|(\boldsymbol{u}_{h2}, \theta_{h2})\right\|_{1,h},
  			\end{aligned}
  			$$
  			which, together with the fact that $\boldsymbol{u}_{h2} \in \boldsymbol{K}_h$, implies
            \begin{align*}
                \left\|(\boldsymbol{u}_{h1}, \theta_{h1})-(\boldsymbol{u}_{h2}, \theta_{h2})\right\|_{1,h} & \leq \frac{1}{\hat{\alpha}}C(\kappa, \nu, \boldsymbol{u}_{\star}, \theta_{\star}, \boldsymbol{f}, g, \alpha)\left\|(\boldsymbol{w}_{h1},\psi_{h1})-(\boldsymbol{w}_{h2},\psi_{h2})\right\|_{1,h} \\
               \left\|(\boldsymbol{u}_{h1}, \theta_{h1})-(\boldsymbol{u}_{h2}, \theta_{h2})\right\|_{1,h} & \leq \frac{\hat{\alpha}}{2}  \left\|(\boldsymbol{w}_{h1},\psi_{h1})-(\boldsymbol{w}_{h2},\psi_{h2})\right\|_{1,h}.
            \end{align*}
  			Assumption \eqref{55} directly implies that $\mathcal{J}_h$ is a contraction mapping.
  			Now, to establish the estimate \eqref{qa8}, let $(\boldsymbol{u}_h, \theta_h) \in \boldsymbol{K}_h$ be the unique fixed point of $\mathcal{J}_h$ and let $(\boldsymbol{u}_h, \theta_h) \in \mathrm{V}_h \times \mathrm{M}_h $ be the unique solution of \eqref{E} with $( \boldsymbol{u}_h,\theta_h,p_h) \in \mathrm{V}_h \times \mathrm{M}_h \times \Pi_h $. By the definition of $\boldsymbol{K}_h$, it is evident that $(\boldsymbol{u}_h, \theta_h)$ satisfies the following
  			$$
  			\|(\boldsymbol{u}_h, \theta_h)\|_{1,h} \leq {\hat{\alpha}}^{-1}( \|\boldsymbol{f}\|_{\mathrm{V}^{\prime}} + \|g\|_{\mathrm{M}^{\prime}}).
  			$$
  			Consequently, utilizing \eqref{35} on $\mathcal{A}_{\boldsymbol{u}_h}$, referring back to the definition of $\mathcal{A}_{\boldsymbol{u}_h}$ given in \eqref{A_wh}, and using the fact that $( \boldsymbol{u}_h,\theta_h,p_h)$ satisfies \eqref{E}, we obtain
            \begin{align*}
                \|p_h\|_{0,\Omega}\leq\|( \boldsymbol{u}_h,\theta_h,p_h)\| &\leq \frac{2}{\hat{\alpha}} \sup _{\boldsymbol{0} \neq( \boldsymbol{v}_h,\varphi_h,q_h) \in \mathrm{V}_h \times \mathrm{M}_h \times \Pi_h} \frac{\mathcal{A}_{\boldsymbol{u}_h}\left[(\boldsymbol{u}_h,\theta_h,p_h);(\boldsymbol{v}_h,\varphi_h,q_h)\right]}{\|( \boldsymbol{v}_h,\varphi_h,q_h)\|} \\ & =\frac{2}{\hat{\alpha}} \sup _{\boldsymbol{0} \neq( \boldsymbol{v}_h,\varphi_h,q_h) \in \mathrm{V}_h \times \mathrm{M}_h \times \Pi_h} \frac{\mathcal{F}(\boldsymbol{v}_h, \varphi_h, q_h) + D(\theta_h,\boldsymbol{v}_h)}{\|( \boldsymbol{v}_h,\varphi_h,p_h)\|},
            \end{align*}
            Now, we have 
            \begin{align*}
        \frac{\mathcal{F}(\boldsymbol{v}_h, \varphi_h, q_h) + D(\theta_h,\boldsymbol{v}_h)}{\|( \boldsymbol{v}_h,\varphi_h,p_h)\|} & =    \frac{D(\theta_h,\boldsymbol{v}_h)+F_v(\boldsymbol{v}_h) + F_q(q_h) + F_{\theta}(\varphi_h)}{\|( \boldsymbol{v}_h,\varphi_h,p_h)\|} \\ & \frac{\sum_{K\in \mathcal{K}_{h}} (\alpha \theta_h \boldsymbol{f}, \boldsymbol{v}_h)_K +\sum_{E\in \mathcal{E}_I}\bigg(-\int_{E} 2 \nu \varepsilon(\boldsymbol{v}_h) \cdot \boldsymbol{u}_{\star} d s + \gamma_N \int_{E} {h_e}^{-1}(\boldsymbol{u}_{\star} \cdot \boldsymbol{v}_h) d s }{\|( \boldsymbol{v}_h,\varphi_h,p_h)\|}
        \\ & \frac{ + \int_{E} q_h (\boldsymbol{n} \cdot \boldsymbol{u}_{\star}) d s +\int_{\Omega} g \varphi~ dx  -\int_{E} (\nabla \varphi_h \cdot \boldsymbol{n}) \theta_{\star} d s +\gamma_N \int_{E} {h_e}^{-1}(\theta_{\star} \varphi_h)d s \bigg)}{\|( \boldsymbol{v}_h,\varphi_h,p_h)\|} \\ & \leq C(\kappa, \nu, \boldsymbol{u}_{\star}, \theta_{\star}, \boldsymbol{f}, g, \alpha).
            \end{align*}
            Hence, $\|p\|_{0,\Omega} \lesssim C(\kappa, \nu, \boldsymbol{u}_{\star}, \theta_{\star}, \boldsymbol{f}, g, \alpha). $
  		\end{proof}
Let us denote $\mathrm{I}_h$ be the interpolation operator. Under usual assumptions, the following approximation property holds.
\begin{lemma}\label{error}
   Let $C_7>0$ and $C_{8}>0$ be constants, independent of $h$, such that for each $\boldsymbol{u} \in \textbf{\textit{H}}^{l+1}(K)$ with $0 \leq l \leq k$, there holds
$$
\begin{gathered}
\left\|\boldsymbol{u}-\mathrm{I}_h \boldsymbol{u} \right\|_{\textbf{\textit{L}}^2(K)} \leq C_7 \frac{h_K^{l+2}}{\rho_K}|\boldsymbol{u}|_{\textbf{\textit{H}}^{l+1}(K)} \leq \hat{C}_1 h_K^{l+1}|\boldsymbol{u}|_{\textbf{\textit{H}}^{l+1}(K)}, \\[4pt]
\left|\boldsymbol{u}-\mathrm{I}_h \boldsymbol{u} \right|_{\textbf{\textit{H}}^1(K)} \leq C_{8} \frac{h_K^{l+2}}{\rho_K^2}|\boldsymbol{u}|_{\textbf{\textit{H}}^{l+1}(K)} \leq \hat{C}_2 h_K^l|\boldsymbol{u}|_{\textbf{\textit{H}}^{l+1}(K)},
\end{gathered}
$$
where $h_K$ is the diameter of $K, \rho_K$ is the diameter of the largest sphere contained in $K$, and $k$ is the degree of the polynomial. The constants $\hat{C}_1$ and $\hat{C}_2$ are defined as
$$\hat{C}_1 = C_7\frac{ h_K}{\rho_K} \leq C_7 \hat{\sigma} \ \ \text{ and } \ \ \hat{C}_2 = C_{8} \frac{ h_K^2}{\rho_K^2} \leq C_{8} \hat{\sigma}^2. $$
Here $\hat{\sigma}$ denotes the shape regularity constant.
\end{lemma} 
\begin{proof}
    See \cite{MR2373954}.
\end{proof}
  		\section{A priori error bounds}\label{section4}
  		This section aims to establish the convergence of Nitsche's scheme \eqref{E} and determine the convergence rate. We derive the corresponding C\`{e}a's estimate. 
  		\begin{theorem}\label{RR1}
  			Assume that
  			\begin{align}\label{57}
  			\| \boldsymbol{u} \|_{1}    + \| \theta \|_{1} \leq \frac{\hat{\alpha}}{4},
  			\end{align}
  			with $\hat{\alpha}$ being the positive constant in Theorem~\ref{51}. Let $(\boldsymbol{u},\theta, p) \in$ $\mathrm{V} \times \mathrm{M} \times \Pi$ and $\left(\boldsymbol{u}_h,\theta_h,p_h \right) \in \mathrm{V}_h \times \mathrm{M}_h \times\mathrm{\Pi}_h$ be the unique solutions of problems \eqref{bb} and \eqref{E}, respectively. Then, we have 
  			\begin{align}\label{58}
  			\left\|\left( \boldsymbol{u}-\boldsymbol{u}_h, \theta - \theta_h, p - p_h \right)\right\| \lesssim  \inf _{\boldsymbol{0} \neq  \left(\boldsymbol{v}_h,\varphi_h,q_h \right) \in \mathrm{V}_h \times \mathrm{M}_h \times  \mathrm{\Pi}_h}\left\|\left(\boldsymbol{u}-\boldsymbol{v}_h, \varphi - \varphi_h, p-q_h \right)\right\|.
  			\end{align}	
  		\end{theorem} 
  		\begin{proof}
  			In order to simplify the subsequent analysis, we define $\boldsymbol{e}_{\mathrm{u}} = \boldsymbol{u}-\boldsymbol{u}_h, \boldsymbol{e}_{\theta} = \theta -\theta_h,$ and  $\boldsymbol{e}_p = p - p_h$, and for any $\left( \boldsymbol{z}_h, \phi_h, \zeta_h \right) \in \mathrm{V}_h \times \mathrm{M}_h \times \mathrm{\Pi}_h $, we write
  			\begin{align}\label{error_d}
  			\boldsymbol{e}_{\boldsymbol{u}}&=\xi_{\boldsymbol{u}}+\chi_{\boldsymbol{u}}=\left(\boldsymbol{u}-\boldsymbol{z}_h\right)+\left(\boldsymbol{z}_h-\boldsymbol{u}_h\right), \nonumber \\
            \boldsymbol{e}_{\theta}&=\xi_{\theta}+\chi_{\theta}=\left(\theta -\phi_h\right)+\left(\phi_h-\theta _h\right), 
            \\   \boldsymbol{e}_p&=\xi_p+\chi_p=\left(p-\zeta_h\right)+\left(\zeta_h-p_h\right) \nonumber .
  			\end{align}
  			By recalling the definition of the bilinear forms, we observe the validity of following identities:
  			$$
  			\mathcal{C}_h\left[( \boldsymbol{u},\theta,p);(\boldsymbol{v}_h,\varphi_h,q_h)\right] - D(\theta, \boldsymbol{v}_h)+O_h^S(\boldsymbol{u} ; \boldsymbol{u}, \boldsymbol{v}_h) +O_T( \boldsymbol{u}; \theta, \varphi_h)= \mathcal{F}(\boldsymbol{v}_h, \varphi_h, q_h),
  			$$
  			and
  			$$
  			\mathcal{C}_h\left[( \boldsymbol{u}_h,\theta_h,p_h);(\boldsymbol{v}_h,\varphi_h,q_h)\right] - D(\theta_h, \boldsymbol{v}_h)+O_h^S(\boldsymbol{u}_h ; \boldsymbol{u}_h, \boldsymbol{v}_h) +O_T( \boldsymbol{u}_h; \theta_h, \varphi_h)= \mathcal{F}(\boldsymbol{v}_h, \varphi_h, q_h), 
  			$$
  			for all $(\boldsymbol{v}_h,\varphi_h,q_h) \in \mathrm{V}_h \times \mathrm{M}_h \times \Pi_h$. Based on these observations, we deduce the Galerkin orthogonality property
     \begin{align}\label{galerkin_ortho}
\mathcal{C}_h\left[(\boldsymbol{e}_{\boldsymbol{u}}, \boldsymbol{e}_{\theta}, \boldsymbol{e}_p);( \boldsymbol{v}_h, \varphi_h, q_h)\right]  - D (\boldsymbol{e}_{\theta}, \boldsymbol{v}_h) + \left[O_S^h\left(\boldsymbol{u}; \boldsymbol{u}, \boldsymbol{v}_h \right) -O_
S^h \left(\boldsymbol{u}_h; \boldsymbol{u}_h, \boldsymbol{v}_h \right)\right] + \left[O_T\left(\boldsymbol{u}; \theta , \varphi_h \right) \nonumber \right. \\ \quad \left.  - O_T \left(\boldsymbol{u}_h; \theta_h, \varphi_h \right)\right] = 0,  \forall \left( \boldsymbol{v}_h, \varphi_h, q_h \right) \in \mathrm{V}_h \times \mathrm{M}_h \times {\Pi}_h.
\end{align}
     Now, using  \eqref{error_d}, and \eqref{galerkin_ortho}, we deduce that for all $\left( \boldsymbol{v}_h,\varphi_h, q_h \right) \in \mathrm{V}_h \times \mathrm{M}_h \times \mathrm{\Pi}_h $, the following relationship holds
     $$
     \begin{aligned}
& \mathcal{A}_{\boldsymbol{u}_{\mathrm{h}}}\left[\left( \chi_{\boldsymbol{u}},\chi_{\theta},\chi_p\right);\left(\boldsymbol{v}_h,\varphi_h,q_h \right)\right] - D(\chi_{\theta}, \boldsymbol{v}_h) \\ 
& \hspace{-4mm} = 
  \mathcal{C}_h \left[\left( \chi_{\boldsymbol{u}}, \chi_{\theta}, \chi_p \right), \left( \boldsymbol{v}_h, \varphi_h, q_h \right)\right] - D(\chi_{\theta}, \boldsymbol{v}_h) + O_S^h\left(\boldsymbol{u}_h; \chi_{\boldsymbol{u}}, \boldsymbol{v}_h\right) + O_T(\boldsymbol{u}_h; \chi_{\theta}, \varphi_h ) \\
  & \hspace{-4mm} = - \mathcal{C}_h\left[\left( \xi_{\boldsymbol{u}}, \xi_{\theta}, \xi_p \right); \left( \boldsymbol{v}_h, \varphi_h, q_h\right)\right] +
  \mathcal{C}_h\left[\left( \boldsymbol{e}_{\boldsymbol{u}}, \boldsymbol{e}_{\theta}, \boldsymbol{e}_p \right); \left( \boldsymbol{v}_h, \varphi_h, q_h\right)\right] + D(\xi_{\theta}, \boldsymbol{v}_h) - D(\boldsymbol{e}_{\theta}, \boldsymbol{v}_h) \\
  & + O_S^h\left(\boldsymbol{u}_h; \chi_{\boldsymbol{u}}, \boldsymbol{v}_h\right) + O_T(\boldsymbol{u}_h; \chi_{\theta}, \varphi_h ) \\
  & \hspace{-4mm} =- \mathcal{C}_h\left[\left( \xi_{\boldsymbol{u}}, \xi_{\theta}, \xi_p \right); \left( \boldsymbol{v}_h, \varphi_h,q_h\right)\right] + D(\xi_{\theta}, \boldsymbol{v}_h) - \left[O_S^h\left(\boldsymbol{u}; \boldsymbol{u}, \boldsymbol{v}_h \right) -O_
S^h \left(\boldsymbol{u}_h; \boldsymbol{u}_h, \boldsymbol{v}_h \right)\right] \\ 
& - \left[O_T\left(\boldsymbol{u}; \theta , \varphi_h \right) - O_T \left(\boldsymbol{u}_h; \theta_h, \varphi_h \right)\right] + O_S^h\left(\boldsymbol{u}_h; \chi_{\boldsymbol{u}}, \boldsymbol{v}_h\right) + O_T(\boldsymbol{u}_h; \chi_{\theta}, \varphi_h ) \\
  & \hspace{-4mm} =- \mathcal{C}_h\left[\left( \xi_{\boldsymbol{u}}, \xi_{\theta}, \xi_p \right); \left( \boldsymbol{v}_h, \varphi_h,q_h\right)\right] + D(\xi_{\theta}, \boldsymbol{v}_h) - \left[O_S^h\left(\boldsymbol{u} - \boldsymbol{u}_h; \boldsymbol{u}, \boldsymbol{v}_h \right) \right.  \\ & \quad \left. + O_S^h\left( \boldsymbol{u}_h; \boldsymbol{u}, \boldsymbol{v}_h\right) -O_
S^h \left(\boldsymbol{u}_h; \boldsymbol{u}_h, \boldsymbol{v}_h \right)\right]  - \left[O_T\left(\boldsymbol{u} - \boldsymbol{u}_h; \theta , \varphi_h \right) \right.  \\ & \quad \left. +O_T\left(\boldsymbol{u}_h; \theta, \varphi_h \right)  - O_T \left(\boldsymbol{u}_h; \theta_h, \varphi_h \right)\right] + O_S^h\left(\boldsymbol{u}_h; \chi_{\boldsymbol{u}}, \boldsymbol{v}_h\right) + O_T(\boldsymbol{u}_h; \chi_{\theta}, \varphi_h )  \\
  & \hspace{-4mm} =- \mathcal{C}_h\left[\left( \xi_{\boldsymbol{u}}, \xi_{\theta}, \xi_p \right); \left( \boldsymbol{v}_h, \varphi_h,q_h\right)\right] + D(\xi_{\theta}, \boldsymbol{v}_h) - O_S^h\left(\xi_{\boldsymbol{u}} + \chi_{\boldsymbol{u}}; \boldsymbol{u}, \boldsymbol{v}_h \right)  \\ & \quad - O_S^h\left( \boldsymbol{u}_h; \xi_{\boldsymbol{u}} + \chi_{\boldsymbol{u}}, \boldsymbol{v}_h\right) + O_S^h\left(\boldsymbol{u}_h; \chi_{\boldsymbol{u}}, \boldsymbol{v}_h\right) -  O_T\left(\xi_{\boldsymbol{u}} + \chi_{\boldsymbol{u}}; \theta , \varphi_h \right) \\& \quad - O_T\left(\boldsymbol{u}_h; \xi_{\theta} + \chi_{\theta}, \varphi_h \right) + O_T(\boldsymbol{u}_h; \chi_{\theta}, \varphi_h )  \\
  & \hspace{-4mm} =- \mathcal{C}_h\left[\left( \xi_{\boldsymbol{u}}, \xi_{\theta}, \xi_p \right); \left( \boldsymbol{v}_h, \varphi_h,q_h\right)\right] + D(\xi_{\theta}, \boldsymbol{v}_h) - O_S^h\left(\xi_{\boldsymbol{u}}; \boldsymbol{u}, \boldsymbol{v}_h \right) - O_S^h\left(\chi_{\boldsymbol{u}}; \boldsymbol{u}, \boldsymbol{v}_h \right) \\ & \quad   - O_S^h\left( \boldsymbol{u}_h; \xi_{\boldsymbol{u}}, \boldsymbol{v}_h\right) -  O_T\left( \xi_{\boldsymbol{u}}; \theta , \varphi_h \right) -  O_T\left( \chi_{\boldsymbol{u}}; \theta , \varphi_h \right)  - O_T\left(\boldsymbol{u}_h;  \xi_{\theta}, \varphi_h \right)
\end{aligned}
$$
 which, together with the definition of $\mathcal{C}_h$ given in equation \eqref{aa}, implies
\begin{align}\label{61}
  \mathcal{A}_{\boldsymbol{u}_{\mathrm{h}}}\left[\left( \chi_{\boldsymbol{u}}, \chi_{\theta}, \chi_p \right); \left( \boldsymbol{v}_h, \varphi_h, q_h \right)\right] - D(\chi_{\theta}, \boldsymbol{v}_h)  & = -{A}_S^h(\xi_{\boldsymbol{u}}, \boldsymbol{v}_h) - {B}^h(\boldsymbol{v}_h,\xi_{p}) - {B}^h(\xi_{\boldsymbol{u}}, q_h) -  {A}^h_T(\xi_{\theta},\varphi_h) + D(\xi_{\theta}, \boldsymbol{v}_h) \nonumber \\ & \quad  - O_S^h\left(\xi_{\boldsymbol{u}}; \boldsymbol{u}, \boldsymbol{v}_h \right) - O_S^h\left(\chi_{\boldsymbol{u}}; \boldsymbol{u}, \boldsymbol{v}_h \right) - O_S^h\left( \boldsymbol{u}_h; \xi_{\boldsymbol{u}}, \boldsymbol{v}_h\right) \nonumber  \\ & \quad -  O_T\left( \xi_{\boldsymbol{u}}; \theta , \varphi_h \right)   -  O_T\left( \chi_{\boldsymbol{u}}; \theta , \varphi_h \right) - O_T\left(\boldsymbol{u}_h;  \xi_{\theta}, \varphi_h \right), 
\end{align}
for all $\left( \boldsymbol{v}_h,\varphi_h,q_h\right) \in \mathrm{V}_h \times \mathrm{M}_h \times \mathrm{\Pi}_h$. Next, utilizing the discrete inf-sup condition \eqref{37} at the left hand side of \eqref{61}, and applying the continuity properties of $A_S^h, A_T^h, B^h, D, O_T, O_S^h$ stated in Theorem~\ref{qa4} to the right hand side of \eqref{61}, we can derive
   \begin{align*}
  & \left\|\chi_{\boldsymbol{u}}\right\|_{1,h} +   \left\|\chi_{\theta}\right\|_{1,h} + \left\|\chi_p\right\|_{0} \\
  & \hspace{-4mm} \lesssim \frac{2}{\hat{\alpha} \| (\boldsymbol{v}_h, \varphi_h, q_h) \|} \bigg( \|\xi_{\boldsymbol{u}}\|_{1,h} \|\boldsymbol{v}_h\|_{1,h} + \|\boldsymbol{v}_h\|_{1,h} \| \xi_p\|_{0} + \|\xi_{\boldsymbol{u}}\|_{1,h} \|q_h\|_{0} + \| \xi_{\theta}\|_{1,h} \| \varphi_h \|_{1,h} \\& \quad + \| \xi_{\theta} \|_{1,h}   \|\boldsymbol{v}_h\|_{1,h}  + \|\xi_{\boldsymbol{u}}\|_{1,h} \|\boldsymbol{v}_h\|_{1,h} \| \boldsymbol{u} \|_{1}  + \|\chi_{\boldsymbol{u}}\|_{1,h} \|\boldsymbol{v}_h\|_{1,h} \| \boldsymbol{u} \|_{1}  + \|\xi_{\boldsymbol{u}}\|_{1,h} \|\boldsymbol{v}_h\|_{1,h} \| \boldsymbol{u}_h \|_{1,h} \\
  &\quad  + \| \xi_{\boldsymbol{u}} \|_{1,h} \| \theta \|_{1} \| \varphi_h \|_{1,h} + \| \chi_{\boldsymbol{u}} \|_{1,h} \| \theta \|_{1} \| \varphi_h \|_{1,h} + \| \boldsymbol{u}_h \|_{1,h} \| \xi_{\theta} \|_{1,h} \| \varphi_h \|_{1,h} \bigg)  \\      
  & \hspace{-4mm} \lesssim \frac{2}{\hat{\alpha}} \bigg( \|\xi_{\boldsymbol{u}}\|_{1,h}  +  \| \xi_p\|_{0,\Omega} + \|\xi_{\boldsymbol{u}}\|_{1,h} + \| \xi_{\theta}\|_{1,h} + \| \xi_{\theta} \|_{1,h}  + \|\xi_{\boldsymbol{u}}\|_{1,h} \| \boldsymbol{u} \|_{1} + \|\chi_{\boldsymbol{u}}\|_{1,h} \| \boldsymbol{u} \|_{1}  \\
  &\quad  + \|\xi_{\boldsymbol{u}}\|_{1,h} \| \boldsymbol{u}_h \|_{1,h}  + \| \xi_{\boldsymbol{u}} \|_{1,h} \| \theta \|_{1,h}  + \| \chi_{\boldsymbol{u}} \|_{1,h} \| \theta \|_{1} + \| \boldsymbol{u}_h \|_{1} \| \xi_{\theta} \|_{1,h}  \bigg) \\
  & \hspace{-4mm} \lesssim \frac{2}{\hat{\alpha}} \bigg(  \| \xi_p\|_{0,\Omega} + \| \xi_{\theta} \|_{1,h} + \|\xi_{\boldsymbol{u}}\|_{1,h} \left( 1+ \| \boldsymbol{u} \|_{1} + \| \boldsymbol{u}_h \|_{1,h} +  \| \theta \|_{1,h} \right)  \\
  &\quad + \|\chi_{\boldsymbol{u}}\|_{1,h} \left(\| \boldsymbol{u} \|_{1}    + \| \theta \|_{1} \right) + \| \boldsymbol{u}_h \|_{1,h} \| \xi_{\theta} \|_{1,h}  \bigg),
\end{align*}
as well as
 \begin{align*}
& \left( 1- \frac{2}{\hat{\alpha}} \left(\| \boldsymbol{u} \|_{1}    + \| \theta \|_{1} \right) \right)  \left\|\chi_{\boldsymbol{u}}\right\|_{1,h} +  \left\|\chi_{\theta}\right\|_{1,h} + \left\|\chi_p\right\|_{0} \\
& \hspace{-4mm} \lesssim \frac{2}{\hat{\alpha}} \bigg(  \| \xi_p\|_{0,\Omega} + \| \xi_{\theta} \|_{1,h} +  \left( 1+ \| \boldsymbol{u} \|_{1} + \| \boldsymbol{u}_h \|_{1,h} +  \| \theta \|_{1,h} \right) \|\xi_{\boldsymbol{u}}\|_{1,h} + \| \boldsymbol{u}_h \|_{1,h} \| \xi_{\theta} \|_{1,h}  \bigg).
\end{align*}
  			Therefore, by taking into account the assumption \eqref{57} and the fact that $(\boldsymbol{u}_h, \theta_h) \in \boldsymbol{K}_h$, we can conclude
  			\begin{align}\label{63}
  			\left\|\chi_p\right\|_{0}+\left\|\chi_{\boldsymbol{u}}\right\|_{1,h} +\left\|\chi_{\theta}\right\|_{1,h} \lesssim \left(\left\|\xi_p\right\|_{0}+\left\|\xi_{\boldsymbol{u}}\right\|_{1,h} + \left\| \xi_{\theta}\right\|_{1,h}\right).
  			\end{align}
            In this way, from \eqref{error_d}, \eqref{63} and the triangle inequality we obtain
  			$$
  			\left\|\left(\boldsymbol{e}_p, \boldsymbol{e}_{\boldsymbol{u}},  \boldsymbol{e}_{\theta} \right)\right\| \leq\left\|\left(\chi_p, \chi_{\boldsymbol{u}}, \chi_{\theta}\right)\right\|+\left\|\left({\xi}_p, {\xi}_{\boldsymbol{u}}, \xi_{\theta}\right)\right\| \lesssim \left\|\left({\xi}_p, {\xi}_{\mathrm{u}}, \xi_{\theta}\right)\right\|.
  			$$
  		This, together with the fact that  $\left(\boldsymbol{z}_h,\zeta_h, \phi_h \right) \in \mathrm{V}_h \times \mathrm{\Pi}_h \times \mathrm{M}_h $ is arbitrary, concludes the proof.
  		\end{proof}
  		\begin{theorem}\label{101}
  			Let $(\boldsymbol{u},\theta, p) \in \mathrm{V} \times \mathrm{M} \times \Pi$ and $\left( \boldsymbol{u}_h, \theta_h, p_h\right) \in \mathrm{V}_h \times \mathrm{M}_h\times \mathrm{\Pi}_h$ denotes the unique solutions of the continuous problem \eqref{bb} and discrete problem \eqref{E}, respectively,  satisfying \eqref{57}. Suppose that
  			$(\boldsymbol{u},\theta,p) \in$ $  (\textbf{\textit{H}}^{l+1}(\Omega) \cap \mathrm{V}) \times ({H}^{l+1}(\Omega) \cap \mathrm{M}) \times ({H}^l(\Omega) \cap \Pi)$ with $l \geq 1$, Then we have 
  			$$
  			\left\|\left( \boldsymbol{u}-\boldsymbol{u}_h,\theta- \theta_h,p-p_h\right)\right\| \lesssim  h^{l}\left\{|\boldsymbol{u}|_{\textbf{\textit{H}}^{l+1}(\Omega)}+ |\theta|_{H^{l+1}(\Omega)}+|p|_{{H}^{l}(\Omega)}\right\}.
  			$$
  		\end{theorem}
  		\begin{proof}
  			The conclusion can be easily obtained by directly applying Theorem~\ref{RR1} and Lemma~\ref{error}.
  		\end{proof}
\section{A posteriori error estimation}\label{section5}
For each  $K \in \mathcal{K}_h$ and
each $E \in \mathcal{E}_h$, we define the element-wise and facet-wise residuals as follows:
\begin{align*}
    \mathbf{R}_{1,K} & = \{ \alpha \theta_h \boldsymbol{f}_h + \nu \Delta \boldsymbol{u}_h - (\boldsymbol{u}_h \cdot \nabla )\boldsymbol{u}_h - \nabla p_h \}|_K, \\
    \mathbf{R}_{2,K} & = \{ g_h + \kappa \Delta \theta_h + \boldsymbol{u}_h \cdot \nabla \theta_h \}|_K, \\
  \mathbf{R}_{1,E}  &=
\begin{cases}
\dfrac{1}{2}
\llbracket
\bigl(p_h \boldsymbol{I}
- 2 \nu \varepsilon(\boldsymbol{u}_h)\bigr)\boldsymbol{n}
\rrbracket_E,
& \text{for } E \in \mathcal{E}_h \setminus \mathcal{E}_\Gamma,
\end{cases} \\
\mathbf{R}_{2,E}  &=
\begin{cases}
\dfrac{1}{2}
\llbracket
-\kappa  \nabla \theta_h \cdot \boldsymbol{n}
\rrbracket_E,
& \text{for } E \in \mathcal{E}_h \setminus \mathcal{E}_\Gamma ,
\end{cases} \\
\boldsymbol{J}_{1,K} &=
\begin{cases}
\left(\bigl(p_h \boldsymbol{I}
- 2 \nu \varepsilon(\boldsymbol{u}_h)\bigr)\boldsymbol{n} \right)_E,
& \text{for } E \in \mathcal{E}_I,\\ 
\bigl( \left[ \mathbb{T}(\boldsymbol{u}_h,p_h) \boldsymbol{n} \right]_{\tau} + \gamma \boldsymbol{u}_h \bigr)_E, & \text{for } E \in \mathcal{E}_W, \\
\bigl(\mathbb{T}(\boldsymbol{u}_h,p_h) \boldsymbol{n}\bigr)_E, & \text{for } E \in  \mathcal{E}_O,
\end{cases} \\
\mathbf{J}_{2,K}  &=
\begin{cases}
\left(-\kappa  \nabla \theta_h \cdot \boldsymbol{n}\right)_E, & \text{for } E \in \mathcal{E}_I,  \\
\bigl( \kappa  \nabla \theta_h \cdot \boldsymbol{n}  + \beta \theta_h \bigr)_E, & \text{for } E \in \mathcal{E}_W, \\
\bigl( \kappa  \nabla \theta_h \cdot \boldsymbol{n} - (\boldsymbol{u}_h \cdot \boldsymbol{n}) \theta_h \psi(\boldsymbol{u}_h \cdot \boldsymbol{n}) \bigr)_E, & \text{for } E \in \mathcal{E}_O,
\end{cases}
\end{align*} 
where $\boldsymbol{f}_h$ and $g_h$ be a piecewise polynomial approximation of $\boldsymbol{f}$ and $g$. We then introduce the element-wise error estimators \( \Psi_K^2 = \Psi_{R_K}^2 + \Psi_{R_E}^2 + \Psi_{J_K}^2 \), with the contributions defined as:
\begin{align*}
    \Psi_{R_K}^2 &= h_K^2 \bigl(\| \boldsymbol{R}_{1,K} \|_{0,K}^2 + \| \boldsymbol{R}_{2,K} \|_{0,K}^2 \bigr), \\
    \Psi_{R_E}^2 &= \sum_{E \in \partial K } h_E  \bigl(  \| \boldsymbol{R}_{1,E} \|_{0,E}^2 + \| \boldsymbol{R}_{2,E} \|_{0,E}^2   \bigr), \\
    \Psi_{J_1}^2 &= \sum_{E \in \mathcal{E}_I \cup \mathcal{E}_W \cup  \mathcal{E}_O} h_E   \| \boldsymbol{J}_{1,K} \|_{0,E}^2  + \sum_{E \in \mathcal{E}_W} h_E^{-1} \| \boldsymbol{u}_h \cdot \boldsymbol{n} \|^2_{0,E} + \sum_{E \in \mathcal{E}_I} h_E^{-1} \| \boldsymbol{u}_h - \boldsymbol{u}^{\star} \|^2_{0,E}, \\
    \Psi_{J_2}^2 &= \sum_{E \in \mathcal{E}_I \cup \mathcal{E}_W \cup  \mathcal{E}_O} h_E  \| \boldsymbol{J}_{2,K} \|_{0,E}^2 + \sum_{E \in \mathcal{E}_I} h_E^{-1} \| \theta_h - \theta^{\star} \|^2_{0,E}, \\
    \Psi_{J_K}^2 &= \Psi_{J_1}^2 + \Psi_{J_2}^2.
\end{align*}
The global a posteriori error estimator for the nonlinear stationary problem is given by:
\begin{align}\label{total_estimate}
    \Psi := \left( \sum_{K \in \mathcal{K}_h} \Psi_K^2 \right)^{1/2}.
\end{align}
For a fixed $\boldsymbol{\tilde{u}} \in \boldsymbol{H}^1(\Omega)$, we define the bilinear form $\mathcal{A}_h^{\boldsymbol{\tilde{u}}} \left( \cdot ; \cdot \right)$ as 
\begin{align*}
\mathcal{A}_h^{\mathrm{NS},\boldsymbol{\tilde{u}}} \left( \boldsymbol{u}_h,  p_h, \theta_h ;  \boldsymbol{v}_h,  q_h, \varphi_h \right) & = {A}_S^h(\boldsymbol{u}_h, \boldsymbol{v}_h) + {O}_S^h(\boldsymbol{\tilde{u}} ; \boldsymbol{u}_h, \boldsymbol{v}_h) + {B}^h(\boldsymbol{v}_h,p_h) - D(\theta_h,\boldsymbol{v}_h) + {B}^h(\boldsymbol{u}_h, q_h) \\ & \quad + {A}^h_T(\theta_h,\varphi_h)   + O_T(\boldsymbol{\tilde{u}}; \theta_h, \varphi_h).
\end{align*}
\subsection{Reliability}
\begin{theorem}[Global inf-sup stability]\label{inf_sup_global}
    Let $\boldsymbol{\tilde{u}} \in \boldsymbol{H}^1(\Omega)$ satisfy $\| \boldsymbol{\tilde{u}} \|_{1,\Omega} < M$, for
a sufficiently small $M > 0$. For any $(\boldsymbol{u}, p, \theta) \in \mathrm{V} \times \Pi  \times \mathrm{M}$, there exists $(\boldsymbol{v}, q, \varphi) \in \mathrm{V} \times \Pi  \times \mathrm{M}$ with $\left|\!\left|\!\left| (\boldsymbol{v}, q, \varphi) \right|\!\right|\!\right| \leq 1$ such
that
\[
\mathcal{A}_h^{\tilde{\boldsymbol{u}}}\left(\boldsymbol{u},p, \theta ; \boldsymbol{v},q , \varphi \right) \geq C \left|\!\left|\!\left| (\boldsymbol{u}, p, \theta )\right|\!\right|\!\right|.
\]
\end{theorem}
\begin{proof}
The proof of the global inf-sup stability is already established in the intermediate steps of Theorem \ref{qa7}. Here, we state it as a separate theorem so that it can be easily used in the framework of a posteriori  analysis. 
\end{proof}

Since $\mathrm{V}_h$ and $\mathrm{M}_h$ are not conforming spaces, we define the conforming spaces $\mathrm{V}_h^c = \mathrm{V}_h \cap \mathrm{V}$ and $\mathrm{M}_h^c = \mathrm{M}_h \cap \mathrm{M}$. Finally, we decompose the approximate velocity and temperature  uniquely into $\boldsymbol{u}_h = \boldsymbol{u}_h^c + \boldsymbol{u}_h^r $ and $\theta_h = \theta_h^c + \theta_h^r $, where $(\boldsymbol{u}_h^c, \theta_h^c) \in \mathrm{V}_h^c \times \mathrm{M}_h^c  $ and $(\boldsymbol{u}_h^r, \theta_h^r)  \in (\mathrm{V}_h^c)^{\perp} \times (\mathrm{M}_h^c)^{\perp} $, and we note that $ \boldsymbol{u}_h^r = \boldsymbol{u}_h - \boldsymbol{u}_h^c \in \mathrm{V}_h $ and $\theta_h^r = \theta_h - \theta_h^c \in \mathrm{M}_h$.
 \begin{lemma}\label{SRL1}
     There holds 
     $$
   \| \boldsymbol{u}_h^r\|_{1,\mathcal{K}_h}+  \| \theta_h^r\|_{1,\mathcal{K}_h} \leq C_r \left( \sum_{K \in \mathcal{K}_h}  \Psi_{J_K}^2 \right)^{1/2}.
     $$
 \end{lemma}
 \begin{proof}
     It follows straightforwadly from the decomposition $ \boldsymbol{u}_h = \boldsymbol{u}_h^r + \boldsymbol{u}_h^c$ and $\theta_h = \theta_h^c + \theta_h^r$ and from the facet residual as given in \cite{ MR2532864}. 
 \end{proof}    
 \begin{lemma}\label{SRL2}
    If $\| \boldsymbol{u}\|_{1,\infty} < M, \| \theta \|_{1,\infty} < M$ for sufficiently small $M$, then the following estimate holds:
  \begin{align*}
  \frac{C}{2} \left|\!\left|\!\left| ( \boldsymbol{e}^{\boldsymbol{u}} , e^p,  \boldsymbol{e}^{\theta}) \right|\!\right|\!\right| & \leq \int_{\Omega} g ( \varphi - \varphi_h )   - \mathcal{A}_h^{\mathrm{NS},\boldsymbol{u}_h}(\boldsymbol{u}_h,p_h, \theta_h; \boldsymbol{v}-\boldsymbol{v}_h,q, \varphi - \varphi_h ) + (1+C) C_r \left( \sum_{K \in \mathcal{K}_h} \Psi_{J_K}^2  \right)^{1/2} \\ & \quad +  \sum_{E\in \mathcal{E}_I}\bigg(\int_{E} 2 \nu \varepsilon(\boldsymbol{v}_h) \cdot \boldsymbol{u}_{\star} d s - \gamma_N \int_{E} {h_e}^{-1}(\boldsymbol{u}_{\star} \cdot \boldsymbol{v}_h) d s  \bigg) \\ & \quad + \sum_{E\in \mathcal{E}_I}\bigg(\int_{E} (\nabla \varphi_h \cdot \boldsymbol{n}) \theta_{\star} d s - \gamma_N \int_{E} {h_e}^{-1}(\theta_{\star} \varphi_h)d s \bigg),
 \end{align*}
  where $\boldsymbol{e}^{\boldsymbol{u}} := \boldsymbol{u} - \boldsymbol{u}_h$, $\boldsymbol{e}^{\theta} := \theta - \theta_h$, $e^p:= p-p_h$ and $\boldsymbol{v}_h$ denotes the Cl\'{e}ment interpolation \cite{MR400739} of $\boldsymbol{v} \in \mathrm{V}$.
\end{lemma}
\begin{proof}
    Using $\boldsymbol{u}_h = \boldsymbol{u}_h^c + \boldsymbol{u}_h^r$, $\theta_h = \theta_h^c + \theta_h^r$, $\boldsymbol{e}_c^{\boldsymbol{u}} = \boldsymbol{u} - \boldsymbol{u}_h^c$, $\boldsymbol{e}_c^{\theta} = \theta - \theta_h^c$ and the triangle inequality implies 
   $$
     \left|\!\left|\!\left|( \boldsymbol{e}^{\boldsymbol{u}} , e^p, \boldsymbol{e}^{\theta}) \right|\!\right|\!\right| \leq \left|\!\left|\!\left| (\boldsymbol{e}_c^{\boldsymbol{u}} , e^p, \boldsymbol{e}_c^{\theta} )\right|\!\right|\!\right| + \|\boldsymbol{u}_h^r\|_{1,\mathcal{K}_h} + + \|\theta_h^r\|_{1,\mathcal{K}_h} \leq  \left|\!\left|\!\left| (\boldsymbol{e}_c^{\boldsymbol{u}}, e^p, \boldsymbol{e}_c^{\theta}) \right|\!\right|\!\right| + C_r \left( \sum_{K \in \mathcal{K}_h} \Psi_{J_K}^2 \right)^{1/2},  
   $$
   where $(\boldsymbol{e}_c^{\boldsymbol{u}}, e^p, \boldsymbol{e}_c^{\theta}) \in\mathrm{V}\times \mathrm{M} \times \Pi$. Then, Theorem \ref{inf_sup_global} gives 
   \begin{align*}
       C \left|\!\left|\!\left| (\boldsymbol{e}_c^{\boldsymbol{u}}, e^p, \boldsymbol{e}_c^{\theta}) \right|\!\right|\!\right| & \leq\mathcal{A}_h^{\mathrm{NS},\boldsymbol{u}_h} (\boldsymbol{e}_c^{\boldsymbol{u}}, e^p, \boldsymbol{e}_c^{\theta}; \boldsymbol{v},q, \varphi )  \\ & \leq  \mathcal{A}_h^{\mathrm{NS},\boldsymbol{u}_h} (\boldsymbol{e}^{\boldsymbol{u}}, e^p, \boldsymbol{e}^{\theta}; \boldsymbol{v},q, \varphi ) + \mathcal{A}_h^{\mathrm{NS},\boldsymbol{u}_h} (\boldsymbol{u}_h^r, 0, \theta_h^r; \boldsymbol{v},q, \varphi ) \\ & \leq \mathcal{A}_h^{\mathrm{NS},\boldsymbol{u}_h} (\boldsymbol{e}^{\boldsymbol{u}}, e^p, \boldsymbol{e}^{\theta}; \boldsymbol{v},q, \varphi) + C_r \left( \sum_{K \in \mathcal{K}_h} \Psi_{J_K}^2 \right)^{1/2}, 
   \end{align*}
   with $\left|\!\left|\!\left|(\boldsymbol{v},q, \varphi) \right|\!\right|\!\right| \leq   1 $. Owing to the relation 
   $$
  \mathcal{A}_h^{\mathrm{NS},\boldsymbol{u}_h}(\boldsymbol{u},p, \theta;\boldsymbol{v},q, \varphi) = \mathcal{A}_h^{\mathrm{NS},\boldsymbol{u}}(\boldsymbol{u},p, \theta;\boldsymbol{v},q,\varphi) - O_S^T(\boldsymbol{e}^{\boldsymbol{u}}; \boldsymbol{u}, \boldsymbol{v}) -O_T(\boldsymbol{e}^{\boldsymbol{u}}; \theta, \varphi)
   $$
   we then have 
   \begin{align*}
       C \left|\!\left|\!\left|( \boldsymbol{e}^{\boldsymbol{u}}, e^p, \boldsymbol{e}^{\theta}) \right|\!\right|\!\right| & \leq C \left|\!\left|\!\left|( \boldsymbol{e}_c^{\boldsymbol{u}}, e^p, \boldsymbol{e}^{\theta}_c ) \right|\!\right|\!\right| + C C_r \left( \sum_{K \in \mathcal{K}_h} \Psi_{J_K} ^2 \right)^{1/2} 
    \\   & \leq  \mathcal{A}_h^{\mathrm{NS},\boldsymbol{u}_h} (\boldsymbol{e}^{\boldsymbol{u}}, e^p, \boldsymbol{e}^{\theta}; \boldsymbol{v},q, \varphi)  +(1+C) C_r \left( \sum_{K \in \mathcal{K}_h} \Psi_{J_K}^2 \right)^{1/2}, \\& \leq \mathcal{A}_h^{\mathrm{NS},\boldsymbol{u}}(\boldsymbol{u},p, \theta;\boldsymbol{v},q,\varphi) - O_S^T(\boldsymbol{e}^{\boldsymbol{u}}; \boldsymbol{u}, \boldsymbol{v}) -O_T(\boldsymbol{e}^{\boldsymbol{u}}; \theta, \varphi) - \mathcal{A}_h^{\mathrm{NS},\boldsymbol{u}_h} (\boldsymbol{u}_h, p_h, {\theta}_h; \boldsymbol{v},q, \varphi) \\& \quad  +(1+C) C_r \left( \sum_{K \in \mathcal{K}_h} \Psi_{J_K}^2 \right)^{1/2},
   \end{align*}
   while using the properties $ O_S^T(\boldsymbol{e}^{\boldsymbol{u}}; \boldsymbol{u}, \boldsymbol{v}) \leq C_1 M \|\boldsymbol{e}^{\boldsymbol{u}}\|_{1, \mathcal{K}_h}$ and $ O_T(\boldsymbol{e}^{\boldsymbol{u}}; \theta, \varphi) \leq C_2 M \|\boldsymbol{e}^{\boldsymbol{u}}\|_{1, \mathcal{K}_h}$ yields the bounds
   \begin{align*}
        C \left|\!\left|\!\left|(\boldsymbol{e}^{\boldsymbol{u}}, e^p, \boldsymbol{e}^{\theta} )\right|\!\right|\!\right| & \leq \mathcal{A}_h^{\mathrm{NS},\boldsymbol{u}}(\boldsymbol{u},p, \theta;\boldsymbol{v},q,\varphi)  - \mathcal{A}_h^{\mathrm{NS},\boldsymbol{u}_h} (\boldsymbol{u}_h, p_h, {\theta}_h; \boldsymbol{v},q, \varphi)  +(1+C) C_r \left( \sum_{K \in \mathcal{K}_h} \Psi_{J_K}^2 \right)^{1/2} \\ & \quad - (C_2+C_3) M  \left|\!\left|\!\left|(\boldsymbol{e}^{\boldsymbol{u}}, e^p, \boldsymbol{e}^{\theta} )\right|\!\right|\!\right|.
   \end{align*}
   Moreover, from \eqref{bb} we have 
   \begin{align*}
C \left|\!\left|\!\left| ( \boldsymbol{e}^{\boldsymbol{u}}, e^p, \boldsymbol{e}^{\theta} ) \right|\!\right|\!\right| & \leq \int_{\Omega} g \varphi   - \mathcal{A}_h^{\mathrm{NS},\boldsymbol{u}_h} (\boldsymbol{u}_h, p_h, {\theta}_h; \boldsymbol{v},q, \varphi)  +(1+C) C_r \left( \sum_{K \in \mathcal{K}_h} \Psi_{J_K}^2 \right)^{1/2}.
\end{align*}
By adding and subtracting \( \mathcal{A}_h^{\mathrm{NS},\boldsymbol{u}_h}\left(\boldsymbol{u}_h, p_h, \theta_h; \boldsymbol{v}_h, 0, \varphi_h \right) \), respectively, and utilizing the definition of the discrete weak formulation \eqref{E}, we obtain the desired result.
\end{proof}   
\begin{lemma}\label{SRL3}
     For $\left( \boldsymbol{v},q, \varphi \right) \in \mathrm{V} \times \Pi \times \mathrm{M} $, there is $(\boldsymbol{u}_h, p_h, \theta_h) \in \mathrm{V}_h \times \Pi_h \times \mathrm{M}_h$ such that 
     \begin{align*}
    &  \int_{\Omega} g ( \varphi - \varphi_h )   - \mathcal{A}_h^{\mathrm{NS},\boldsymbol{u}_h}(\boldsymbol{u}_h,p_h, \theta_h; \boldsymbol{v}-\boldsymbol{v}_h,q, \varphi - \varphi_h ) + (1+C) C_r \left( \sum_{K \in \mathcal{K}_h} \Psi_{J_K}^2  \right)^{1/2} \\ & +  \sum_{E\in \mathcal{E}_I}\bigg(\int_{E} 2 \nu \varepsilon(\boldsymbol{v}_h) \cdot \boldsymbol{u}_{\star} d s - \gamma_N \int_{E} {h_e}^{-1}(\boldsymbol{u}_{\star} \cdot \boldsymbol{v}_h) d s  \bigg) + \sum_{E\in \mathcal{E}_I}\bigg( \int_{E} (\nabla \varphi_h \cdot \boldsymbol{n}) \theta_{\star} d s - \gamma_N \int_{E} {h_e}^{-1}(\theta_{\star} \varphi_h)d s \bigg) \\
    & \leq C \left(\Psi + \| \theta_h ( \boldsymbol{f} - \boldsymbol{f}_h) \|_{0,\Omega} + \| g -g_h\|_{0,\Omega} \right) \left|\!\left|\!\left| (\boldsymbol{v}, q, \varphi)\right|\!\right|\!\right|.
     \end{align*}
\end{lemma}
\begin{proof}
    Using integration by parts element-wise gives
     \begin{align*}
    &  \int_{\Omega} g ( \varphi - \varphi_h )   - \mathcal{A}_h^{\mathrm{NS},\boldsymbol{u}_h}(\boldsymbol{u}_h,p_h, \theta_h; \boldsymbol{v}-\boldsymbol{v}_h,q, \varphi - \varphi_h ) + (1+C) C_r \left( \sum_{K \in \mathcal{K}_h} \Psi_{J_K}^2  \right)^{1/2} \\ & \quad - \left[   \sum_{E\in \mathcal{E}_I}\bigg(-\int_{E} 2 \nu \varepsilon(\boldsymbol{v}_h) \cdot \boldsymbol{u}_{\star} d s + \gamma_N \int_{E} {h_e}^{-1}(\boldsymbol{u}_{\star} \cdot \boldsymbol{v}_h) d s  \bigg) + \sum_{E\in \mathcal{E}_I}\bigg( -\int_{E} (\nabla \varphi_h \cdot \boldsymbol{n}) \theta_{\star} d s \right. \\ & \quad \left. +\gamma_N \int_{E} {h_e}^{-1}(\theta_{\star} \varphi_h)d s \bigg) \right] = T_1 +T_2 +T_3 +T_4 +T_5, 
     \end{align*}
     where we define the terms 
     \begin{align*}
         T_1 & \coloneqq \sum_{K \in \mathcal{K}_h} \int_K  \left( \alpha \theta_h \boldsymbol{f}_h + 2 \nu \nabla \cdot \varepsilon (\boldsymbol{u}_h) - \boldsymbol{u}_h \cdot \nabla \boldsymbol{u}_h - \nabla p_h \right) \left( \boldsymbol{v} - \boldsymbol{v}_h\right) + \sum_{K \in \mathcal{K}_h} \int_K \alpha \theta_h ( \boldsymbol{f} - \boldsymbol{f}_h) \boldsymbol{v}, \\
         T_2 & \coloneqq \sum_{K \in \mathcal{K}_h} \int_K (\nabla \cdot \boldsymbol{u}_h) q, \\
          T_3 & \coloneqq \sum_{K \in \mathcal{K}_h} \int_{\partial K_{\text{in}} } \bigg((p_h \boldsymbol{I} - 2 \nu \varepsilon (\boldsymbol{u}_h))   \cdot \boldsymbol{n} \bigg) (\boldsymbol{v}-\boldsymbol{v}_h) +  \sum_{E \in \Gamma_W} \left[\int_E \boldsymbol{n}^t \left( 2 \nu \varepsilon (\boldsymbol{v}-\boldsymbol{v}_h) - q I\right) \boldsymbol{n} (\boldsymbol{n} \cdot \boldsymbol{u}_h)  \, ds  \right.  \\& \quad  \left. - \int_E \{  \left[ \mathbb{T}(\boldsymbol{u}_h,p_h) \boldsymbol{n} \right]_{\tau} + \gamma \boldsymbol{u}_h \} (\boldsymbol{v}- \boldsymbol{v}_h)  \, ds - \frac{\gamma_N }{h_E} \int_E (\boldsymbol{u}_h \cdot \boldsymbol{n})(\boldsymbol{v}-\boldsymbol{v}_h) \cdot \boldsymbol{n}  \, ds \right]  -
         \sum_{E \in \Gamma_O} \int_E (\mathbb{T}(\boldsymbol{u}_h,p_h) \boldsymbol{n}) (\boldsymbol{v}- \boldsymbol{v}_h) \\ & \quad 
        + \sum_{E \in \Gamma_I} \left[  \int_E 2 \nu \varepsilon (\boldsymbol{u}_h) \boldsymbol{n} ( \boldsymbol{v}- \boldsymbol{v}_h) + \int_E 2 \nu \varepsilon (\boldsymbol{v} - \boldsymbol{v}_h) \boldsymbol{n} (\boldsymbol{u}_h -\boldsymbol{u}_{\star}) - \int_E p_h \left(\boldsymbol{n} \cdot  (\boldsymbol{v} -\boldsymbol{v}_h) \right) - \int_E q (\boldsymbol{n} \cdot \boldsymbol{u}_h) \right.  \\ & \left. \quad -  \frac{\gamma_N }{h_E} \int_E (\boldsymbol{u}_h - \boldsymbol{u}_{\star}) (\boldsymbol{v} - \boldsymbol{v}_h)   \right], \\ 
         T_4 & \coloneqq \sum_{K \in \mathcal{K}_h} \int_K  \left( g_h + \kappa \Delta \theta_h - \boldsymbol{u}_h \cdot \nabla \theta_h  \right) \left( \varphi - \varphi_h \right) + \sum_{K \in \mathcal{K}_h} \int_K (g-g_h) \varphi, \\
         T_5 & \coloneqq \sum_{K \in \mathcal{K}_h} \int_{\partial K_{\text{in}} } - \kappa \nabla \theta_h \boldsymbol{n} (\varphi - \varphi_h)  -   \sum_{E \in \Gamma_W}  \int_E \left( \kappa  \nabla \theta_h \cdot \boldsymbol{n}  + \beta \theta_h \right) (\varphi - \varphi_h) \\& \quad  -   \sum_{E \in \Gamma_O}  \int_E \left(  \kappa  \nabla \theta_h \cdot \boldsymbol{n} - (\boldsymbol{u}_h \cdot \boldsymbol{n}) \theta_h \psi(\boldsymbol{u}_h \cdot \boldsymbol{n}) \right) (\varphi - \varphi_h) \\
         & \quad + \sum_{E \in \Gamma_I }  \bigl[ \int_E \kappa \nabla \theta_h \boldsymbol{n}  (\varphi -\varphi_h) + \int_E \kappa \nabla (\varphi -\varphi_h) \boldsymbol{n} (\theta_h -\theta_{\star}) - \frac{\gamma_N }{h_E} \int_E (\theta_h - \theta_{\star}) (\varphi -\varphi_h)  \bigr].
     \end{align*}
Applying the Cauchy-Schwarz inequality and Cl\'{e}ment interpolation estimates \cite{MR400739} to $T_1$ implies
\begin{align*}
    T_1 & \leq \left( \sum_{K \in \mathcal{K}_h} h_K^2 \| \boldsymbol{R}_{1,K} \|_{0,K}^2 \right)^{1/2} \left( \sum_{K \in \mathcal{K}_h} h_K^{-2} \|\boldsymbol{v} - \boldsymbol{v}_h \|_{0,K}^2 \right)^{1/2} + \sum_{K \in \mathcal{K}_h} \int_K \alpha \theta_h ( \boldsymbol{f} - \boldsymbol{f}_h) \boldsymbol{v} \\ &  \leq C \left( \sum_{K \in \mathcal{K}_h} h_K^2\| \boldsymbol{R}_{1,K} \|_{0,K}^2 \right)^{1/2} \| \nabla \boldsymbol{v}\|_{0,\Omega} + \| \theta_h ( \boldsymbol{f} - \boldsymbol{f}_h) \|_{0,\Omega} \| \boldsymbol{v} \|_{0,\Omega}.
\end{align*}
The bound for $T_2$ is defined as 
$ T_2 \leq \| \nabla \cdot \boldsymbol{u}_h\|_{0,\Omega} \| q\|_{0,\Omega}$.

Next, we rewrite \( T_3 \) in terms of a sum over interior facets and then apply the Cauchy-Schwarz and trace inequalities. This entails
\begin{align*}
    T_3 
    &\leq \left( \sum_{E \in \mathcal{E}_{\Omega}}  h_E \|\boldsymbol{R}_E\|^2_{0,E}\right)^{1/2} \left( \sum_{E \in \mathcal{E}_{\Omega}} \frac{1}{h_E}  \|\boldsymbol{v}-\boldsymbol{v}_h\|_{0,E}^2 \right)^{1/2} + \left(\sum_{K \in \mathcal{K}_h} h_K^{-2} \|\boldsymbol{v} - \boldsymbol{v}_h \|^2_{0,K} \right)^{1/2} \\
    &\quad   \left( \sum_{E \in \mathcal{E}_{W}} \frac{1}{h_E} \| \boldsymbol{u}_n \cdot \boldsymbol{n} \|^2_{0,E} \right)^{1/2}  + \| \boldsymbol{u}_h \cdot \boldsymbol{n} \|_{0,\Gamma_W} \|q\|_{0,\Gamma_W}  + \left( \sum_{E \in \mathcal{E}_W} h_E \| \boldsymbol{R}_{1,E} \|_{0,E}^2 \right)^{1/2} \\
    &\quad  \left( \sum_{E \in \mathcal{E}_W} \frac{1}{h_E}  \|\boldsymbol{v} - \boldsymbol{v}_h \|_{0,E}^2\right)^{1/2}  + \left(\frac{\gamma_N^2 } {h_E} \sum_{E \in \mathcal{E}_W}  \|\boldsymbol{u}_h \cdot \boldsymbol{n} \|^2_{0,E} \right)^{1/2}  \left(\sum_{E \in \mathcal{E}_W}  \frac{1}{h_E}  \|\boldsymbol{v} - \boldsymbol{v}_h \|_{0,E}^2 \right)^{1/2} \\ & \quad + \left( \sum_{E \in \mathcal{E}_O} h_E \| \boldsymbol{R}_{1,E} \|_{0,E}^2 \right)^{1/2}  \left( \sum_{E \in \mathcal{E}_O} \frac{1}{h_E}  \|\boldsymbol{v} - \boldsymbol{v}_h \|_{0,E}^2\right)^{1/2} +  \left( \sum_{E \in \mathcal{E}_I} h_E \| \boldsymbol{R}_{1,E} \|_{0,E}^2 \right)^{1/2} \\& \quad  \left( \sum_{E \in \mathcal{E}_I} \frac{1}{h_E}  \|\boldsymbol{v} - \boldsymbol{v}_h \|_{0,E}^2\right)^{1/2} + \left(\sum_{K \in \mathcal{K}_h} h_K^{-2} \|\boldsymbol{v} - \boldsymbol{v}_h \|^2_{0,K} \right)^{1/2}    \left( \sum_{E \in \mathcal{E}_{I}} \frac{1}{h_E} \| \boldsymbol{u}_h - \boldsymbol{u}_{\star} \|^2_{0,E} \right)^{1/2} \\& \quad 
    + \| \boldsymbol{u}_h \cdot \boldsymbol{n} \|_{0,\Gamma_I} \|q\|_{0,\Gamma_I} + \left(\frac{\gamma_N^2 } {h_E} \sum_{E \in \mathcal{E}_I}  \|\boldsymbol{u}_h -  \boldsymbol{u}_{\star} \|^2_{0,E} \right)^{1/2}  \left(\sum_{E \in \mathcal{E}_I}  \frac{1}{h_E}  \|\boldsymbol{v} - \boldsymbol{v}_h \|_{0,E}^2 \right)^{1/2}, 
    \\ & \leq \left(\sum_{E \in \partial K}  h_E \|\boldsymbol{R}_{1,E}\|^2_{0,E} +  \Psi_{J_K}^2 \right)^{1/2}  \left( \| \nabla \boldsymbol{v}\|_{0,\Omega}^2 + \|q\|^2_{0,\Omega}  \right)^{1/2}.
\end{align*}
Proceeding in a similar fashion, we are able to establish the following bounds for $T_4$ and $T_5$:
\begin{align*}
    T_4 & \leq  C \left( \sum_{K \in \mathcal{K}_h} h_K^2\| \boldsymbol{R}_{2,K} \|_{0,K}^2 \right)^{1/2} \| \nabla \varphi\|_{0,\Omega} + + \| g -g_h\|_{0,\Omega} \|\varphi\|_{0,\Omega}, \\
    T_5  & \leq  \left(\sum_{E \in \partial K}  h_E \|\boldsymbol{R}_{2,E}\|^2_{0,E} +  \Psi_{J_K}^2 \right)^{1/2}  \| \nabla \varphi \|_{0,\Omega}.
\end{align*}
\end{proof}
 \begin{theorem}
Let $(\boldsymbol{u}, p, \theta)$ be the solution to \eqref{bb} and $(\boldsymbol{u}_h, p_h, \theta_h)$ the solution to \eqref{E}. Let $\Psi$ be the a posteriori error estimator defined in \eqref{total_estimate}. If $\| \boldsymbol{u}\|_{1,\infty} < M, \| \theta \|_{1,\infty} < M$ sufficiently small $M$, then the following estimate holds:
\begin{align*}
    \left|\!\left|\!\left| (\boldsymbol{u}-\boldsymbol{u}_h , p -p_h, \theta - \theta_h)\right|\!\right|\!\right| \leq C \left(\Psi + \| \theta_h ( \boldsymbol{f} - \boldsymbol{f}_h) \|_{0,\Omega} + \| g -g_h\|_{0,\Omega} \right),
\end{align*}
where $C > 0$ is a constant independent of $h$.
 \end{theorem}
\begin{proof}
    It suffices to apply Lemmas \ref{SRL2} and \ref{SRL3}.
\end{proof}
\subsection{Efficiency}
For each $K \in \mathcal{K}_h$, we can define the standard polynomial bubble function $b_K$. Then, for any polynomial function $\boldsymbol{v}$ on $K$, the following estimates hold:
\begin{align}\label{interior_bubb}
\| b_K \boldsymbol{v} \|_{0,K} \le C \| \boldsymbol{v} \|_{0,K}, 
\qquad
\| \boldsymbol{v} \|_{0,K} \le C \| b_K^{1/2} \boldsymbol{v} \|_{0,K},
\qquad
\| \nabla (b_K \boldsymbol{v}) \|_{0,K} \le C h_K^{-1} \| \boldsymbol{v} \|_{0,K},
\end{align}
where $C$ is a positive constant, independent of $K$ and $\boldsymbol{v}$; see \cite{MR3059294}.
\begin{lemma}\label{efficiency_1}
Let $K$ be an element of $\mathcal{K}_h$.   The local equilibrium residual satisfies
\begin{align*}
   h_K \boldsymbol{R}_{1,K} &\leq C \left( \| \boldsymbol{u} - \boldsymbol{u}_h \|_{1,K} + \|p-p_h\|_{0,K} + \| \theta - \theta_h \|_{1,K} + h_K \theta_h \| \boldsymbol{f}- \boldsymbol{f}_h \|_{0,K} \right), \\
     h_K \boldsymbol{R}_{2,K} &\leq C \left( \| \boldsymbol{u} - \boldsymbol{u}_h \|_{1,K} + \|\theta -\theta_h\|_{1,K} + h_K \| g - g_h \|_{0,K} \right).
\end{align*}
Moreover, it also follows that
$$
\Psi_{R_K} \leq \left|\!\left|\!\left| (\boldsymbol{u}-\boldsymbol{u}_h , p -p_h, \theta - \theta_h)\right|\!\right|\!\right|.
$$
\end{lemma}
\begin{proof}
    For each $K \in \mathcal{K}_h$, we define $\boldsymbol{W}_b = b_K \boldsymbol{R}_{1,K}$. Then, using \eqref{interior_bubb}, we have 
    \begin{align*}
    \frac{1}{C^2} \| \boldsymbol{R}_{1,K} \|_{0,K}^2 &\leq \| b_K^{1/2} \boldsymbol{R}_{1,K} \|^2_{0,K} = \int_K \boldsymbol{R}_{1,K} \cdot  \boldsymbol{W}_b  \\ & 
    = \int_K \left( \theta_h \boldsymbol{f}_h + 2 \nu \nabla \cdot \varepsilon (\boldsymbol{u}_h) - \boldsymbol{u}_h \cdot \nabla \boldsymbol{u}_h - \nabla p_h  \right) \cdot \boldsymbol{W}_b.
    \end{align*}
Noting that $\left( \theta \boldsymbol{f} + 2 \nu \nabla \cdot \varepsilon ( \boldsymbol{u}) - \boldsymbol{u} \cdot \nabla \boldsymbol{u} - \nabla p \right)|_K = 0$ for the exact solution $(\boldsymbol{u}, p)$, we simply subtract, then integrate by parts and note that $\boldsymbol{W}_b|_{\partial K} = \boldsymbol{0}$, to give
\begin{align*}
 \frac{1}{C^2} \| \boldsymbol{R}_{1,K} \|_{0,K}^2 &\leq \int_K \theta_h \left( \boldsymbol{f}_h - \boldsymbol{f} \right) \cdot \boldsymbol{W}_b \,  
    + \int_K \left[ -2 \nu \nabla \cdot \varepsilon  (\boldsymbol{u} - \boldsymbol{u}_h) + \nabla (p - p_h) \right] \cdot \boldsymbol{W}_b \,   \\ & \quad +  \int_K \left[  \left( \boldsymbol{u} - \boldsymbol{u}_h \right) \cdot \nabla  \boldsymbol{u} +  \boldsymbol{u}_h \cdot \nabla  \left( \boldsymbol{u} - \boldsymbol{u}_h \right) \right] \cdot \boldsymbol{W}_b - \int_K \boldsymbol{f} (\theta - \theta_h) \cdot \boldsymbol{W}_b \leq   T_1 + T_2,
\end{align*}
where 
\begin{align*}
      T_1 &= \int_K \bigg( 2 \nu \varepsilon (\boldsymbol{u}- \boldsymbol{u}_h ) :  \varepsilon ({W}_b) - (p-p_h) \nabla \cdot \boldsymbol{W}_b \bigg)  + \int_K  \theta_h ( \boldsymbol{f}_h - \boldsymbol{f} ) \cdot \boldsymbol{W}_b \\& \quad - \int_K \boldsymbol{f} (\theta - \theta_h) \cdot \boldsymbol{W}_b, \\
      T_2 &=  \int_K \left[ \left( \boldsymbol{u} - \boldsymbol{u}_h \right) \cdot \nabla  \boldsymbol{u} +  \boldsymbol{u}_h \cdot \nabla  \left( \boldsymbol{u} - \boldsymbol{u}_h \right) \right] \cdot \boldsymbol{W}_b .
\end{align*}
Using the Cauchy-Schwarz inequality and \eqref{interior_bubb}, we have 
\begin{align*}
T_1 &\leq  C_1 \left( \| \boldsymbol{u} - \boldsymbol{u}_h \|_{1,K} + \|p-p_h\|_{0,K} + \|\theta -\theta_h\|_{1,K} + h_K \theta_h \|\boldsymbol{f}- \boldsymbol{f}_h \|_{0,K} \right) h_K^{-1} \| \boldsymbol{R}_{1,K}\|_{0,K} \\&
\leq  C_1 \left( \| \boldsymbol{u} - \boldsymbol{u}_h \|_{1,K}^2 + \|p-p_h\|_{0,K}^2 + \|\theta -\theta_h\|_{1,K}^2 + h_K^2 \|\boldsymbol{f}- \boldsymbol{f}_h \|_{0,K}^2 \right)^{1/2} \left(h_K^{-2} \| \boldsymbol{R}_{1,K}\|_{0,K}^2\right)^{1/2}, \\[4pt]
 T_2 & \leq C_2 \| \boldsymbol{u} - \boldsymbol{u}_h \|_{1,K} h_K^{-1} \|\boldsymbol{R}_{1,K} \|_{0,K},
\end{align*}
and combining these bounds leads to the first stated result. The other two bounds follow similarly. 
\end{proof}
Let $E$ denote an interior facet that is shared by two elements $K$ and $K^{\prime}$. Let $\omega_E$ be the patch corresponding to the
union between $K$ and $K^{\prime}$. Next, we define the facet bubble function $\zeta_E$ on $E$ with the property that it is positive
in the interior of the patch $\omega_E$ and zero on the boundary of the patch. From Verfürth \cite{MR3059294}, the following
results hold:
\begin{align}\label{edge_bubb}
\|q\|_{0,E} &\le C \|\zeta_E^{1/2} q\|_{0,E}, \quad
\|\zeta_E q\|_{0,K} &\le C h_e^{1/2} \|q\|_{0,E}, 
\quad \|\nabla(\zeta_E q)\|_{0,K} \le C h_e^{-1/2} \|q\|_{0,E},
\quad \text{for all } K \in \omega_E. 
\end{align}
\begin{lemma}\label{efficiency_2}
    Let $K$ be an element of $\mathcal{K}_h$. The jump residual satisfies 
\begin{align*}
   h_E^{1/2} \boldsymbol{R}_{1,E} &\leq C \sum_{K \in \omega_E} \left( \| \boldsymbol{u} - \boldsymbol{u}_h \|_{1,K} + \|p-p_h\|_{0,K} + \| \theta - \theta_h \|_{1,K} + h_K \theta_h \| \boldsymbol{f}- \boldsymbol{f}_h \|_{0,K} \right), \\
     h_E^{1/2} \boldsymbol{R}_{2,E} &\leq C \sum_{K \in \omega_E}\left( \| \boldsymbol{u} - \boldsymbol{u}_h \|_{1,K} + \|\theta -\theta_h\|_{1,K} + h_K \| g - g_h \|_{0,K} \right).
\end{align*}
    Moreover, we also have
    $$
\Psi_{R_E} \leq \sum_{E \in \partial K} \sum_{K \in \omega_E} \left(\left|\!\left|\!\left| (\boldsymbol{u}-\boldsymbol{u}_h , p -p_h, \theta - \theta_h)\right|\!\right|\!\right| + h_K \| \theta_h (\boldsymbol{f} - \boldsymbol{f}_h ) \|_{0,K} +h_K \|g-g_h\|_{0,K}  \right).
    $$
\end{lemma}
\begin{proof}
   Suppose $E$ is an interior facet (edge) and recall that the classical solution $(\boldsymbol{u},p)$ satisfies \\ $ \llbracket \left(p \boldsymbol{I} - 2 \nu \varepsilon (\boldsymbol{u}) \right) \boldsymbol{n}\rrbracket |_E =0. $ Now, define the localised jump term $\boldsymbol{W}_E =
   \sum_{E \in \partial K} \frac{h_E}{2 \nu} \boldsymbol{R}_E \zeta_E $. Using \eqref{interior_bubb} gives 
 \begin{align*}
    \frac{h_E}{\nu} \|\boldsymbol{R}_E \|_{0,E}^2 &\leq C \left( \llbracket p_h \boldsymbol{I} - 2 \nu \varepsilon (\boldsymbol{u}_h )\rrbracket , \boldsymbol{W}_E \right)_E
    \\ & \leq C \left( \llbracket \left(p_h \boldsymbol{I} - 2 \nu \varepsilon (\boldsymbol{u}_h) \right) \boldsymbol{n}\rrbracket - \llbracket \left(p \boldsymbol{I} - 2 \nu \varepsilon (\boldsymbol{u}) \right) \boldsymbol{n}\rrbracket , \boldsymbol{W}_E \right)_E.
\end{align*}
Using integration by parts on each element of the patch $\omega_E$ implies
\begin{align}\label{jump_bound}
\left( \llbracket \left(p_h \boldsymbol{I} - 2 \nu \varepsilon (\boldsymbol{u}_h) \right) \boldsymbol{n}\rrbracket - \llbracket \left(p \boldsymbol{I} - 2 \nu \varepsilon (\boldsymbol{u}) \right) \boldsymbol{n}\rrbracket , \boldsymbol{W}_E \right)_E &= \sum_{K \in \omega_E} \left\{ \int_K \left[ - 2 \nu \nabla \cdot \varepsilon (\boldsymbol{u}- \boldsymbol{u}_h) + \nabla (p-p_h) \right] \cdot \boldsymbol{W}_E \right. \nonumber  \\
& \quad + \left. \int_K \left[ -2 \nu \varepsilon (\boldsymbol{u} - \boldsymbol{u}_h ) + (p-p_h) \boldsymbol{I} \right] : \nabla \boldsymbol{W}_E \right\}.
\end{align}
Since the exact solution $(\boldsymbol{u},p)$ satisfies $-2 \nu \nabla \cdot \varepsilon ( \boldsymbol{u}) + \boldsymbol{u} \cdot \nabla \boldsymbol{u} +  \nabla p |_K = \theta \boldsymbol{f}|_K$, we have 
\begin{align}\label{bound_eq}
     \frac{h_E}{\nu} \|\boldsymbol{R}_E \|_{0,E}^2 & =  \sum_{K \in \omega_E} \int_K \{ \theta_h \boldsymbol{f}_h + 2 \nu \nabla \cdot \varepsilon (\boldsymbol{u}_h) - \boldsymbol{u}_h \cdot \nabla \boldsymbol{u}_h - \nabla p_h \} \cdot \boldsymbol{W}_E + \sum_{K \in \omega_E} \int_K \theta_h (\boldsymbol{f} - \boldsymbol{f}_h) \cdot \boldsymbol{W}_E \nonumber  \\ & \quad + \sum_{K \in \omega_E} \int_K \boldsymbol{f} (\theta -\theta_h) \cdot \boldsymbol{W}_E + \sum_{K \in \omega_E} \int_K \{ - 2 \nu \varepsilon (  \boldsymbol{u} - \boldsymbol{u}_h) +(p-p_h) \boldsymbol{I} \} : \nabla \boldsymbol{W}_E \nonumber  \\ & \quad + \sum_{K \in \omega_E} \int_K \{ \boldsymbol{u}_h \cdot \nabla \boldsymbol{u}_h - \boldsymbol{u} \cdot \nabla \boldsymbol{u}\} \cdot \boldsymbol{W}_E =
     T_1 + T_2 + T_3 +T_4 +T_5 .
\end{align}
These four terms will be bounded separately.
First, using the definition of $\boldsymbol{R}_{1,K}$, and then combining the Cauchy-Schwarz inequality with \eqref{edge_bubb}, gives
\begin{align*}
    T_1 &\leq C_1 \left(\sum_{K \in \omega_E} h_K^2 \| \boldsymbol{R}_{1,K}  \|^2_{0,K} \right)^{1/2} \left( \sum_{K \in \omega_E}   h_K^{-2} \|\boldsymbol{W}_E \|^2_{0,K} \right)^{1/2}, \\
    T_2 & \leq C_2\left( \sum_{K \in \omega_E} h_K^2 \| \boldsymbol{f} - \boldsymbol{f}_h \|_{0,K}^2  \right)^{1/2} \left( \sum_{K \in \omega_E} h_K^{-2} \|\boldsymbol{W}_E \|^2_{0,K} \right)^{1/2}, \\
    T_3 & \leq C_3 \left( \sum_{K \in \omega_E} \| \theta - \theta_h \|_{1,K}^2  \right)^{1/2} \left( \sum_{K \in \omega_E} h_E  \|\boldsymbol{R}_E \|^2_{0,K} \right)^{1/2}.
\end{align*}
Next, given the shape regularity of the grid, using the definition of $\boldsymbol{W}_b$ and \eqref{edge_bubb} gives
$$
h_K^{-2} \|\boldsymbol{W}_E \|^2_{0,K}  \leq h_E^{-2} \| \boldsymbol{W}_E \|_{0,K}^2 \leq h_E^{-1} \left\| h_E \boldsymbol{R}_E \right\|_{0,E}^2.  
$$
Hence, the following estimate holds 
\begin{align*}
    T_1 &\leq C_1 \left(\sum_{K \in \omega_E} h_K^2 \| \boldsymbol{R}_{1,K}  \|^2_{0,K} \right)^{1/2} \left( \sum_{K \in \omega_E} h_E\|\boldsymbol{R}_E \|^2_{0,K} \right)^{1/2}
    \\ & \leq  C_1 \left(\sum_{K \in \omega_E}  \| \boldsymbol{u} - \boldsymbol{u}_h \|^2_{1,K} + \|p-p_h\|_{0,K}^2  \right)^{1/2} \left( \sum_{K \in \omega_E} h_E \|\boldsymbol{R}_E \|^2_{0,K} \right)^{1/2}, \\
    T_2 & \leq C_2 \left(\sum_{K \in \omega_E} h_K^2 \| \boldsymbol{f} - \boldsymbol{f}_h \|_{0,K}^2  \right)^{1/2} \left( \sum_{K \in \omega_E} h_E  \|\boldsymbol{R}_E \|^2_{0,K} \right)^{1/2}.
\end{align*}
Hence, 
\begin{align*}
    T_1 +T_2 \lesssim \left(\sum_{K \in \omega_E}  \| \boldsymbol{u} - \boldsymbol{u}_h \|^2_{1,K} + \|p-p_h\|_{0,K}^2  + h_K^2 \| \boldsymbol{f} - \boldsymbol{f}_h \|_{0,K}^2  \right)^{1/2} \left( \sum_{K \in \omega_E} h_E \|\boldsymbol{R}_E \|^2_{0,K} \right)^{1/2}.
\end{align*}
Similarly, 
\begin{align*}
 T_4 &\leq C_4 \left( \sum_{K \in \omega_E} \| \boldsymbol{u} - \boldsymbol{u}_h \|^2_{1,K} + \|p-p_h\|_{0,K}^2   \right)^{1/2}  \left(  \sum_{K \in \omega_E} \|\nabla \boldsymbol{W}_E\|^2_{0,E} \right)^{1/2},
   \end{align*}
   where the second term is bounded using \eqref{edge_bubb} i.e.
$\|\nabla \boldsymbol{W}_E\|^2_{0,K} \lesssim h_E^{-1} \left\| \frac{h_E}{\nu} \boldsymbol{R}_E\right\|^2_{0,E} $. This implies 
\begin{align*}
T_4 & \leq C_4 \left( \sum_{K \in \omega_E} \| \boldsymbol{u} - \boldsymbol{u}_h \|^2_{1,K} + \|p-p_h\|_{0,K}^2   \right)^{1/2}  \left(  \sum_{K \in \omega_E} h_E \| \boldsymbol{R}_E\|^2_{0,E} \right)^{1/2}, \\
T_5 &\leq C_5 \left( \sum_{K \in \omega_E} \|\boldsymbol{u} - \boldsymbol{u}_h \|_{1,K}^2 \right)^{1/2} \left(  \sum_{K \in \omega_E} h_E \| \boldsymbol{R}_E\|^2_{0,E} \right)^{1/2}.
   \end{align*}
   Combining the bounds of $T_1$, $T_2$, $T_3$, $T_4$ and $T_5$ with \eqref{jump_bound} and \eqref{bound_eq} implies the first stated result. Similarly, we can prove the second bound.  
\end{proof}
\begin{lemma}\label{efficiency_3}
Let $K$ be an element of $\mathcal{K}_h$. The local trace residual satisfies
\begin{align*}
    \Psi_{J_1} & \lesssim \left( \| \boldsymbol{u} - \boldsymbol{u}_h \|_{1,K} + \|p-p_h\|_{0,K}  \right), \\
    \Psi_{J_2} & \lesssim \left( \| \theta - \theta_h \|_{1,K} \right).
\end{align*}
Moreover, we also have 
$$
 \Psi_{J_K} \lesssim \left( \| \boldsymbol{u} - \boldsymbol{u}_h \|_{1,K} + \|p-p_h\|_{0,K} + \| \theta - \theta_h \|_{1,K} \right).
 $$
\begin{proof}
Noting the conditions 
\begin{align*}
 \left.  \boldsymbol{u} \cdot \boldsymbol{n} \right|_E = 0, \quad \left. \left[ \mathbb{T}(\boldsymbol{u},p) \boldsymbol{n} \right]_{\tau} + \gamma \boldsymbol{u} \right|_E = 0 & \quad \forall E \in \Gamma_W, \\
 \left.  \mathbb{T}(\boldsymbol{u},p) \boldsymbol{n} \right|_E  = 0 & \quad  \forall  E \in \Gamma_O, \\
 \boldsymbol{u} = \boldsymbol{u}_{\star} & \quad  \forall  E \in \Gamma_I,
\end{align*}
where $(\boldsymbol{u},\theta,\theta,p)$ is the regular solution of \eqref{bb}, it follows that
$$
\begin{aligned}
 \Psi_{J_1}^2 &= \sum_{E \in \gamma_I} h_E   \| \bigl(p_h \boldsymbol{I}
- 2 \nu \varepsilon(\boldsymbol{u}_h)\bigr)\boldsymbol{n} \|_{0,E}^2 + \sum_{E \in \gamma_W} h_E   \|  \left[ \mathbb{T}(\boldsymbol{u}_h,p_h) \boldsymbol{n} \right]_{\tau} + \gamma \boldsymbol{u}_h \|_{0,E}^2 + \sum_{E \in \gamma_O} h_E   \| \mathbb{T}(\boldsymbol{u}_h,p_h) \boldsymbol{n} \|_{0,E}^2 \\ &  \quad  + \sum_{E \in \Gamma_W} h_E^{-1} \| \boldsymbol{u}_h \cdot \boldsymbol{n} \|^2_{0,E} + \sum_{E \in \Gamma_I} h_E^{-1} \| \boldsymbol{u}_h - \boldsymbol{u}^{\star} \|^2_{0,E} \\
& = \sum_{E \in \gamma_I} h_E   \| \bigl((p-p_h) \boldsymbol{I}
- 2 \nu \varepsilon(\boldsymbol{u} - \boldsymbol{u}_h)\bigr)\boldsymbol{n} \|_{0,E}^2 + \sum_{E \in \gamma_W} h_E   \|  \left[ \mathbb{T}(\boldsymbol{u}-\boldsymbol{u}_h,p-p_h) \boldsymbol{n} \right]_{\tau} + \gamma (\boldsymbol{u}-\boldsymbol{u}_h ) \|_{0,E}^2 \\ &  \quad  + \sum_{E \in \gamma_O} h_E   \| \mathbb{T}(\boldsymbol{u} - \boldsymbol{u}_h,p-p_h) \boldsymbol{n} \|_{0,E}^2  + \sum_{E \in \Gamma_W} h_E^{-1} \| ( \boldsymbol{u}-\boldsymbol{u}_h) \cdot \boldsymbol{n} \|^2_{0,E} + \sum_{E \in \Gamma_I} h_E^{-1} \| \boldsymbol{u} - \boldsymbol{u}_h \|^2_{0,E} .
\end{aligned}
$$
This implies 
$$
 \Psi_{J_1}  \lesssim \left( \| \boldsymbol{u} - \boldsymbol{u}_h \|_{1,K} + \|p-p_h\|_{0,K}  \right).
$$
Similarly, we can prove the second bound.
\end{proof}
\end{lemma}
\begin{theorem}
  Let $(\boldsymbol{u}, p, \theta)$ and $\left(\boldsymbol{u}_h, p_h, \theta_h \right)$ be the unique solutions of problems \eqref{bb} and \eqref{E}, respectively. Let $\Psi$ be defined as in \eqref{total_estimate}. Then, there exists a constant $C>0$ that is independent of $h$ such that
$$
\Psi \leq C\left( \left|\!\left|\!\left| (\boldsymbol{u}-\boldsymbol{u}_h , p -p_h, \theta - \theta_h)\right|\!\right|\!\right|    +\left(\sum_{K \in \mathcal{K}_h} h_K^2\left\|\boldsymbol{f}-\boldsymbol{f}_h\right\|_{0, K}^2 +h_K^2 \|g-g_h\|^2_{0,K}\right)^{1 / 2}\right) .
$$
\end{theorem}
\begin{proof}
     Combining Lemmas \ref{efficiency_1}, \ref{efficiency_2} and \ref{efficiency_3} implies the stated result.
\end{proof}

\section{Numerical Tests}\label{section6}
All routines have been implemented using the open-source finite element library FEniCS \cite{alnaes2015fenics}. In some experiments, we employ uniform meshes, while in others, we use adaptive mesh refinement. Specifically, starting with an initial mesh $\mathcal{K}_{0,\Omega}$, we apply the iterative refinement loop
\begin{align*} 
\text{Solve} \rightarrow \text{Estimate} \rightarrow \text{Mark} \rightarrow \text{Refine} 
\end{align*}
to generate a sequence of (nested) regular meshes $\left\{\mathcal{K}_{\ell}\right\}$ with mesh size $h_{\ell}$. At each step, we compute the local error estimators $\Psi_{K}$ for all $K$ over the previous mesh $\mathcal{K}_h$ and refine those elements $K \in \mathcal{K}_h$ according to
\begin{align*} 
\Psi_{K} \geq \eta \max\{\Psi_{K} : K \in \mathcal{K}_h\}, 
\end{align*}
where $\eta \in (0, 1)$ is a prescribed parameter. We assess the quality of the a posteriori error estimator through the \textit{effectivity index}, which is required to remain bounded as $h$ approaches zero and is defined by
\begin{align*}
\mathrm{Effec} := \frac{\Psi}{\|(\boldsymbol{u} - \boldsymbol{u}_h, \theta - \theta_h, p - p_h)\|}. 
\end{align*}
\subsection{Test 1: 2D Convergence test}
In our numerical experiment, the computational domain is taken as 
\(
\Omega = (-1,1)^2
\),
and we consider a sequence of uniformly refined square meshes.
We present a numerical test based on the following exact solution:
\begin{align*}
\boldsymbol{u}(x,y) = (\sin y,\cos x), \quad p(x,y) = 1 + \sin(xy), \quad
\theta(x,y) = 1 + \cos(xy).
\end{align*}
We compute the forcing terms and impose compensating boundary conditions in accordance with the exact solution. The boundary conditions in \eqref{nsstokes0} are prescribed on disjoint portions of $\partial \Omega$ as follows:
\begin{align*}
\eqref{nsstokes0}_1 &:\quad \{ (x,y) \in \partial \Omega : y = -1 \text{ or } y = 1 \}, \\
\eqref{nsstokes0}_2 &:\quad \{ (x,y) \in \partial \Omega : x = -1 \}, \\
\eqref{nsstokes0}_3 &:\quad \{ (x,y) \in \partial \Omega : x = 1 \}.
\end{align*}
The values of the parameters are $\alpha = 10$, $\gamma = 10$, $\kappa = 10$, and $\nu = 10$, and the Nitsche parameter is set to $\gamma_N = 1$.
We approximate the velocity and pressure using the Taylor--Hood element $\mathbb{P}_2$--$\mathbb{P}_1$, and the temperature using $\mathbb{P}_2$ elements. We observe from Table~\ref{table_2d_conv} that the approximation errors for the velocity, pressure, and temperature, as well as the corresponding convergence rates, are in good agreement with the theoretical results.
\begin{table}[h!]
\centering
\begin{footnotesize}
\caption{Errors, convergence rates, estimator, and efficiency under uniform refinement in 2D.}
\begin{tabular}{rcccccccccccr}
\toprule
DOF & $h$ & $\|\nabla(\boldsymbol{u}-\boldsymbol{u}_h)\|$ & rate 
& $\|p-p_h\|_{0}$ & rate 
& $\|\nabla(\theta-\theta_h)\|$ & rate
& $e(\text{total})$ & $r(\text{total})$
& $e(\Psi)$ & $r(\Psi)$
& Effec \\
\midrule
948     & 0.3536 & 8.3e-03 & --   & 1.2e-02 & --   & 2.6e-02 & --   & 1.5e-02 & --   & 6.7e-01 & --   & 45.18 \\
3556    & 0.1768 & 2.0e-03 & 2.13 & 2.9e-03 & 2.15 & 6.7e-03 & 2.07 & 3.6e-03 & 2.14 & 1.7e-01 & 2.09 & 46.60 \\
13764   & 0.0884 & 5.6e-04 & 1.91 & 7.2e-04 & 2.08 & 1.7e-03 & 2.04 & 9.5e-04 & 1.97 & 4.2e-02 & 2.04 & 44.61 \\
54148   & 0.0442 & 1.2e-04 & 2.21 & 1.8e-04 & 2.03 & 4.2e-04 & 2.02 & 2.2e-04 & 2.15 & 1.1e-02 & 2.02 & 48.71 \\
214788  & 0.0221 & 3.0e-05 & 2.05 & 4.9e-05 & 1.88 & 1.1e-04 & 2.00 & 5.7e-05 & 1.94 & 2.7e-03 & 2.01 & 46.53 \\
855556  & 0.0110 & 7.7e-06 & 1.96 & 2.3e-05 & 1.10 & 3.0e-05 & 1.81 & 2.4e-05 & 1.25 & 6.7e-04 & 2.00 & 27.69 \\
3415044 & 0.0055 & 3.3e-06 & 1.24 & 2.0e-05 & 0.19 & 1.7e-05 & 0.86 & 2.0e-05 & 0.25 & 1.7e-04 & 1.98 & 8.34 \\
\bottomrule
\end{tabular}
\label{table_2d_conv}
\end{footnotesize}
\end{table}

\subsection{Test 2: 3D Convergence test}
In our numerical experiment, the computational domain is taken as $\Omega = (0, 1)^3$, and we consider a sequence
of uniformly refined square meshes. We present a numerical test based on the following exact solution:
\begin{align*}
\boldsymbol{u}(x,y,z)&=(\sin(\pi x) \cos(\pi y) \cos(\pi z),-2 \cos(\pi x) \sin(\pi y) \cos(\pi z), \cos(\pi x) \cos(\pi y) \sin(\pi z)), \\
p(x,y,z) & = \sin(\pi x) \sin(\pi y) \sin(\pi z), \\
\theta(x,y,z) & = 1-\sin(\pi x) \cos(\pi y) \sin(\pi z).
\end{align*}
We compute the forcing terms and impose compensating boundary conditions in accordance with the exact solution. The boundary conditions in \eqref{nsstokes0} are prescribed on disjoint portions of $\partial \Omega$ as follows:
\begin{align*}
\eqref{nsstokes0}_1 &:\quad \{ (x,y,z) \in \partial \Omega : z = 0 \text{ or } z = 1 \}, \\
\eqref{nsstokes0}_2 &:\quad \{ (x,y,z) \in \partial \Omega : x = 0 \text{ or } x = 1 \}, \\
\eqref{nsstokes0}_3 &:\quad \{ (x,y,z) \in \partial \Omega : y = 0 \text{ or } y = 1 \}.
\end{align*}
The values of the parameters are $\alpha = 10$, $\gamma = 10$, $\kappa = 1$, and $\nu = 1$, and the Nitsche parameter is set to $\gamma_N = 50$.
We approximate the velocity and pressure using the Taylor--Hood element $\mathbb{P}_2$--$\mathbb{P}_1$, and the temperature using $\mathbb{P}_2$ elements. We observe from Table~\ref{table_3d_conv} that the approximation errors for the velocity, pressure, and temperature, as well as the corresponding convergence rates, are in good agreement with the theoretical results.
\begin{table}[h!]
\centering
\begin{footnotesize}
\caption{Errors, convergence rates, estimator, and efficiency under uniform refinement in 3D.}
\begin{tabular}{rcccccccccccr}
\toprule
DOF & $h$ 
& $\|\nabla(\boldsymbol{u}-\boldsymbol{u}_h)\|$ & rate 
& $\|p-p_h\|_{0}$ & rate 
& $\|\nabla(\theta-\theta_h)\|$ & rate
& $e(\text{total})$ & $r(\text{total})$
& $e(\Psi)$ & $r(\Psi)$
& Effec \\
\midrule
527     & 0.8660 & 1.4e+00 & --   & 5.6e-01 & --   & 5.6e-01 &  --  & 1.6e+00 &   -- & 1.4e+01 &   -- & 8.83 \\
3041    & 0.4330 & 4.0e-01 & 1.83 & 7.6e-02 & 2.88 & 1.6e-01 & 1.78 & 4.4e-01 & 1.89 & 3.5e+00 & 2.02 & 8.03 \\
20381   & 0.2165 & 1.1e-01 & 1.91 & 1.1e-02 & 2.83 & 4.4e-02 & 1.89 & 1.2e-01 & 1.92 & 7.9e-01 & 2.17 & 6.73 \\
148661  & 0.1083 & 2.8e-02 & 1.96 & 2.0e-03 & 2.43 & 1.1e-02 & 1.95 & 3.0e-02 & 1.96 & 1.7e-01 & 2.19 & 5.76 \\
1134437 & 0.0541 & 7.0e-03 & 1.98 & 4.5e-04 & 2.13 & 2.9e-03 & 1.98 & 7.6e-03 & 1.98 & 4.0e-02 & 2.11 & 5.28 \\
\bottomrule
\end{tabular}
\label{table_3d_conv}
\end{footnotesize}
\end{table}
\subsection{Test 3: Non-convex domains}
Consider the non-convex  L-shaped and T-shaped domains 
$$ 
\Omega_L = (-1,1)^2 \setminus (0,1)^2, \qquad
\Omega_T = \bigl((-1.5,1.5)\times(0,1)\bigr) \cup \bigl((-0.5,0.5)\times(-2,0)\bigr),
$$ 
respectively. We consider the sources $\boldsymbol{f} = (1,1)^T $ and $ g= 1$. 

The boundary conditions in \eqref{nsstokes0} are prescribed on disjoint portions of $\partial \Omega_L$ and $\partial \Omega_T$ as follows:
\begin{align*}
\eqref{nsstokes0}_1 &:\quad \{ (x,y) \in \partial \Omega_L : x = -1 \text{ or } y = -1 \}, \\
\eqref{nsstokes0}_2 &:\quad \{ (x,y) \in \partial \Omega_L : x = 1 \text{ or } y = 0 \}, \\
\eqref{nsstokes0}_3 &:\quad \{ (x,y) \in \partial \Omega_L : x = 0 \text{ or } y = 1 \},
\\
\eqref{nsstokes0}_1 &:\quad \{ (x,y) \in \partial \Omega_T : x = -1.5 \text{ or } y = 1 \text{ or } x = 1.5\}, \\
\eqref{nsstokes0}_2 &:\quad \{ (x,y) \in \partial \Omega_T : x = 0.5 \text{ or } y = -2 \text{ or } x = -0.5 \}, \\
\eqref{nsstokes0}_3 &:\quad \{ (x,y) \in \partial \Omega_T :  y = 0 \}.
\end{align*}
We choose $\gamma_N=40, \gamma =10, \beta=1, \nu=1, \alpha=0.1, \eta = 0.6$ and $ \kappa=1$. The exact solutions for this problem are not
known. However, we anticipate significant challenges in convergence when using uniform mesh refinement, primarily due to the presence of re-entrant corners, which lead to singularities in the solution \cite{choi2013stationary}. In contrast, our adaptive refinement strategy proves to be much more effective in dealing with these issues. As demonstrated in Fig. \ref{conv4}, the adaptive method focuses on the refinement in the regions surrounding the re-entrant corners. This targeted approach helps to accurately capture the singularities and complex behavior in these areas, which uniform refinement often fails to do efficiently. As we continue refining the mesh adaptively, we observe that the global error estimators decrease optimally. This behavior, depicted in Fig. \ref{con_test3}, validates the efficacy of the adaptive
scheme. Additionally, Figs. \ref{con_test1}, \ref{con_test2} provide detailed visualizations of the numerical solutions. These plots illustrate the improved accuracy and resolution achieved by the adaptive strategy. The finer mesh around the re-entrant corners allows for a more precise approximation of the solution, which is critical for capturing the true nature of the problem.

\begin{figure}[H]
    \centering
    \begin{subfigure}[b]{0.32\textwidth}
        \centering
        \includegraphics[width=\textwidth]{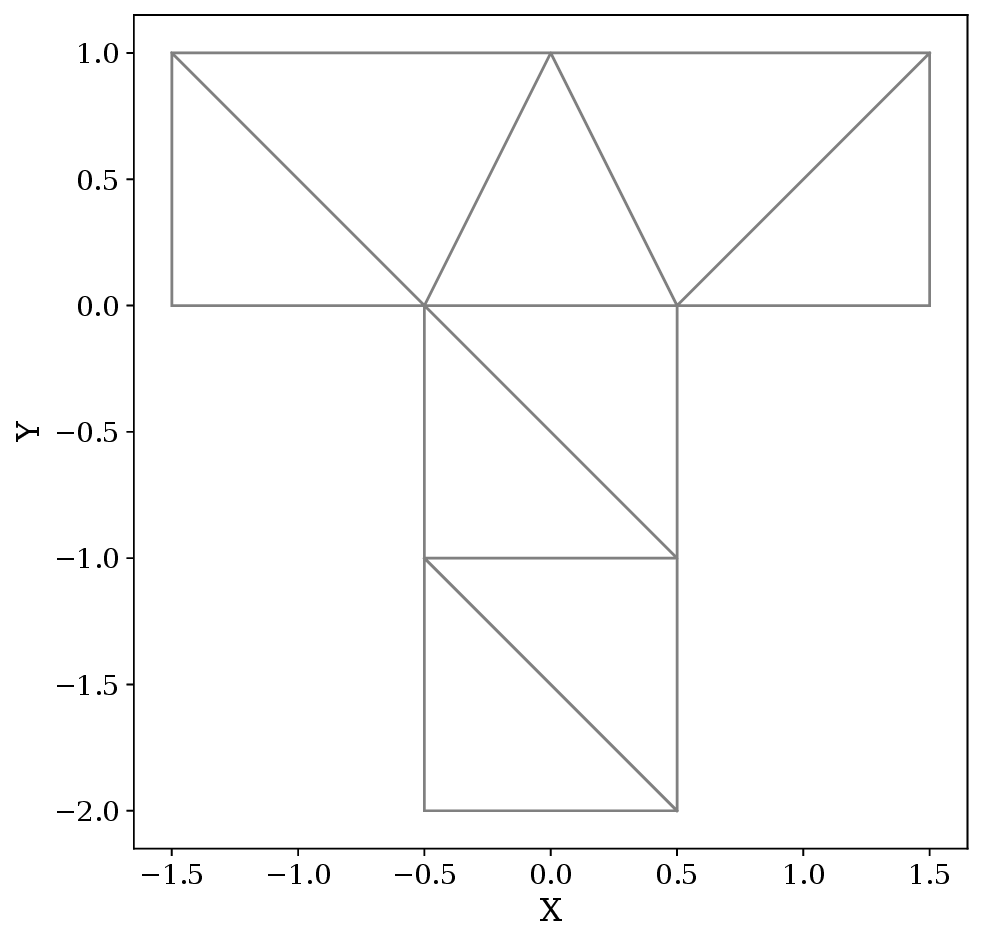}
         \caption{(a) Initial T-shaped mesh}
    \end{subfigure}
    \hfill
    \begin{subfigure}[b]{0.32\textwidth}
        \centering
        \includegraphics[width=\textwidth]{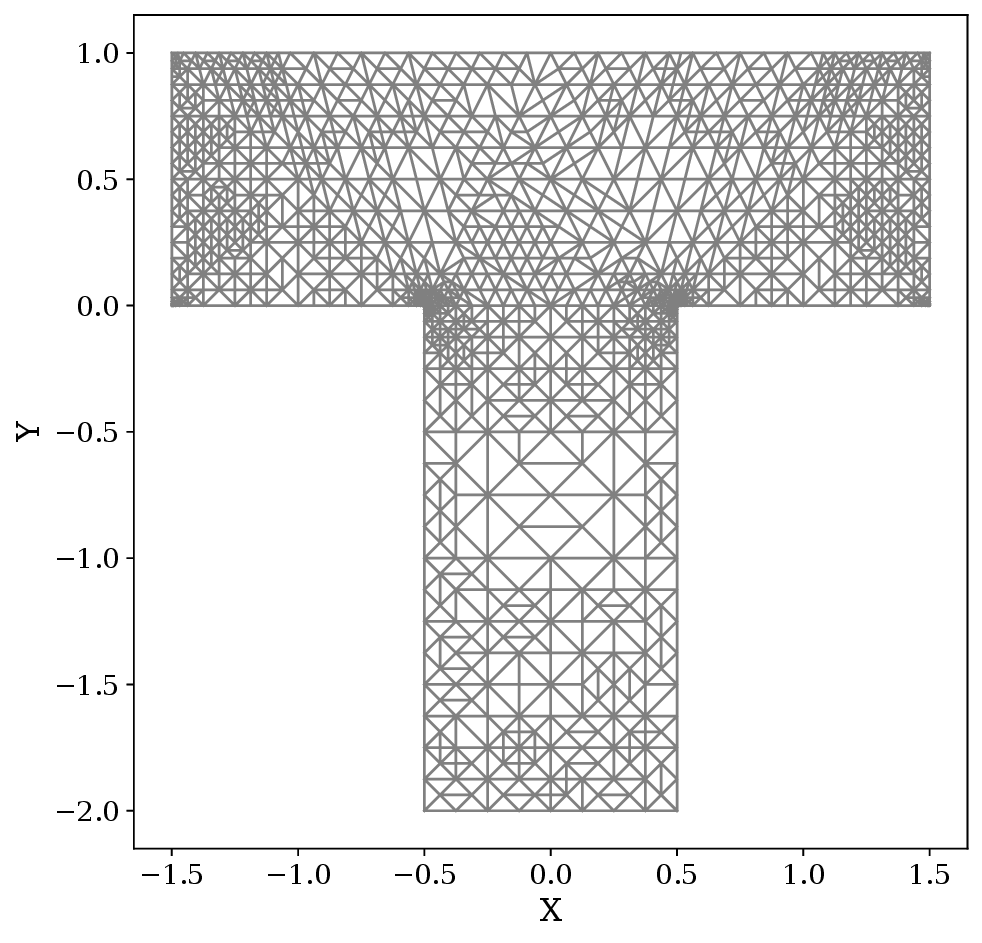}
        \caption{$14420$ DOF}
    \end{subfigure}
    \hfill
    \begin{subfigure}[b]{0.32\textwidth}
        \centering
        \includegraphics[width=\textwidth]{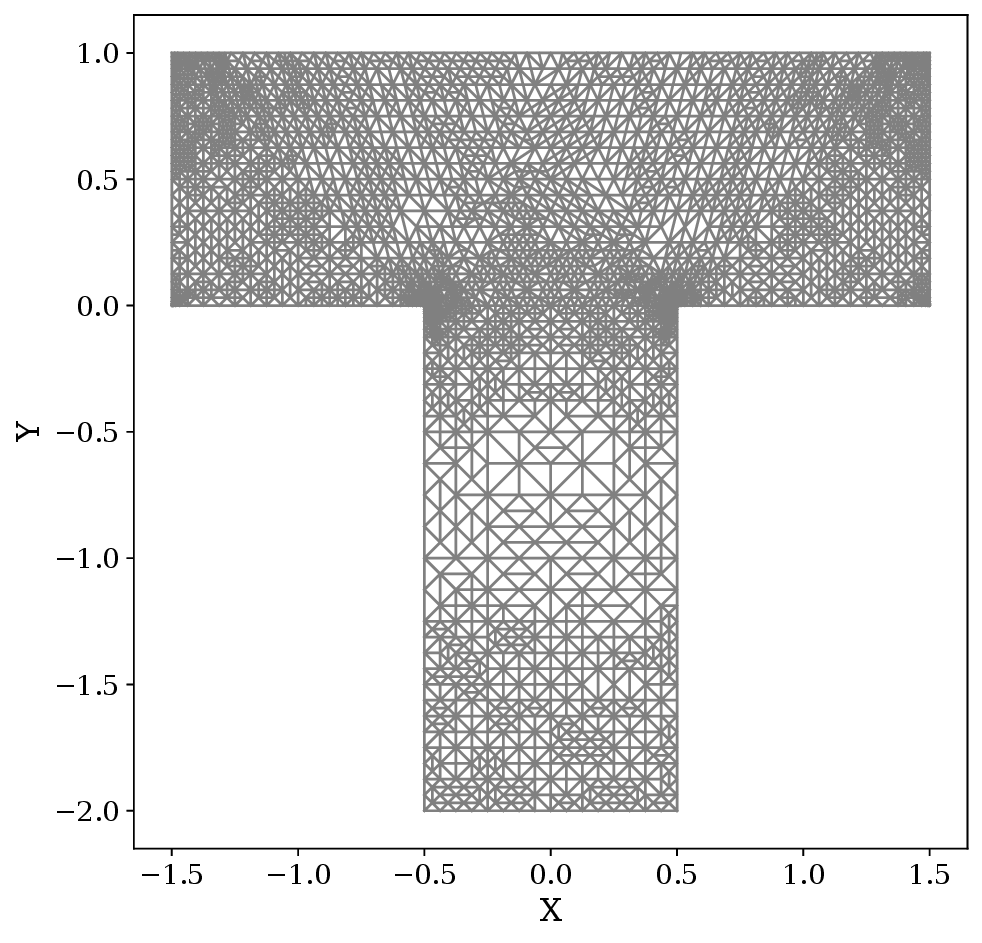}
        \caption{$43371 $ DOF}
    \end{subfigure}
    \begin{subfigure}[b]{0.32\textwidth}
        \centering
        \includegraphics[width=\textwidth]{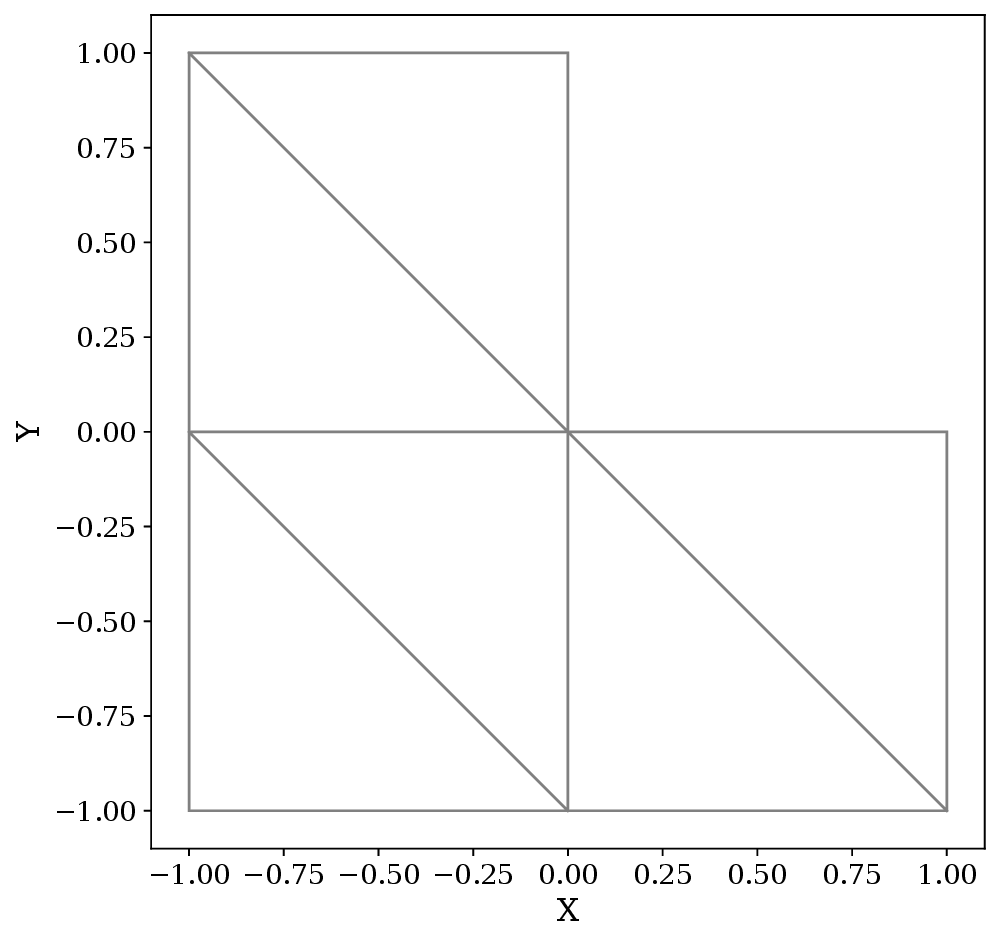}
         \caption{(a) Initial L-shaped mesh}
    \end{subfigure}
    \hfill
    \begin{subfigure}[b]{0.32\textwidth}
        \centering
        \includegraphics[width=\textwidth]{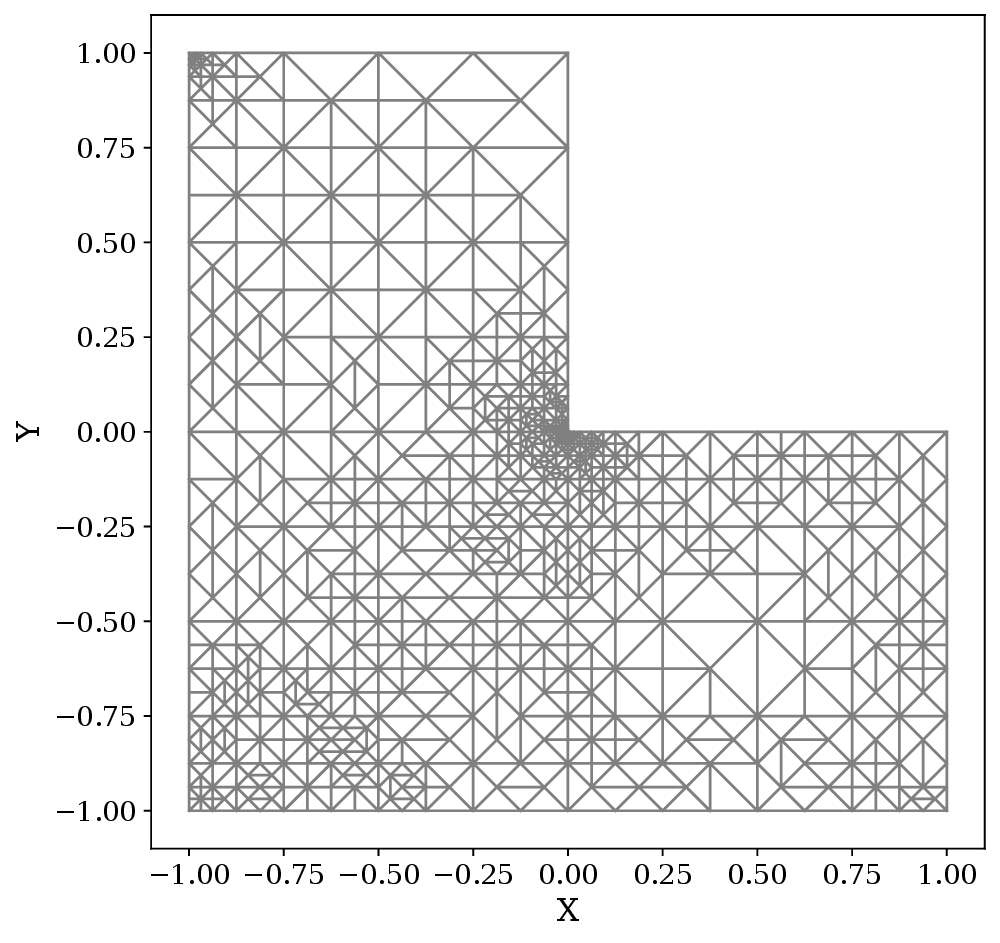}
        \caption{$7382$ DOF}
    \end{subfigure}
    \hfill
    \begin{subfigure}[b]{0.32\textwidth}
        \centering
        \includegraphics[width=\textwidth]{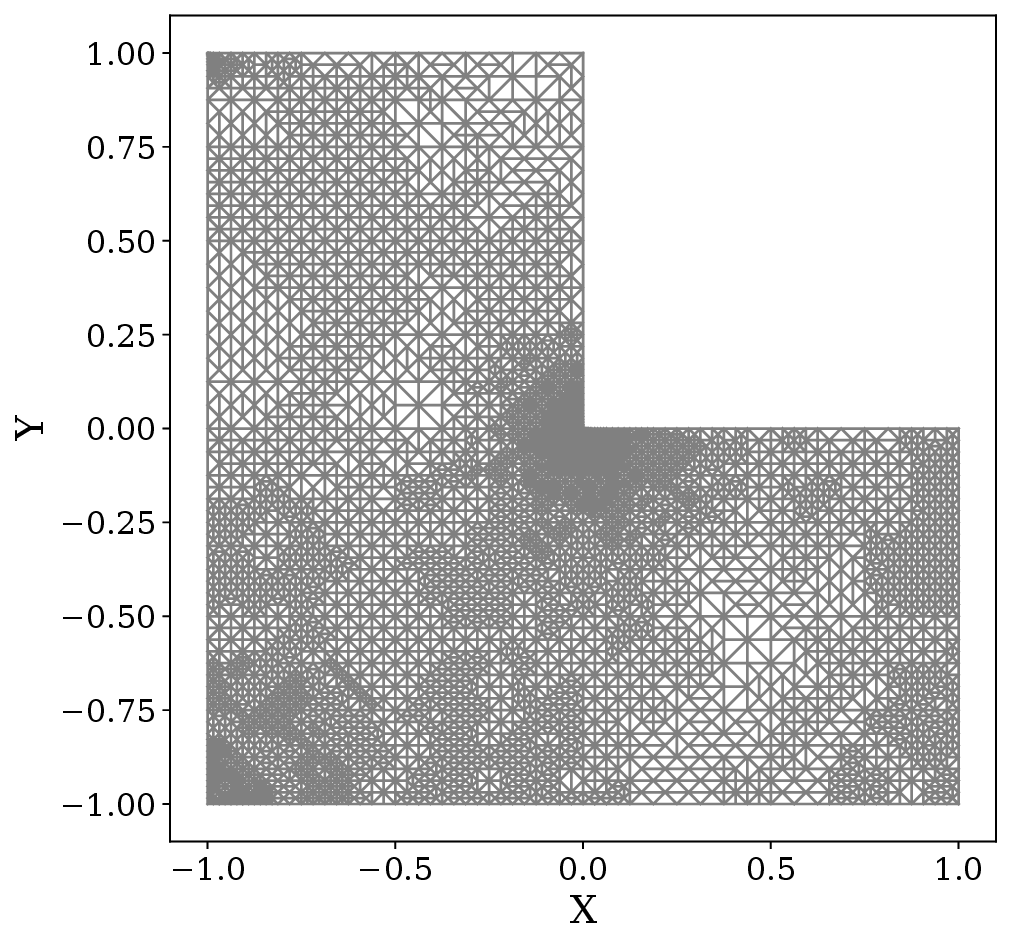}
        \caption{$55351 $ DOF}
    \end{subfigure}
    \caption{Test 3: Initial and adaptively refined meshes showing refinement near re-entrant corners.}
    \label{conv4}
 \end{figure}
\begin{figure}[H]
	\centering
	\begin{subfigure}[b]{0.3\textwidth}
		\centering
		\includegraphics[width=\textwidth]{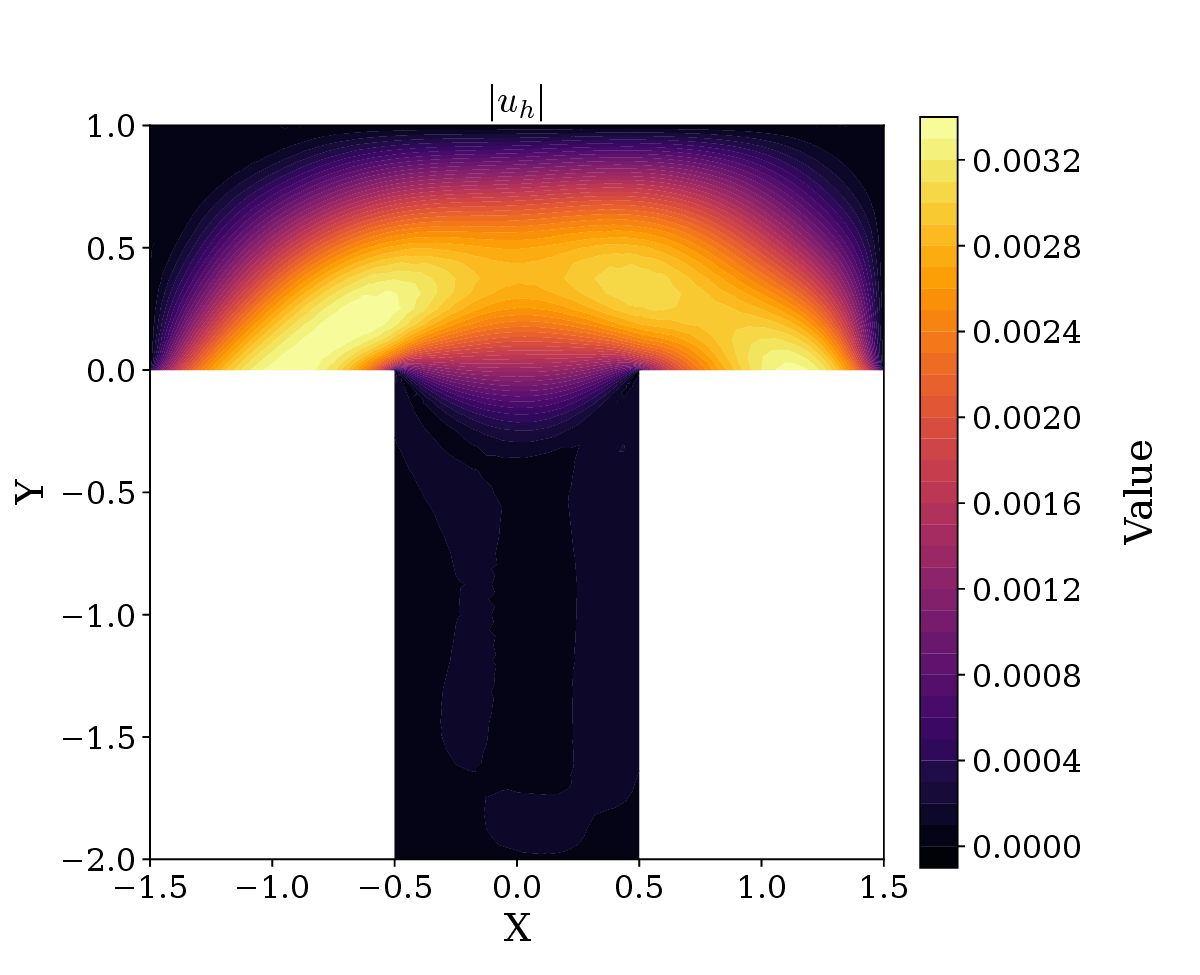}
		\caption{$| \boldsymbol{u}_h |$}
	\end{subfigure}
	\hspace{0.5cm} 
	\begin{subfigure}[b]{0.3\textwidth}
		\centering
		\includegraphics[width=\textwidth]{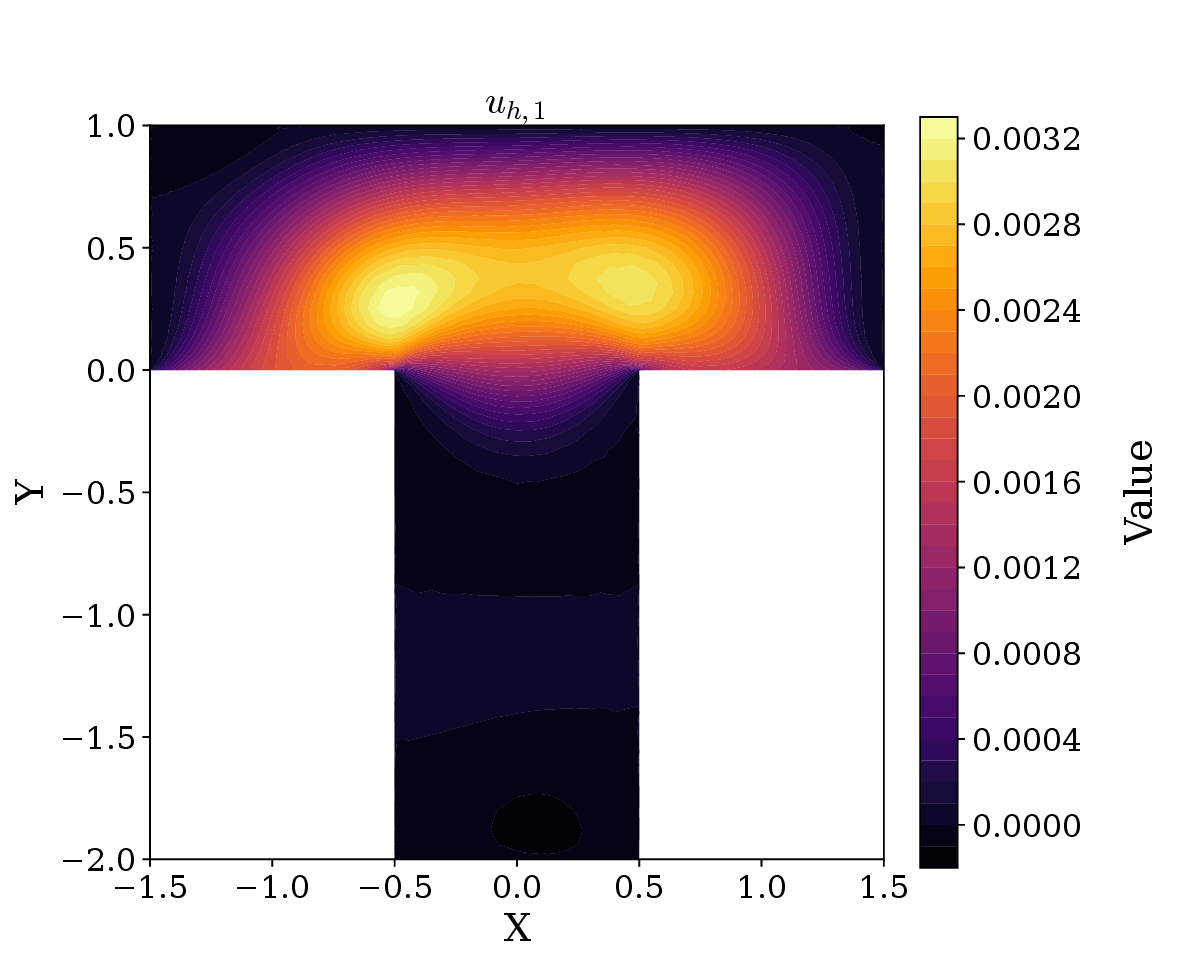}
		\caption{$\boldsymbol{u}_{h1} $}
	\end{subfigure}
    \hspace{0.5cm} 
	\begin{subfigure}[b]{0.3\textwidth}
		\centering
		\includegraphics[width=\textwidth]{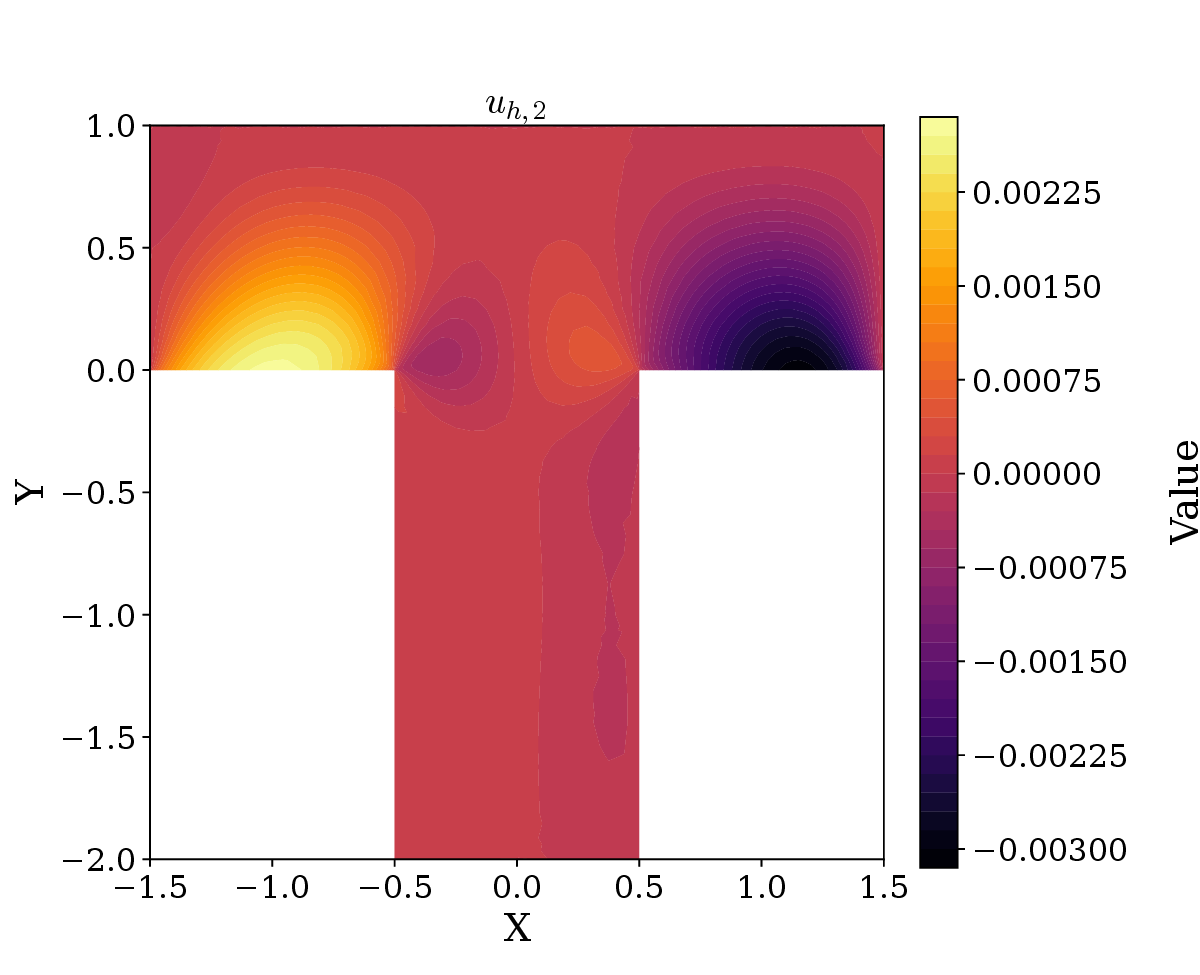}
		\caption{$\boldsymbol{u}_{h2}$}
	\end{subfigure} 
    \hspace{0.5cm} 
	\begin{subfigure}[b]{0.3\textwidth}
		\centering
		\includegraphics[width=\textwidth]{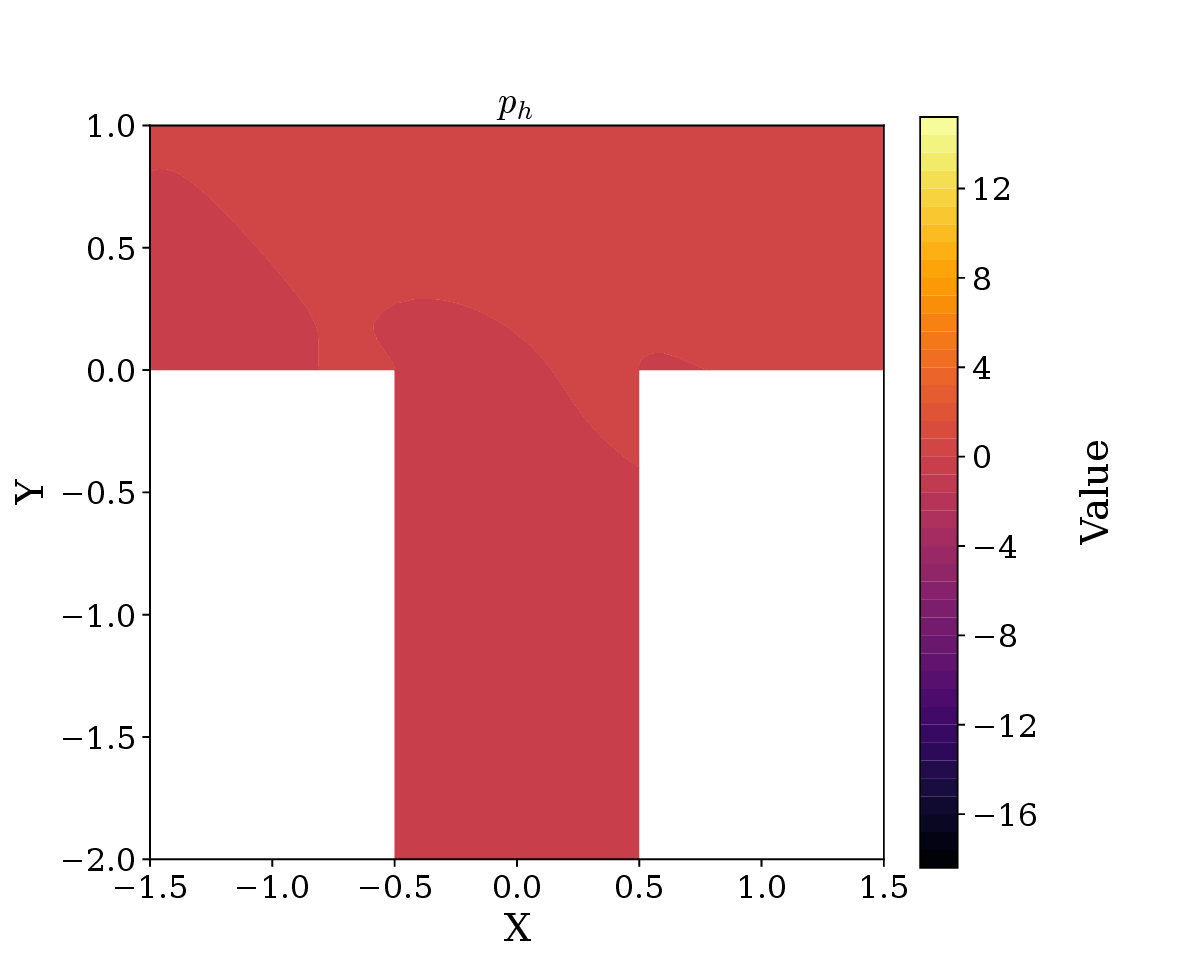}
		\caption{$p_h$}
	\end{subfigure}
    \hspace{0.5cm} 
	\begin{subfigure}[b]{0.3\textwidth}
		\centering
		\includegraphics[width=\textwidth]{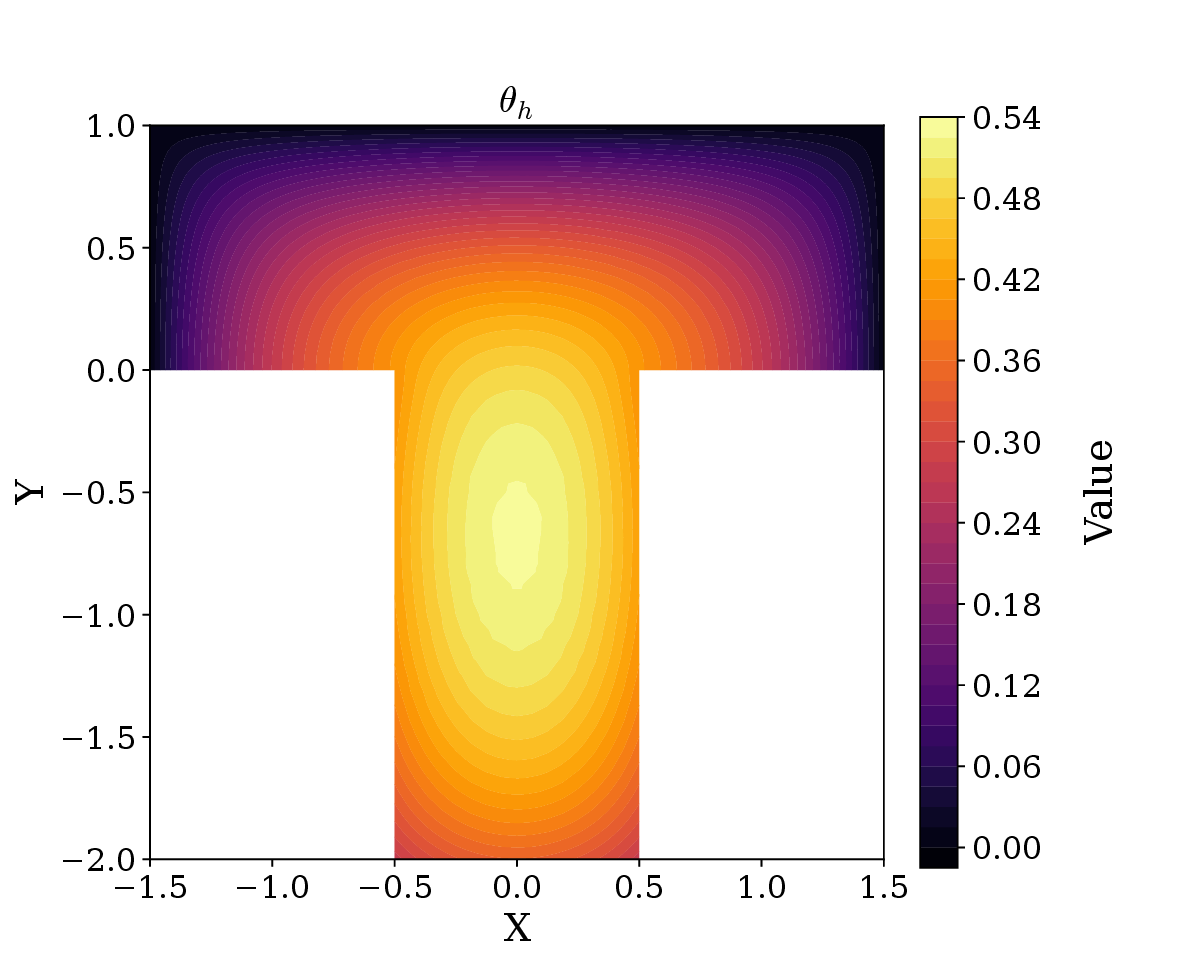}
		\caption{$\theta_h$}
	\end{subfigure}
	\caption{Test 3: Plots of numerical solutions of the velocity magnitude $(| \boldsymbol{u}_h |)$, velocity components $(\boldsymbol{u}_{h1}, \boldsymbol{u}_{h2})$, pressure $(p_h)$ and temperature $(\theta_h)$, respectively, for the T-shape domain.}
	\label{con_test1}
\end{figure}
\begin{figure}[H]
	\centering
	\begin{subfigure}{0.3\textwidth}
		\centering
		\includegraphics[width=\textwidth]{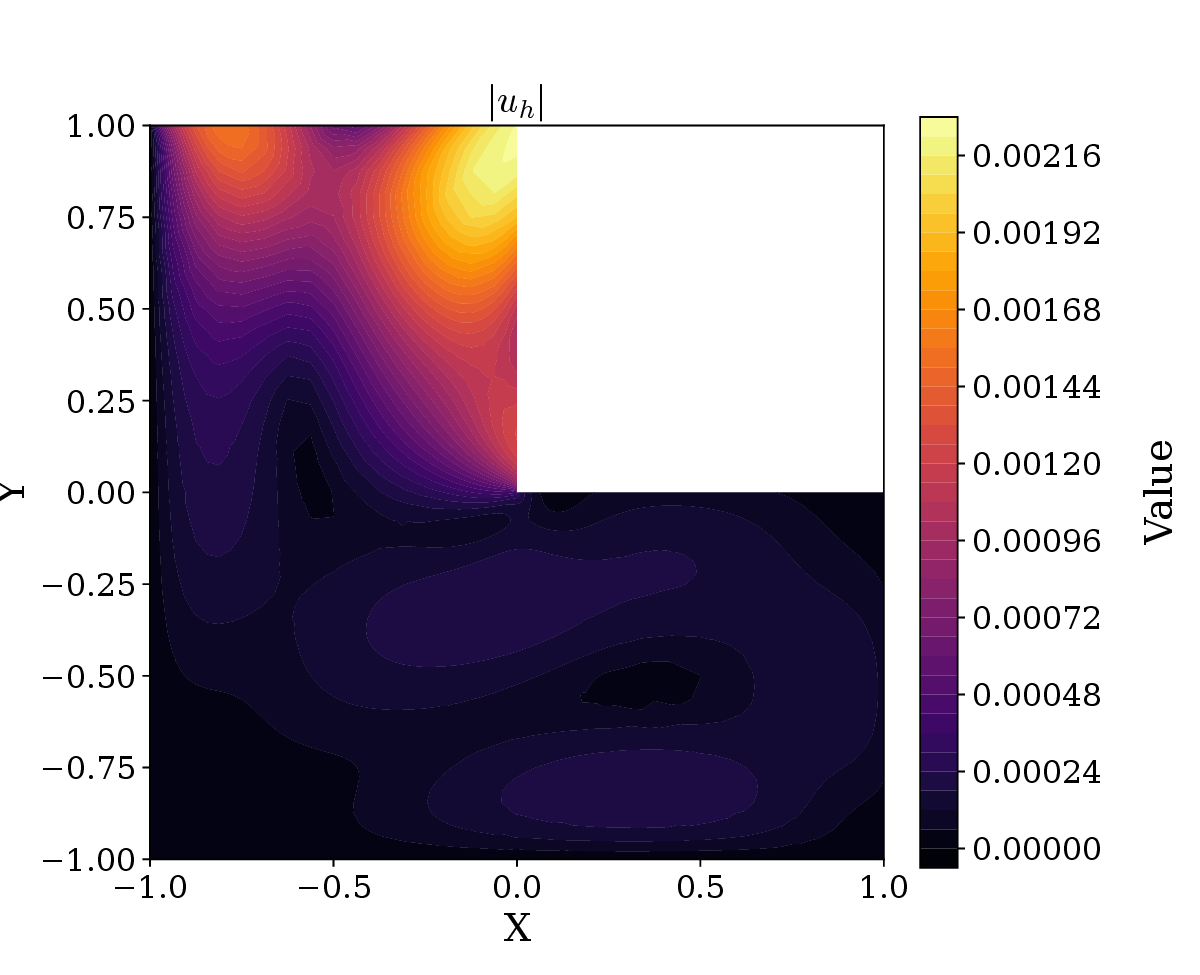}
		\caption{$| \boldsymbol{u}_h |$}
	\end{subfigure}
	\hspace{0.5cm} 
	\begin{subfigure}{0.3\textwidth}
		\centering
		\includegraphics[width=\textwidth]{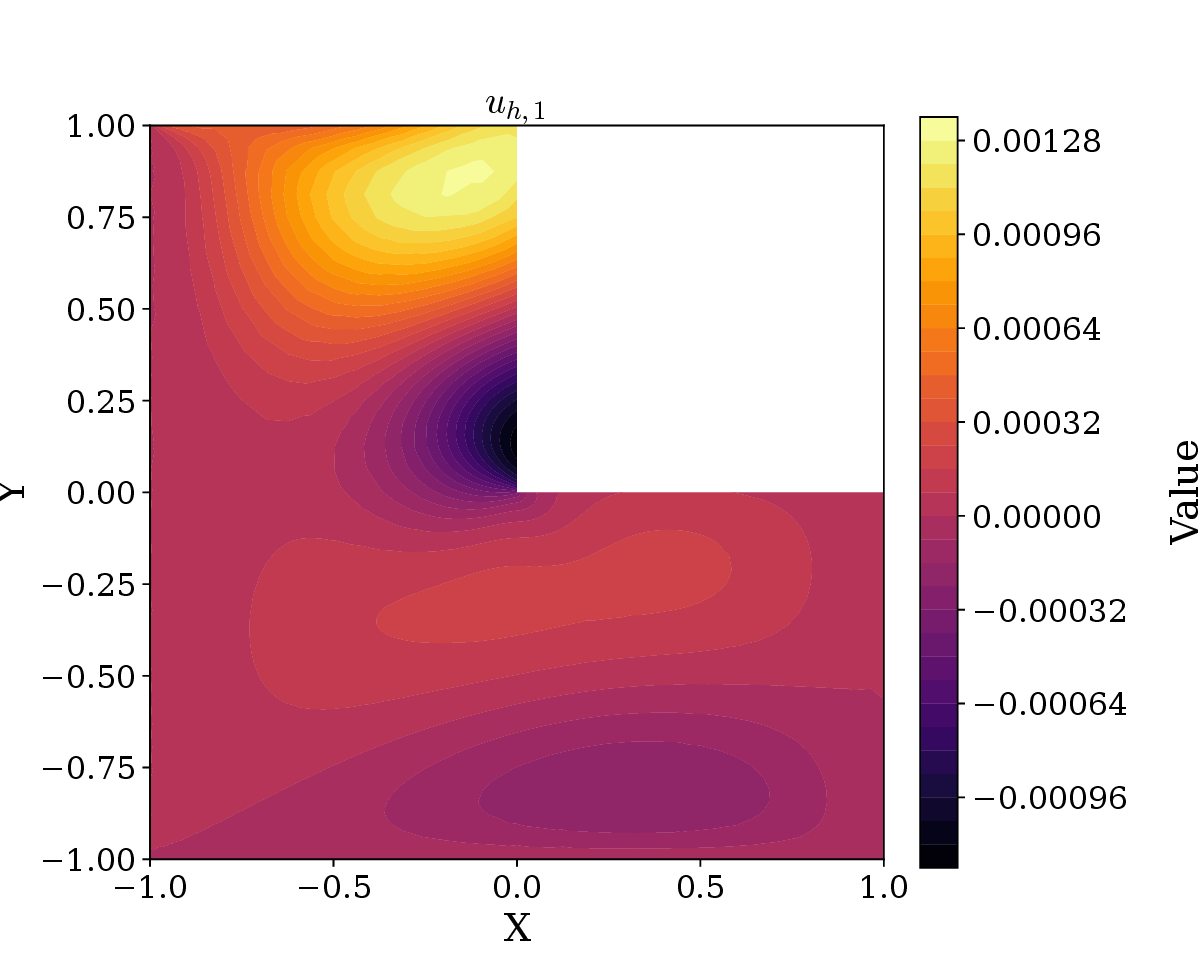}
		\caption{$\boldsymbol{u}_{h1} $}
	\end{subfigure}
    \hspace{0.5cm} 
	\begin{subfigure}{0.3\textwidth}
		\centering
		\includegraphics[width=\textwidth]{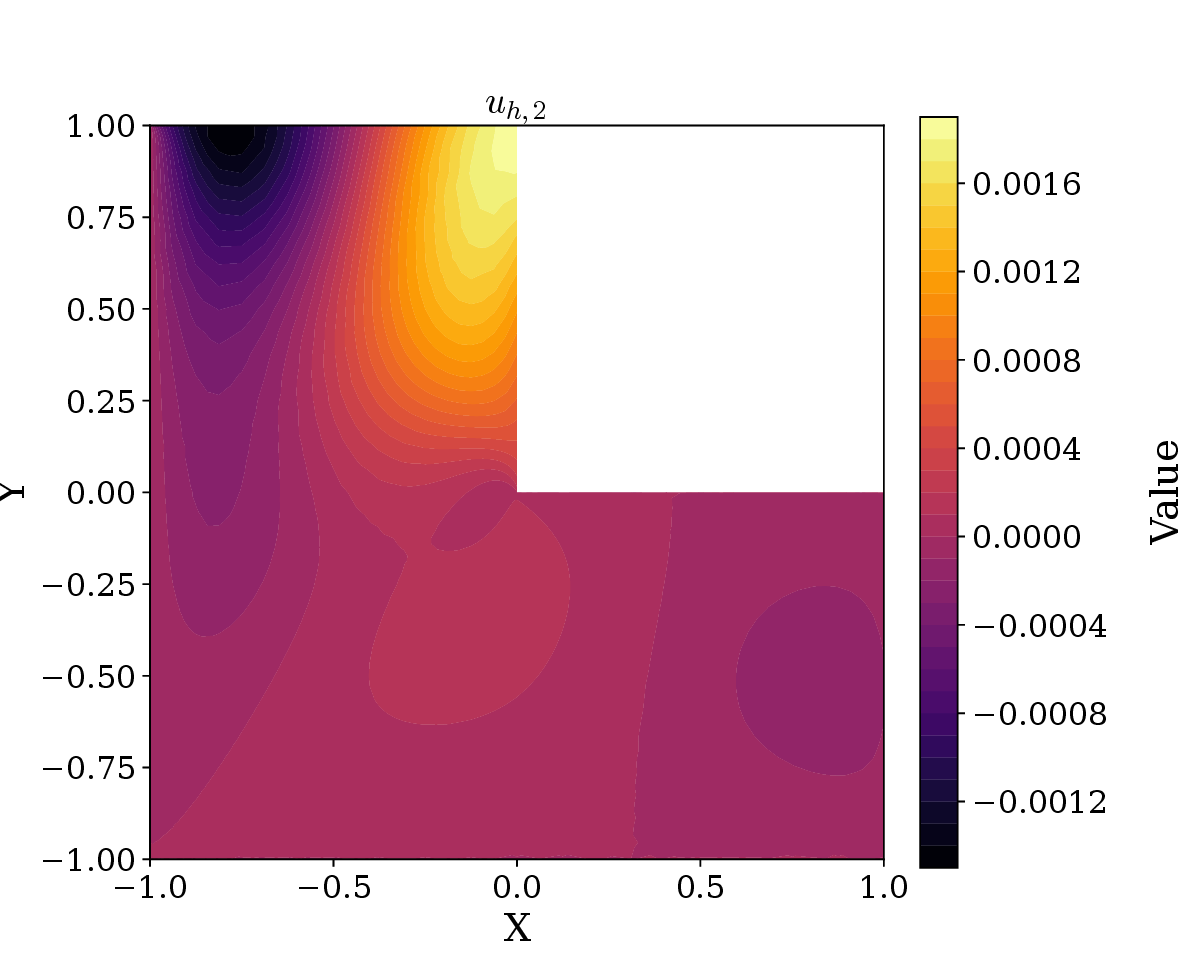}
		\caption{$\boldsymbol{u}_{h2}$}
	\end{subfigure} 
    \hspace{0.5cm} 
	\begin{subfigure}{0.3\textwidth}
		\centering
		\includegraphics[width=\textwidth]{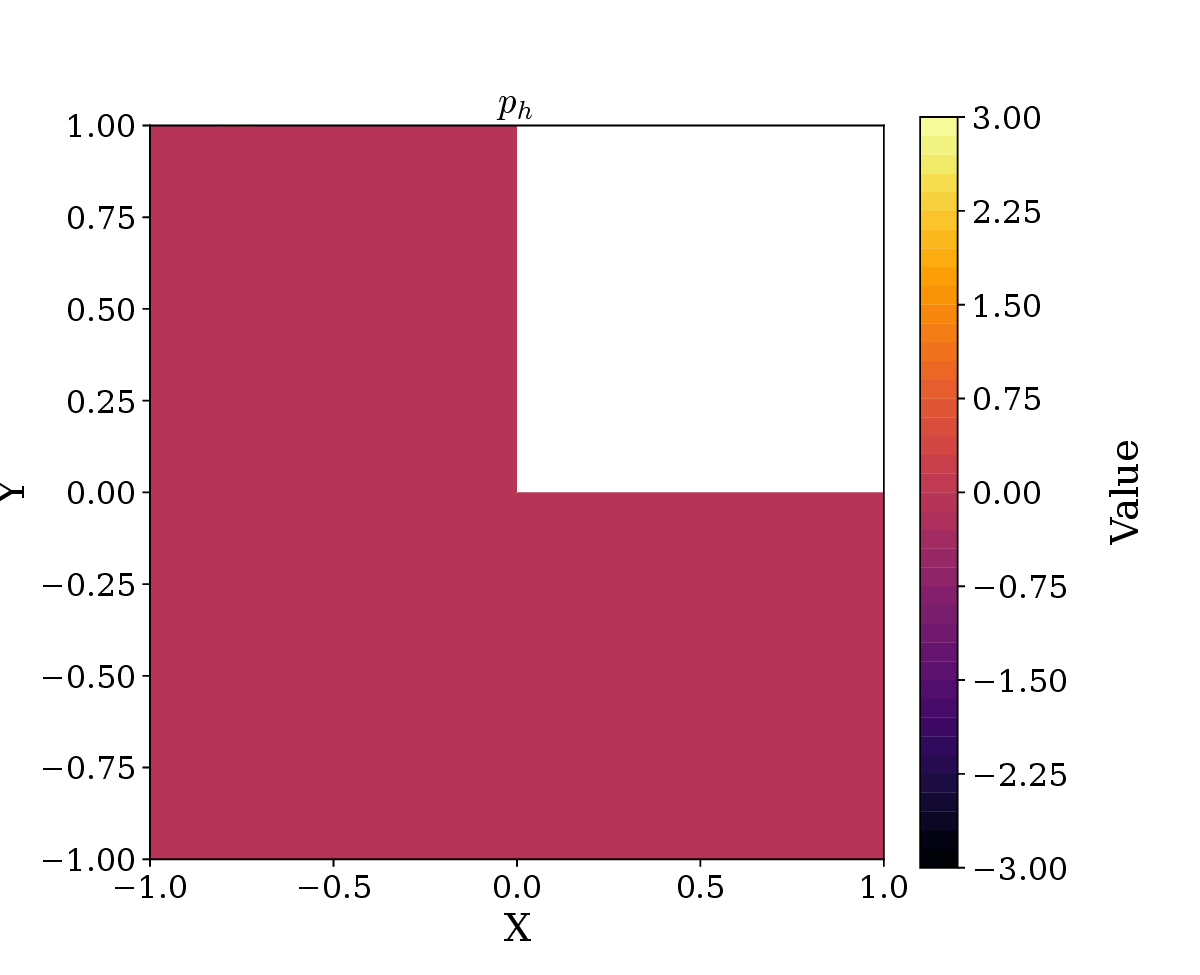}
		\caption{$p_h$}
	\end{subfigure}
    \hspace{0.5cm} 
	\begin{subfigure}{0.3\textwidth}
		\centering
		\includegraphics[width=\textwidth]{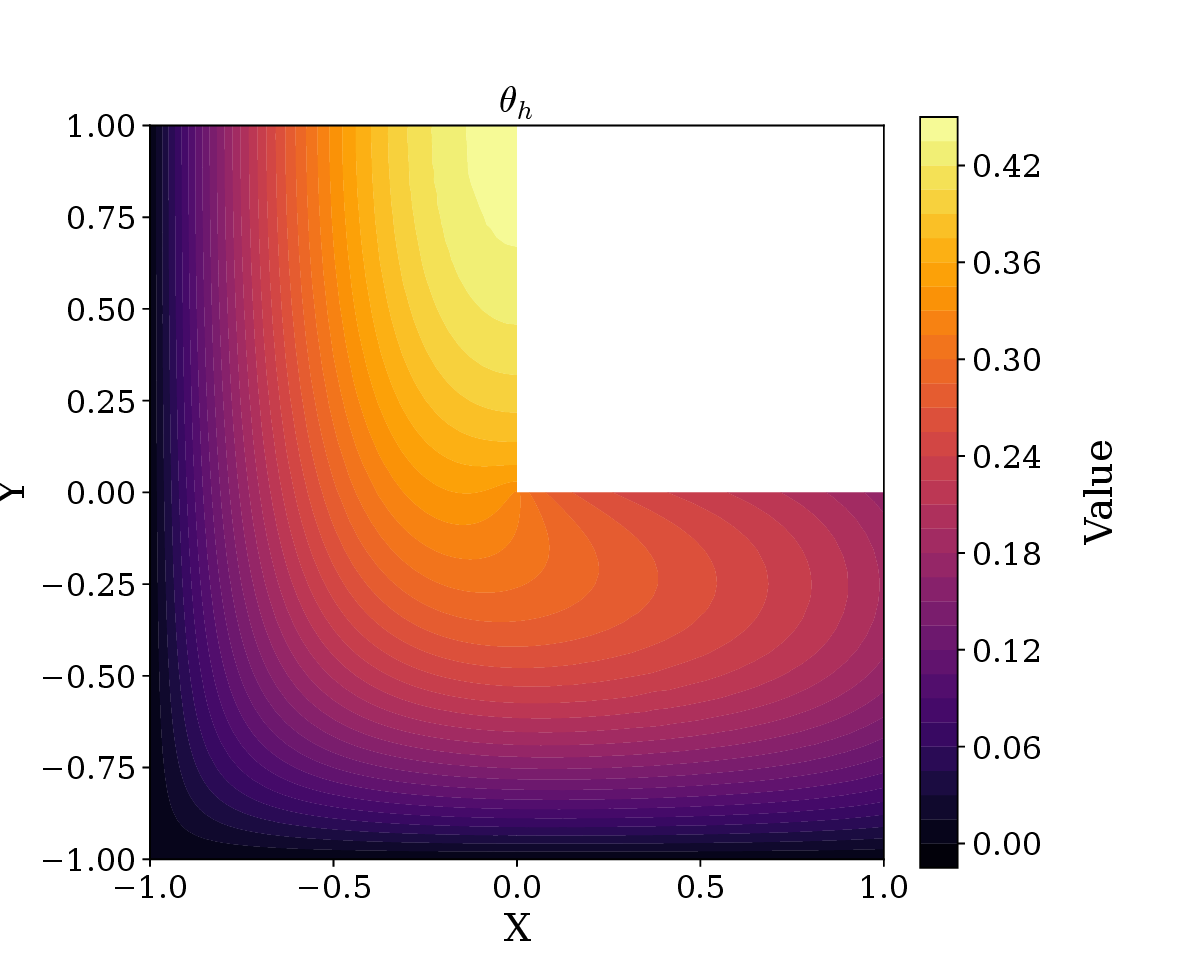}
		\caption{$\theta_h$}
	\end{subfigure}
    \caption{Test 3: Plots of numerical solutions of the velocity magnitude $(| \boldsymbol{u}_h |)$, velocity components $(\boldsymbol{u}_{h1}, \boldsymbol{u}_{h2})$, pressure $(p_h)$ and temperature $(\theta_h)$, respectively, for the L-shape domain.}
	\label{con_test2}
\end{figure}
\begin{figure}[H]
	\centering
	\begin{subfigure}{0.4\textwidth}
		\centering
		\includegraphics[width=\textwidth]{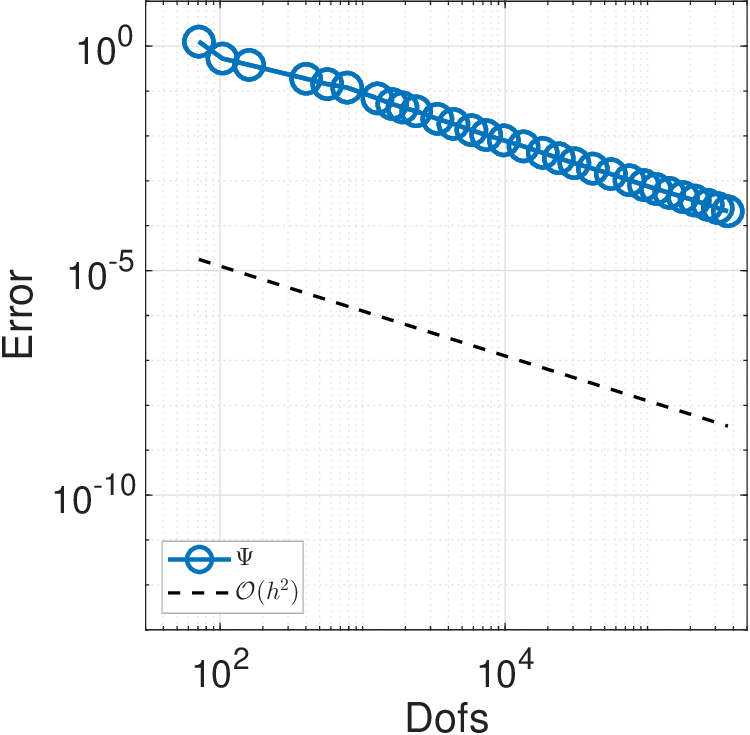}
		\caption{L-shape}
	\end{subfigure}
    \hspace{0.5cm} 
	\begin{subfigure}{0.4\textwidth}
		\centering
		\includegraphics[width=\textwidth]{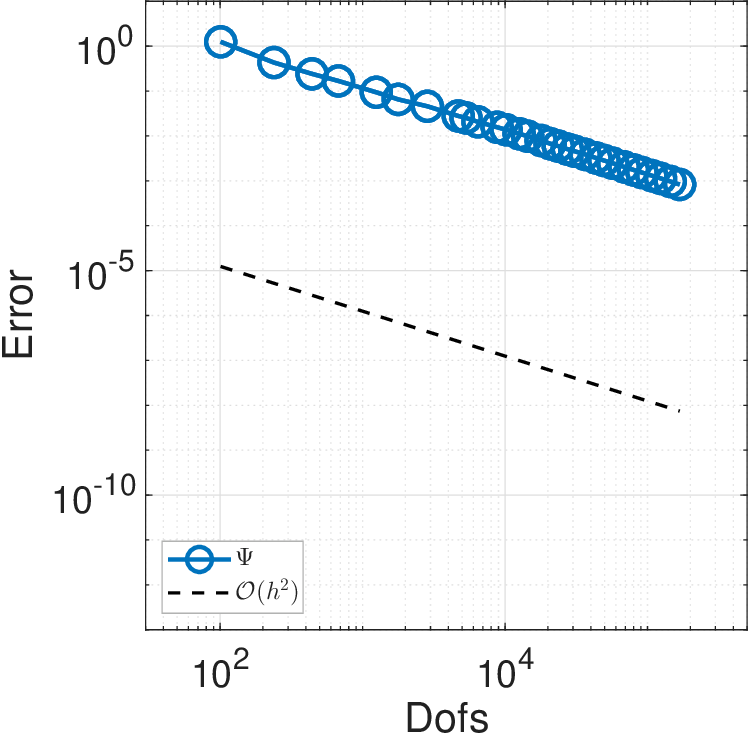}
		\caption{T-shape}
	\end{subfigure}
    \caption{Test 3: Global indicators for (a) L-shape (b) T-shape.}
	\label{con_test3}
\end{figure}

\subsection{Test 4: Flow past through a circular cylinder}
The geometrical setting of the domain is taken from \cite{MR4744101, Bayraktar2012}. 
The domain $\Omega$ is defined as the region $]0,2.5[ \times ]0,H[ \times ]0,H[$, with $H = 0.41\,\mathrm{m}$, excluding a cylinder of diameter $D = 0.1\,\mathrm{m}$. In this problem, we impose $\eqref{nsstokes0}_1$ on the inlet with
\[
\boldsymbol{u}_\star := \left(\frac{16U\, y z (H - y)(H - z)}{H^4},\, 0,\, 0\right)^T,
\]
where $U = 0.45\,\mathrm{m/s}$ and $\theta_\star = 1$. 
We impose $\eqref{nsstokes0}_2$ on the obstacle and $\eqref{nsstokes0}_3$ on the outlet, with
$\psi(\boldsymbol{u} \cdot \boldsymbol{n}) = \frac{1}{2}\big(\boldsymbol{u} \cdot \boldsymbol{n} + |\boldsymbol{u} \cdot \boldsymbol{n}|\big).$ On the lateral walls, we impose no-slip boundary conditions for the fluid velocity and do-nothing boundary conditions for the temperature. 
The parameters are chosen as
$\gamma = 0.01, \gamma_N = 100,  \beta = 0.01, \nu = 1,  \alpha = 1, \kappa = 1, \eta = 0.01.$ The right-hand side of the momentum equation is $\boldsymbol{f} = (0,-1,0)$, and $g = 0$. The numerical solutions for the velocity, pressure, and temperature fields are shown in Figures~~\ref{velocity_mag_FPCC}, \ref{velocity_FPCC}, \ref{pressure_FPCC}, and \ref{temp_FPCC}. 
The results illustrate the good performance of the method on this more complex example.
\begin{figure}[H]
    \centering
\begin{subfigure}[b]{0.7\textwidth}
        \centering
        \includegraphics[width=\textwidth]{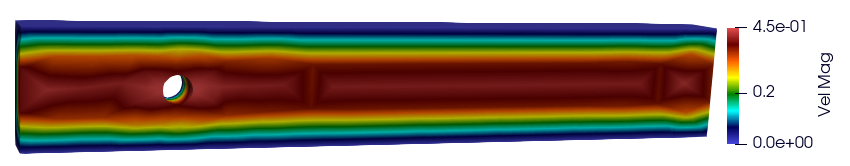}
        \caption{Velocity Magnitude}
        \label{velocity_mag_FPCC}
    \end{subfigure}
    \begin{subfigure}[b]{0.7\textwidth}
        \centering
        \includegraphics[width=\textwidth]{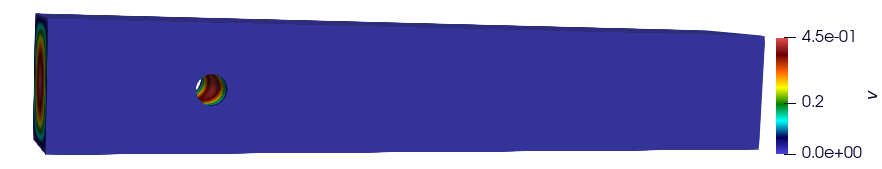}
        \caption{Velocity}
        \label{velocity_FPCC}
    \end{subfigure}

    \begin{subfigure}[b]{0.7\textwidth}
        \centering
        \includegraphics[width=\textwidth]{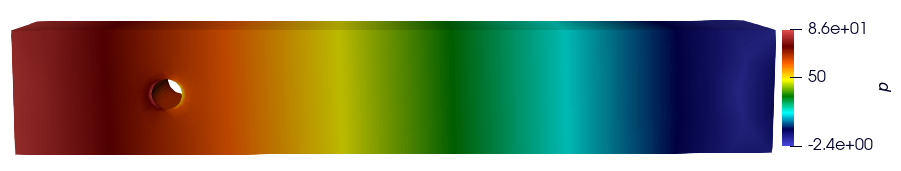}
        \caption{Pressure isovalues}
        \label{pressure_FPCC}
    \end{subfigure}

    \begin{subfigure}[b]{0.7\textwidth}
        \centering
        \includegraphics[width=\textwidth]{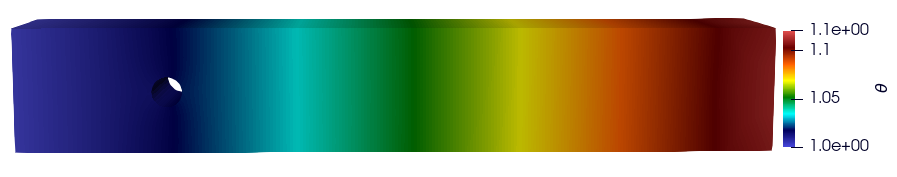}
        \caption{Temperature isovalues}
        \label{temp_FPCC}
    \end{subfigure}

    \caption{Test 4: Velocity Magnitude, Velocity, pressure isovalues, and temperature isovalues, respectively.}
    \label{test4}
\end{figure}

\noindent
{\bf Acknowledgements.} NAB is supported by Centro de Modelamiento Matemático (CMM), Proyecto Basal FB210005, and by the \textit{Chilean National Agency for Research and Development} (ANID Postdoctoral), Proyecto 3230326. GS is supported by ANID through the \textit{Fondecyt Iniciación} grant 11250322.
\par\smallskip
\noindent
{\bf Data availability.} Data generated during the research discussed in the paper will be made available upon reasonable request.
\par\smallskip
\noindent
{\bf Conflict of interest statement}.  The Authors declare that they have no conflict of interest.

\bibliography{references}  
\bibliographystyle{alpha}
\end{document}